\documentclass[12pt,reqno]{amsart}
\usepackage[T1]{fontenc}
\usepackage{amssymb}
\usepackage[mathscr]{eucal}
\usepackage{xypic}
\usepackage{amsmath}
\usepackage[all]{xy}
\usepackage{lineno}
\usepackage{epsfig}
\usepackage{amscd}
\usepackage{mathrsfs}

\usepackage{hyperref}
\usepackage[nameinlink]{cleveref}
\usepackage[thinc]{esdiff}
\usepackage{tikz-cd}
\usepackage[normalem]{ulem}

\usepackage{xcolor}

\theoremstyle{definition}
\newtheorem{theorem}{Theorem}[section]
\newtheorem{definition}{Definition}[section]
\newtheorem{lemma}[theorem]{Lemma}
\newtheorem{proposition}[theorem]{Proposition}
\newtheorem{corollary}[theorem]{Corollary}
\newtheorem{remark}[theorem]{Remark}
\newtheorem{example}[theorem]{Example}

\numberwithin{equation}{section}

\newcommand{\mc}{\mathcal}
\newcommand{\mb}{\mathbb}

\newcommand{\xra}{\xrightarrow}
\newcommand{\ra}{\rightarrow}
\newcommand{\rra}{\rightrightarrows}
\newcommand{\ghta}{(G,H,\tau, \alpha)}
\newcommand{\hrtag}{H \rtimes_{\alpha} G}
\newcommand{\Lrrw}{\Longrightarrow}
\newcommand{\widga}{\widetilde \gamma}
\DeclareMathOperator{\pr}{pr}

\DeclareMathOperator{\ad}{ad}
\DeclareMathOperator{\Ad}{Ad}
\DeclareMathOperator{\At}{At}
\DeclareMathOperator{\Mor}{Mor}
\DeclareMathOperator{\Obj}{Obj}

\def\og{\leavevmode\raise.3ex\hbox{$\scriptscriptstyle\langle\!\langle$~}}
\def\fg{\leavevmode\raise.3ex\hbox{~$\!\scriptscriptstyle\,\rangle\!\rangle$}}

\begin{document}
	
	\title
	{Atiyah sequence and Gauge transformations of a principal $2$-bundle over a Lie groupoid}
	
	\author[Saikat Chatterjee]{Saikat Chatterjee}
	\author[Adittya Chaudhuri]{Adittya Chaudhuri}
	\address{School of Mathematics,
		Indian Institute of Science Education and Research Thiruvananthapuram,
		Maruthamala P.O., Vithura, Kerala 695551, India}
	\email{saikat.chat01@gmail.com,adittyachaudhuri15@iisertvm.ac.in}

	\author[Praphulla Koushik]{Praphulla Koushik}
	
	\address{School of Mathematics,
		IISER Pune, Dr. Homi Bhabha Road, Pashan, Pune, 411008, India}
	\email{koushik16@iisertvm.ac.in, praphullakowshik@gmail.com}
	
	

	\begin{abstract} 
		In this paper, a notion of a principal $2$-bundle over a Lie groupoid has been introduced.
		For such principal $2$-bundles, we produced a short exact sequence of VB-groupoids, namely, the Atiyah sequence. Two notions of connection structures viz. strict connections and semi-strict connections on a principal $2$-bundle arising respectively, from a retraction of the Atiyah sequence and a retraction up to a natural isomorphism have been introduced.
		We constructed a class of principal $\mb{G}=[G_1\rra G_0]$-bundles and connections  from a given principal $G_0$-bundle $E_0\ra X_0$ over $[X_1\rra X_0]$ with connection.
		An existence criterion for the connections on a principal $2$-bundle over a proper, \'etale Lie groupoid is proposed. The action of the $2$-group of gauge transformations on the category of strict and semi-strict connections has been studied. Finally we noted 
		an extended symmetry of the category of semi-strict connections. 
	\end{abstract}

	\maketitle
	\addtocontents{toc}{\setcounter{tocdepth}{1}} 
	\tableofcontents
	\section{Introduction}
	The notion of (traditional) principal bundles and their connection structures over a manifold plays a central role at the interface of geometry and physics. 	In particular, a connection or, in the physics lexicon, a  gauge field describes the dynamics of a particle. In the last decade or so, higher gauge theories were developed as frameworks to describe the dynamics of string-like extended `higher dimensional objects'. A higher gauge theory typically involves suitable categorified versions of `spaces'. For example, a manifold is replaced by a Lie groupoid or a category internal to smooth spaces, and a Lie group is replaced by a Lie $2$-group, a smooth map is replaced by a smooth functor, and so on, and an appropriate notion of a connection consistent with this categorification.  In this context, intuitively, a principal $2$-bundle would be understood as a  smooth functor $\pi\colon \mb{E}\ra \mb{B}$ between categories internal to smooth spaces with the structure group a suitably categorified version of a Lie group (weak or strict or coherent Lie $2$-group). The precise descriptions of the `categorical spaces' depend on the framework. Among the earliest works in this direction, we refer to the following papers  \cite{MR2709030, MR2342821, MR2068522, baez2004higher, MR2068521},	
	and the other papers cited therein. The list is, of course, far from complete. 		
	
	More recently, Wockel in \cite{MR2805195} introduced a notion of semi-strict principal $2$-bundles over a discrete smooth $2$-space and describes their classification up to Morita equivalence by non-abelian C\v ech cohomology. A further generalization of Wockel's framework of principal $2$-bundles can be found in   \cite{MR2800361}. 	
	Among other approaches to principal $2$-bundles, one can mention non-abelian bundle gerbes of Aschieri, Cantini, and Jur\v co \cite{MR2117631},	and $G$-gerbes \cite{MR2493616}	of Laurent-Gengoux, Sti\'{e}non, and Xu. \cite{MR3089401} investigate the relation between non-abelian bundle gerbes of \cite{MR2117631} and principal $2$-bundles over a discrete smooth $2$-space. Ginot and Sti\'{e}non introduced the notion of a principal $2$-bundle over a Lie groupoid $\mb{X}$ with a strict structure Lie $2$-group   as 
	a  Hilsum \& Skandalis  generalized morphism of Lie $2$-groupoids $\mb{X}\ra \mb{G},$ where both $\mb{X}$ and $\mb{G}$ are treated as Lie $2$-groupoids  \cite{MR3480061}. In particular, the paper shows that a $G$-gerbe is the same as (up to a Morita equivalence) a principal automorphism $2$-group bundle.

	The connections structure on the principal $2$-bundles over a discrete smooth $2$-space and its various aspects are also well studied. Some of the works which will be relevant 
	in this context, 	are Breen-Messing \cite{MR2183393}, Baez-Schreiber \cite{MR2342821, schreiber2005loop} , Jurco-Samann-Wolf \cite{MR3351282}, Aschieri-Cantini-Jurco  \cite{MR2117631}, Gengoux-Stienon-Xu  \cite{MR2493616}. In \cite{MR3894086} Waldorf introduced the notion of a  connection on Wockel's principal $2$-bundle. 
	
	Some of the papers, which discuss higher gauge theory over path-space groupoids are  \cite{MR3126940, MR3504595, MR3213404, MR2764890}. For higher gauge theories, which go beyond $2$-spaces, we refer to \cite{MR3548195, bakovic2009simplicial, bakovic2008bigroupoid} for a framework involving  Kan simplicial manifolds, and more generally, for  $\infty$-topos theory framework, we refer to  \cite{fiorenza2011cech,   MR3423073, MR3385700, schreiber2013differential}.

	In this paper, we introduce a notion of a principal $2$-bundle $\mb{E}:=[E_1\rra E_0]$ over a Lie groupoid $\mb{X}:=[X_1\rra X_0],$	 having a structure group a strict Lie $2$-group $\mb{G}=[G_1\rra G_0],$ as a smooth functor $\mb{E}\ra \mb{X}$ with a functorial  action $\mb{G}\times \mb{E}\ra \mb{E}$ of $\mb{G}$ on $\mb{E}$ such that we have  $G_1$ and  $G_0$-bundles 
	$E_1\ra X_1$ and $E_0\ra X_0$ respectively over $X_1, X_0.$ The definition of $2$-principal bundle we consider here similar to the one in  \cite{MR3126940} with suitable smoothness conditions, and the definition is slightly stronger than that of Ginot and Sti\'{e}non in  \cite{MR3480061}. 
	
The definition of a $G$-bundle over a Lie groupoid $[X_1\rra X_0]$,  introduced in \cite{MR2270285}, can be viewed as a special case of ours. The Atiyah sequence associated to a principal $G$-bundle of \cite{MR2270285} have been studied in \cite{biswas2020chern,MR3150770}. The main purpose of this paper is to explore the connections and gauge transformations on a strict Lie $2$-group bundle over a Lie groupoid in terms of the Atiyah sequence. 
	Recall that in the traditional set-up, with a principal $G$-bundle $P\ra M$ one can associate the Atiyah sequence,
	a short exact sequence of vector bundles over $M$  \cite{MR86359},
	\begin{equation}\nonumber
		0\ra {\rm Ad}(P)\xra{j^{/G}} {\rm At}(P)\xra{\pi_{*}^{/G}} TM\ra 0.
	\end{equation}
	Here $L(G)$ denote the Lie algebra of $G,$ ${\rm Ad}(P):=(P\times  L(G))/G\ra M$ and  ${\rm At}(P):= TP/G\ra M$ are respectively the adjoint bundle and the Atiyah vector bundle of the principal $G$-bundle $\pi\colon P\ra M.$ Then a retraction of the Atiyah sequence is a connection on $P\ra M$. We show that associated to a principal $2$-bundle $\mb{E}:=[E_1\rra E_0]$ over a Lie groupoid $\mb{X}:=[X_1\rra X_0],$ we have a short exact sequence of VB-groupoids (for the discussion on VB-groupoids, we refer to \cite{MR3696590, MR3744376}).  Then a `strict connection' on the principal $2$-bundle  is defined as a retraction of the short exact sequence of the VB-groupoids. 
	However,  enrichment is more interesting when we impose an obstruction, given by   a certain natural isomorphism, on the retraction. The corresponding structure thus obtained will be called a `semi-strict connection'. Just as a side note on our terminology,  we remark here that we reserve the word `weak connection' for a more general set-up of a simplicial principal bundle obtained from a principal $2$-bundle as above, which we will  pursue in our subsequent work. The category of gauge transformations on our principal $2$-bundles naturally defines a  $2$-group, which acts on the category of strict and semi-strict connections. We show that the category of semi-strict connections enjoys an `extended gauge symmetry.'
	In particular, considering the category of semi-strict connections on a $2$-bundle over a discrete Lie groupoid $[X\rra X],$ this extended gauge symmetry offers a natural interpretation of the gauge transformation of the connection $1$-forms in higher BF theories (for example, see \cite{MR2661492, martins2011fundamental, MR2764890}).

	Along the course of this paper, we illustrate our constructions and results with several examples and highlight the relations with other relevant works. We believe that the framework provides a general treatment of connections on Lie $2$-group bundles over a Lie groupoid.  We show that given a Lie group $G$-bundle over a Lie groupoid $[X_1\rra X_0]$ with connection, in the sense of \cite{MR2270285},  one can construct a Lie $2$-group  $\mb{G}$ bundle with a connection, with $\Obj(\mb{G})=G.$ Moreover, the construction gives us a criterion for the existence of strict and semi-strict connections on a principal $2$-bundle over a proper, \'etale Lie groupoid (or an orbifold). As an outcome, we prove the existence of connections on a Lie $2$-group bundle over a discrete Lie groupoid (compare,  \cite[Theorem $5.2.14$]{MR3894086}).   We also emphasize that to the best of our knowledge, the approach of the Atiyah sequence has not been explored in the higher gauge theory literature, other than a casual reference in the lecture notes \cite{stevenson}. Particularly one interesting outcome of this approach is the notion of semi-strict connections on a Lie $2$-group bundle.  
Having said that, we have not studied Chern-Weil theory or Morita invariance of our construction for this paper. In a sequel of this paper, we intend to study these topics in the set-up of simplicial principal bundles. 

\subsection*{The outline and organization of the paper} \Cref{$2$-groupsCrossedmodules} recalls the definition of a strict Lie $2$-group,  its correspondence with a Lie crossed module, and some properties and identities of the associated Lie $2$-algebra.

In \Cref{$2$-groupbundleoverLiegroupoid} we introduced the definition of a principal 
$2$-bundle over a Lie groupoid and constructed a principal  $2$-bundle, namely a decorated bundle over a Lie groupoid from a Lie group bundle over the same. We introduce a notion of a categorical connection and characterize the decorated principal bundles. We also explored the relationship with a `Lie groupoid $G$-extension'.

In \Cref{Atiyahsequencefor$2$-groupbundle} we show that a principal $2$-bundle over a Lie groupoid produces a short exact sequence of VB-groupoids, namely the Atiyah sequence. 

In \Cref{Connection on principal $2$-bundles over Lie groupoids}	we introduce the strict and semi-strict connections on a principal $2$-bundle, respectively, as a retraction of the associated Atiyah sequence and a retraction up to a certain natural isomorphism. We express the strict and semi-strict connections in terms of  Lie $2$-algebra valued differential forms. Further, we give an explicit construction of the connection structure on a decorated bundle and obtain a criterion for the existence of a connection on a principal $2$-bundle over an orbifold. 

In \Cref{Gauge 2 group of a principal 2 bundle}, we introduce the notion of the  $2$-group of the gauge transformations on a principal $2$-bundle. The gauge $2$-group naturally acts on the categories of strict and semi-strict connections. We observe an extended gauge symmetry of the category of semi-strict connections, and we relate to the gauge transformations of the connection $1$-forms in higher BF theories. 

\textbf{Notations and conventions}: Here, we fix some conventions and notations which will  be followed throughout this paper. 

All manifolds assumed to  be the smooth, second countable, Hausdorff. Let $\rm Man$ be category of such manifolds.  For a smooth map $f\colon M\to N,$ the differential at $m\in M$ will be denoted as
$f_{*,m}\colon T_m M\ra T_{f(m)}N.$ A smooth right (resp. left) action of a Lie group $G$ on a smooth manifold $P$ will be denoted as $(p, g)\mapsto p g$ (resp. $(g, p)\mapsto g p$), for $g\in G$ and $p\in P$. The corresponding differentials $T_pP\ra T_{p g}P$ or $T_p P\ra T_{g p }P$ will be denoted respectively 	as $v\mapsto v\cdot g$ or $v\mapsto g\cdot v,$ for $v\in T_p P$.  For a fixed $p\in P$, let $\delta_p\colon T_eG\ra T_pP$ denote the differential of the map $G\ra P$,  $g\mapsto pg$, at the identity element. We will denote the fundamental vector field corresponding to 
an element $B\in L(G)$ evaluated at $p$ by $\delta_p(B)$.

We assume our Lie groups to be matrix groups for computational simplicity.  For a Lie group $G$, we denote its Lie algebra by 
$ L(G)$.

In any category, $s, t$ will denote the source and target maps, and the composition of a pair of morphisms $f_2, f_1$  will be denoted as $f_2\circ f_1,$ where $t(f_1)=s(f_2).$	

By a \textit{Lie groupoid} we mean a groupoid object in $\rm Man,$ whose source, target maps are surjective submersions.   We use the blackboard bold notation to denote Lie groupoids; that is, $\mb{E}$ for $[E_1\rra E_0]$ and so on. 
Given a Lie groupoid $\mb{X}=[X_1\rra X_0],$ we have the associated \textit{tangent Lie groupoid}, denoted $T\mb{X}=[TX_1\rra TX_0],$ structure maps of $T\mb{X}$ are given by the differentials of the respective structure maps of $\mb{X}.$ 	 In particular if $(\gamma_2, X_2), (\gamma_1, X_1)$ are composable morphisms in $T\mb{X},$ then we denote $(\gamma_2, X_2)\circ (\gamma_1, X_1)=\bigl(m(\gamma_2, \gamma_1), m_{* (\gamma_2, \gamma_1)}(X_2, X_1)\bigr)$ as $(\gamma_2\circ \gamma_1, X_2\circ X_1),$ where $m$ is the composition map of the Lie groupoid $\mb{X}.$

\section{Lie $2$-groups, Lie crossed modules}\label{$2$-groupsCrossedmodules}
In this section, we briefly overview the notions of Lie $2$-groups and Lie crossed modules. We then describe a correspondence between the two. This section does not contain any new material. For further reading on this topic, we refer to several papers by Baez, Schreiber, Lauda, Crans and others, \cite{{MR2342821}, {MR2068521}, {MR2068522}, {baez2004higher}, {MR3126940}, {MR3504595}}. 

\begin{definition}[strict $2$-category {\cite{MR1291599}}]
	A \textit{ strict $2$-category}  $\mc{D}$, consists of the following data,
	\begin{enumerate}
		\item a collection of objects,
		\item a small category $\mc{D}(x,y)$ for each pair of objects $(x,y)$ in $\mc{D}$,
		\item a functor $\{1\}\ra \mc{D}(x,x)$ for each object $x$ in $\mc{D} $,
		\item a functor $\circ_{xyz}:\mc{D}(x,y)\times\mc{D}(y,z)\ra \mc{D}(x,z)$ for each triple $(x,y,z)$ of objects in $\mc{D}$,
	\end{enumerate}
	satisfying the usual strict  associativity and unital compatibility conditions.
\end{definition}

\begin{definition}
	A \textit{strict Lie $2$-group} $\mb{G}=[G_1 \rightrightarrows G_0]$ is a group object in the category of Lie groupoids.
\end{definition}

Though there are more general versions of Lie $2$-groups, as we will be working with only strict Lie $2$-groups, by a Lie $2$-group we always mean a strict Lie $2$-group in this paper. For a more general notion of Lie $2$-groups, we refer the reader to \cite{MR2068521}.

\begin{remark}One can show that a  Lie $2$-group is a Lie groupoid $\mb{G}$ equipped with a morphism of Lie groupoids $\otimes: \mb{G} \times \mb{G} \rightarrow \mb{G}$ such that $\otimes$ induce Lie group structures on both 
	$\Obj(\mb{G})$ and $\Mor(\mb{G})$, and all the structure maps are Lie group homomorphisms. 
	Let $k_2, k'_2\in \Mor(\mb{G})$ and $k_1, k'_1\in \Mor(\mb{G})$	be pairs of composable morphisms. Let ${}^{-1}$ and $\circ$ respectively denote the group inverse functor and the composition. Then 
	\begin{equation}\label{E:Identitiesechangeinverse}
		\begin{split}
			&(k_2\otimes k_1)\circ (k'_2\otimes k'_1)=(k_2\circ k'_2)\otimes (k_1\circ k'_1),\\
			&(k_2\circ k_1)^{-1}={k_2}^{-1}\circ {k_1}^{-1}.		
		\end{split}
	\end{equation}
\end{remark}

Another description of Lie $2$-group is given by a Lie crossed module which we define below:
\begin{definition}
	A \textit{Lie crossed module} is a $4$-tuple 
	\[\ghta\]where, 
	$G,H$ are Lie groups, $\alpha:G\times H\ra H$ is a smooth map giving an action of $G$ on $H$, and  $\tau:H\ra G$ is a morphism of Lie groups such that the following conditions are satisfied:
	\begin{equation}\label{E:Peiffer}
		\begin{split}
			& \tau(\alpha(g,h))=g\tau(h)g^{-1} \, \textit {\rm for all} \, (g,h) \in G\times H,\\
			& \alpha(\tau(h),h')=hh'h^{-1} \, \textit {\rm for all}\, h,h'\in H.
		\end{split}
	\end{equation}
\end{definition}

The conditions in \Cref{E:Peiffer}  are called \textit{Peiffer identities}.

\subsection{Lie $2$-group associated to a Lie crossed module}\label{Lie 2 group from Lie crossed module}
Given a Lie crossed module $(G, H, \tau, \alpha)$, the associated  Lie $2$-group is given by the Lie groupoid $ \mb{G}=[H \rtimes_{\alpha} G \rightrightarrows G],$ where $\hrtag$ denotes the semidirect product of groups $H$ and $G$ with respect to the action $\alpha$ of $G$ on $H.$ The various maps and operations on this Lie $2$-group are as listed below.
\begin{itemize}
	\item the source map is given by $s(h,g)=g$,
	\item the target map is given by $t(h,g)=\tau(h)g$
	\item the composition map is given by $m((h_2, g_2),(h_1, g_1))=(h_2 h_1,g_1)$,
	\item the identity map is given by $1_g=(e_H,  g)$
	\item the inverse map is given by $i(h,g)=(h^{-1},\tau(h)g)$
	\item the group properties of  $H\rtimes_{\alpha} G$ are those of the standard semidirect product of the groups, that is 
	the bi-functor $\otimes: \mb{G} \times\mb{G} \rightarrow \mb{G}$ is defined as 
	\begin{equation}\label{E:grouppro}
		\begin{split}
			&\otimes_0: (g_1, g_2) \mapsto g_1g_2,
			\\
			&\otimes_{1}:  ((h_2, g_2), (h_1, g_1)) \mapsto (h_2 \alpha(g_2,h_1), g_2 g_1),
		\end{split}
	\end{equation}
\end{itemize}
whereas  the group inverse and the identity elements are respectively   given by $(h,g)^{-1}=(\alpha(g^{-1}, h^{-1}), g^{-1})$ and  $(e_H,e_G)$.
It is a routine verification that  $\mb{G}$ is a Lie $2$-group. We call $\mb{G}$ to be the \textit{Lie $2$-group associated  to the Lie crossed module} $(G,H,\tau, \alpha)$.
\subsection{Lie crossed module associated to a Lie $2$-group}\label{Lie crossed module from a Lie2group}
Let $\mb{G}=[G_1\rra G_0]$ be a Lie $2$-group. Consider the Lie group 
$\ker(s)=\{\gamma\in G_1: s(\gamma)=1_{G_0}\}$, and the 
morphism $\alpha:G_0\times \ker(s)\ra \ker(s)$ defined as $(a,\gamma)\mapsto 1_a\cdot \gamma\cdot 1_{a^{-1}}$. Then, the $4$-tuple
\[(G_0, \ker(s), t|_{\ker(s)}:\ker(s)\ra G_0,\alpha:G_0\times \ker(s)\ra \ker(s)),\]
defines a Lie crossed module. 
We call the above Lie crossed module to be the \textit{Lie crossed module associated to the Lie $2$-group} $\mb{G}$.

\begin{example}\label{Ex:Discretecrossedmodule}
	The Lie crossed module associated to the discrete  Lie $2$-group $[G\rra G]$ is $(G, \{e\})$ with trivial $\tau$ and $\alpha.$
\end{example}

\begin{example}\label{Ex:ordinary}
	At the other extreme of the last example, one can consider the Lie $2$-group associated to a crossed module  $(\{e\}, H)$ with trivial $\tau$ and $\alpha={\rm Id}\colon H\to H.$ It is obvious that the Lie $2$-group	is the single object Lie $2$-group $[H\rra {e}]$
	
\end{example}
\begin{example}\label{Ex:singleobjlie2iso}
	The Lie $2$-group associated to the Lie crossed module $(G, G, \tau, \alpha)$ with $\tau\colon G\ra G$ being the identity map, and $\alpha\colon G\times G\ra G$ the action of $G$ on itself by conjugation is the pair groupoid $[G\times G\rra G]$. 
\end{example}

\begin{example}\label{Ex:isocrossedmodule}
	To any simply connected Lie group $G$ one associates a Lie crossed module  $\bigl({\rm Aut}(G), G, \tau, \alpha\bigr),$ where $\tau\colon G\ra {\rm Aut}(G)$ sends an element $g\in G$
	to its inner-automorphism, and $\alpha\colon {\rm Aut}(G)\ra {\rm Aut}(G)$ is the identity. The corresponding Lie $2$-group is called the \textit{automorphism $2$-group} of $G.$	\end{example}

We will not spend much time on examples of Lie $2$-groups and Lie crossed modules, which abound in higher gauge theory literature. An interested reader can go through the articles previously mentioned.

Given a Lie $2$-group $\mb{G}=[G_1\rra G_0]$ we can associate the Lie groupoid $[L(G_1) \rra L(G_0)]$ whose structure maps are given by taking differentials of the structure maps of $\mb{G}$ at the identity.
\begin{remark}\label{Remark:Interchange law of Lie 2-algebra}
	The vector space structures on $L(G_1)$ and $L(G_0)$ are compatible with the composition of $[L(G_1)\rra L(G_0)]$. 
\end{remark}
In particular, if the Lie  $2$-group $\mb{G}=[G_1 \rra G_0]$ is obtained from a Lie crossed module  $(G,H,\tau:H\ra G, \alpha:G\times H\ra H)$ 
then the  groupoid  $[L(G_1) \rra L(G_0)]$ is described as   $[L(H)\oplus L(G)\rra L(G)]$ having  structure maps
\begin{itemize}
	\item the source map is given by $s(A,B)=B$,
	\item the target map is given by $t(A,B)=\tau(A)+B$,
	\item the composition map is given by $m((A_2,B_2),(A_1,B_1))=(A_2+A_1,B_1)$, where 
	$s(A_2,B_2)=t(A_1,B_1)$,
	\item the identity map is given by $1_B=(0, B)$,
	\item the inverse map is given by $i(A,B)=(-A,\tau(A)+B)$.
\end{itemize} 
\begin{remark}
	It should be noted that $L(G_1)$ splits into the direct sum $L(H)\oplus L(G)$ only as a vector space. It is not a direct sum of Lie algebras.
\end{remark}
The differential on the action $\alpha$ produces an action $\alpha_{*,( e, e)}\colon L(G)\times L(H)\ra L(H)$ of $L(G)$ on $L(H)$ as a Lie algebra derivation, and the commutators  on $L(G)$ and $L(H\rtimes G)$ are respectively  given as
\begin{equation}\label{E:Commutator}
	\begin{split}
		(B_1, B_2)&\mapsto [B_1, B_2],\\
		\bigl((A_1, B_1), (A_2, B_2)  \bigr)&\mapsto [(A_1, B_1), (A_2, B_2)]\\
		& :=\bigl([A_1,A_2]+\alpha_{*,( e, e)} (A_1, B_2)-\alpha_{*,( e, e)} (A_2, B_1), [B_1, B_2]\bigr),
	\end{split}
\end{equation}
for all $A_1, A_2\in L(H),  B_1, B_2\in L(G)$. The Lie groupoid $[L(G_1)\rra L(G_0)]$ is the usual strict Lie $2$-algebra of the Lie $2$-group $[G_1\rra G_0]$.
We will not spell out the general definition of a Lie $2$-algebra. For a detailed discussion on Lie $2$-algebras and the associated Lie $2$-algebra of a Lie $2$-group we refer to \cite{MR2068522}.  

Let $(G,H,\tau:H\ra G, \alpha:G\times H\ra H)$ be  the Lie crossed module associated to the  Lie $2$-group $\mb{G}=[G_1 \rra G_0].$ Consider the smooth map 
$\alpha:G\times H\ra H.$ For a fixed $g\in G,$ we have the map $\alpha(g)\colon H\to H, h\mapsto \alpha(g, h)$ and the differential of this map at the identity element of $H$
gives the Lie algebra homomorphism
\begin{equation}\label{E:alphadifffixedg}
	\alpha(g)_{*, e_H}\colon L(H)\to L(H).
\end{equation}
On the other hand for a fixed $h\in H,$ we have the map ${\bar \alpha}(h)\colon G\to H, g\mapsto \alpha(g, h)$ and by taking the differential of this map at the identity element of $G$
we obtain the linear map,
\begin{equation}\label{E:alphadifffixedh}
	{\bar \alpha}(h)_{*, e_G}\colon L(G)\to T_hH.
\end{equation}

\begin{remark} \label{Remark:notations}
	To simplify our notations and avoid notational cluttering, henceforth, we adopt the following conventions.
	\begin{itemize}
		\item $\tau_{*, e_H}$ will also be denoted as $\tau.$ Similarly  $\alpha(g)_{*, h}$, $\alpha(g)_{*, e_H}$ and ${\bar \alpha}(h)_{*, e_G}$ will be denoted respectively as $\alpha(g)_h, \alpha(g)$ and ${\bar \alpha}(h).$
		\item To avoid cluttering of  parentheses we often write $\alpha(g)$ etc as $\alpha_g$ etc.
		\item Unless necessary, we do not notationally distinguish between identity elements of various groups. The common notation for the same will be $e.$
\end{itemize} 
\end{remark}

For future reference, we note the following identities. The identities are technical and would not be required anywhere else other than the proofs of \Cref{cor:stricconnto semi} and \Cref{lemma:Funcdecconn}. So a reader may postpone its reading till then.

\begin{lemma}\label{Lemma:Identityfor later} 		
	Let $\ghta$ be a Lie crossed module. 
	\begin{enumerate}
		\item Then for any $h_2, h_1\in H$ and $B\in L(G),$ we have 
		$$h_2\cdot \biggl({\bar \alpha}(h_2^{-1})({\rm ad}_{\tau(h_1)}B )\biggr)+h_1\cdot \biggl({\bar \alpha}(h_1^{-1}) (B) \biggr)=h_2h_1 \cdot \biggl({\bar \alpha}(h_1^{-1}h_2^{-1}) (B)\biggr).$$
		\item Then for any $g\in G, h\in H, A\in L(H)$
		$$\alpha_{g^{-1}}(h^{-1})\,\cdot \bigl({\bar \alpha}_{\alpha_{g^{-1}}(h)}   {({\ad}_{g^{-1}}(\tau(A)\bigr)}
		+\alpha_{g^{-1}}\bigl(A\bigr) =\alpha_{g^{-1}}(\ad_{h^{-1}}(A)).$$
	\end{enumerate}
	\begin{proof}
		\begin{enumerate}
			\item Note that the required identity follows directly as a infinitesimal version of the equation  
			$$\bigg[h_2 \biggl({\bar \alpha}(h_2^{-1})({\Ad}_{\tau(h_1)}(g) )\biggr)\bigg] \bigg[h_1 \biggl({\bar \alpha}(h_1^{-1}) (g) \biggr)\bigg]=h_2h_1 \biggl({\bar \alpha}(h_1^{-1}h_2^{-1}) (g)\biggr)$$ for any $h_2, h_1\in H$ and $g\in G$. 
			In particular, using \Cref{E:Peiffer} the expression \[h_2 \bigg[\biggl({\bar \alpha}(h_2^{-1})({\Ad}_{\tau(h_1)}(g) )\biggr)\bigg]\] reduces to  $h_2 \bigg[h_1 \alpha(g)(h_1^{-1} h_2^{-1} h_1) h_1^{-1}\bigg]$.Then a straight forward calculation will prove the identity. 			
			
			\item A straightforward verification gives  $$\big[\alpha_{g^{-1}}(h^{-1})\,\cdot \bigl({\bar \alpha}_{\alpha_{g^{-1}}(h)}   {({\Ad}_{g^{-1}}(\tau(h')\bigr)}\bigr]
			\big[\alpha_{g^{-1}}\bigl(h'\bigr)\big] =\alpha_{g^{-1}}(\Ad_{h^{-1}}(h')),$$ for any $h, h'\in H, g\in G.$ Then the identity follows as the infinitesimal version of the above equation.
		\end{enumerate}
	\end{proof}
\end{lemma}
\subsection{Adjoint actions of a Lie $2$-group }\label{SS:adjointLie 2 group}
Let $\mb{G}=[G_1\rra G_0]$ be a Lie $2$-group and  $\ghta$  the associated  Lie crossed module. It would be convenient to write down the adjoint actions for the object and morphisms groups in terms of a crossed module. As we have seen $G_0=G$ and $G_1=\hrtag.$ 

In order to simplify our calculation of the adjoint action for $\hrtag,$ we write ${{\rm Ad}}_{(h, g)}(h', g')$ as ${{\rm Ad}}_{(h, g)}(h', g')=[{{\rm Ad}}_{(h, g)}(h', e)][{{\rm Ad}}_{(h, g)}(e, g')]$ and then compute the bracketed terms on the right hand side separately. The group product in  \Cref{E:grouppro} gives
\begin{equation}\label{E:Adjongroups}
	\begin{split}
		&{{\rm Ad}}_{(h,\, g)}(h', e)=\biggl({\rm Ad}_{h}\bigl(\alpha_g (h ')\bigr),\, e\biggr), \\
		&{{\rm Ad}}_{(h,\, g)}(e, g')=\biggl(h\, \bigl({\alpha}_{({\rm Ad}_g(g '))}(h^{-1})\bigr) , \, {\rm Ad}_g (g')\biggr)=\biggl(h\, \bigl({\bar \alpha}_{h^{-1}}{({\rm Ad}_g(g '))}\bigr) , \, {\rm Ad}_g (g')\biggr),\\
		&{{\rm Ad}}_{(h, g)}(h', g')=[{{\rm Ad}}_{(h, g)}(h', e)][{{\rm Ad}}_{(h, g)}(e, g')].
	\end{split}
\end{equation}
Now given an element $(A, B)\in L(H)\oplus L(G)=L(G_1),$ we write $(A, B)=(A, 0)+(0, B).$ Then the  adjoint actions of $H\rtimes G$ on $A\in L(H)$ and $B\in L(G)$ are respectively  computed by the   formulae	$\ad_{(h, g)}(A):=\diff{{}}{t}\Ad_{(h, g)}{{\rm e}}^{t A}\big|_{t=0}$ and  $\ad_{(h, g)}(B):=\diff{{}}{t}\Ad_{(h, g)}{{\rm e}}^{t B}\big|_{t=0},$	
\begin{equation}\label{E:Adjonalgebras}
	\begin{split}
		&{\rm{ad}}_{(h,\, g)}(A, 0)=\biggl({\rm ad}_{h}\bigl(\alpha_g (A)\bigr),\, 0\biggr), \\
		&{\rm{ad}}_{(h,\, g)}(0, B)=\biggl(h\,\cdot \bigl({\bar \alpha}_{h^{-1}}{({\rm ad}_g(B))}\bigr) , \, {\rm ad}_g (B)\biggr),\\
		&{\rm{ad}}_{(h, g)}(A, B)=[{\rm{ad}}_{(h, g)}(A, 0)]+[{\rm{ad}}_{(h, g)}(0, B)].
	\end{split}
\end{equation}	 

The notations $\alpha_g, \bar \alpha_{h^{-1}}$ in \Cref{E:Adjonalgebras} should be interpreted as per  \Cref{Remark:notations} (third and fourth bullet points).


	\section{Principal $2$-bundles over Lie groupoids}\label{$2$-groupbundleoverLiegroupoid}
In this section, we introduce the notion of a principal $2$-bundle over a Lie groupoid. 	
The definition we introduce is a natural generalization of the classical definition.

\begin{definition}[Action of a Lie $2$-group on a Lie groupoid]\label{Definition:Lie 2 groupaction}
	An \textit{action of a  Lie $2$-group $\mb{G}$ on a Lie groupoid $\mb{X}$ } is defined as a morphism of Lie groupoids $\rho: \mb{X} \times \mb{G} \rightarrow \mb{X}$  such that 
	\begin{itemize}
		\item $\rho_0: X_0 \times G_0 \rightarrow X_0$,
		\item $\rho_1:X_1 \times G_1 \rightarrow X_1$,
	\end{itemize}
	are Lie group actions on manifolds $X_0$ and $X_1$ respectively.
\end{definition}
Following proposition is evident. 
\begin{proposition}\label{Prop:Adjtrans}
	Let $\mb{G}$ be a Lie $2$-group. 
	\begin{enumerate}
		\item There is an action of Lie $2$-group $\mb{G}$ on  
		Lie groupoid  $L(\mb{G})=[L(G_1) \rightrightarrows L(G_0)]$	by adjoint action. 
		\item Suppose $\mb{G}$ acts on a Lie groupoid $\mb{X}$. Then there is an action of $\mb{G}$ on  
		Lie groupoids $T\mb{X}=[TX_1 \rightrightarrows TX_0],$ given by the differential of the given action.  
	\end{enumerate}
\end{proposition}	
A weaker definition compared to the one given in \Cref{Definition:Lie 2 groupaction} also exists in literature; for example, in  \cite{MR2805195}, the identity and compatibility axioms of a group action at the object level hold up to an isomorphism. One can also generalize this action by replacing the category $\mb{X}\times \mb{G}$ 
by a `twisted product  category' $\mb{X}\rtimes_{\eta}\mb{G}$ of $\mb{X}$ and $\mb{G}$ with respect to certain map $\eta$, introduced in \cite{MR3213404}. 
While we will stick to the definition of action given in \Cref{Definition:Lie 2 groupaction} for this paper and not pursue either of the other two generalizations very seriously, we  recall the definition of a twisted product to relate  with the notion of `Lie groupoid $G$-extensions'.

Let $\mb{G}$ and $\mb{X}$ be as before. Consider a smooth map $\eta\colon \Mor(\mb{X})\times \Mor(\mb{G})\ra \Mor(\mb{X})$ satisfying, 
\begin{equation}\label{E:Contwist}
	\begin{split}
		\eta(\gamma, k)&\in {\rm Hom}_{\mb{X}}(x, y), 	\forall \gamma\in {\rm Hom}_f{\mb{X}}(x, y), \\
		\eta(\gamma_2\circ \gamma_1, k)&=\eta(\gamma_2, k)\circ \eta(\gamma_1, k),\\
		\eta(1_x, k)&=1_x\\
		\eta(\gamma, k_2\circ k_1)&=\eta(\eta(\gamma, k_2), k_1), \\ 
		\eta(\gamma, 1_g)&=\gamma.
	\end{split}
\end{equation}
Then we have a smooth category  $\mb{X}\rtimes_{\eta}\mb{G}$ with the following description \cite[Proposition $5.1$]{MR3213404}:
\begin{equation}
	\begin{split}
		&\Obj(\mb{X}\rtimes_{\eta}\mb{G})= \Obj(\mb{X})\times \Obj(\mb{G}),\\
		&\Mor(\mb{X}\rtimes_{\eta}\mb{G})= \Mor(\mb{X})\times \Mor(\mb{G}),\\
		&s(\gamma, k)=\bigl(s(\gamma), s(k)\bigr), t(\gamma, k)=\bigl(t(\gamma), t(k)\bigr)\\
		&(\gamma_2, k_2)\circ_{\eta} (\gamma_1, k_1)=\bigl(\gamma_2\circ \eta(\gamma_1, k_2), k_2\circ k_1\bigr).
	\end{split} 
\end{equation}
The category $\mb{X}\rtimes_{\eta}\mb{G}$ will be called the  \textit{$\eta$-twisted category}.

\begin{definition}[Twisted action of a Lie $2$-group on a Lie groupoid]\label{Definition:Lie 2 twistgroupaction}
	Let $\mb{X}$ be a Lie groupoid, $\mb{G}$ a Lie $2$-group and $\eta\colon \Mor(\mb{X})\times \Mor(\mb{G})\ra \Mor(\mb{X})$ a smooth map satisfying conditions in \Cref{E:Contwist}. 	An \textit{$\eta$-twisted action of  $\mb{G}$ on the Lie groupoid $\mb{X}$ } is defined as a smooth functor  $\rho\colon \mb{X} \rtimes_{\eta} \mb{G} \rightarrow \mb{X}$  such that 
	\begin{itemize}
		\item $\rho_0: X_0 \times G_0 \rightarrow X_0$,
		\item $\rho_1:X_1 \times G_1 \rightarrow X_1$,
	\end{itemize}
	are Lie group actions on manifolds $X_0$ and $X_1$ respectively.
\end{definition}
\begin{remark}\label{Remark:functorialitytwisted}
	\begin{enumerate}
		\item Note that if we take $\eta={\rm pr}_1$ in \Cref{Definition:Lie 2 twistgroupaction}  then we recover the usual direct product of categories .	
		\item Functoriality of $\rho: \mb{X} \rtimes_{\eta} \mb{G} \rightarrow \mb{X}$	implies $\rho\bigl(\gamma_2\circ \eta(\gamma_1, k_2), k_2\circ k_1\bigr)=\rho (\gamma_2, k_2)\circ \rho(\gamma_1, k_1)$ and following our convention for group actions  we write,
		$$\bigl(\gamma_2\circ \eta(\gamma_1, k_2)\bigr) (k_2\circ k_1)=(\gamma_2 k_2)\circ (\gamma_1 k_1).$$	
		
	\end{enumerate}
	
\end{remark}

\begin{definition} \label{Equivariant morphism of Lie gorupoid}
	Let $\mb{G}$ be a Lie $2$-group. Suppose $\mb{G}$ acts on a pair of Lie groupoids $\mb{X}$ and  $\mb{Y}.$ Then a morphism of Lie groupoids $F:=(F_1, F_0): \mb{X} \rightarrow \mb{Y}$ is said  to be $\mb{G}$-equivariant if $F_0$ is $G_0$-equivariant and $F_1$ is $G_1$-equivariant. 
	A smooth natural transformation  between two such functors 
	$\eta\colon F\Lrrw F'$  is called a \textit{$\mb{G}$-equivariant natural transformation}  if $\eta(x g)=\eta(x)1_g.$
\end{definition}
\begin{definition}[principal $\mb{G}$-bundle over a Lie groupoid] \label{Definition:principal $2$-bundle over Liegroupoid}
	Let $\mb{G}$ be a  Lie $2$-group.  A \textit{principal $\mb{G}$-bundle over a Lie groupoid} $\mathbb{X}$ is given by a morphism of Lie groupoids $\pi: \mb{E} \rightarrow \mb{X}$ along with an  action $\rho: \mb{E} \times \mb{G} \rightarrow \mb{E}$ of Lie $2$-group $\mb{G}$ on $\mb{E}$ such that, 
	\begin{itemize}
		\item $\pi_0\colon E_0 \rightarrow X_0$ is a principal $G_0$-bundle over $X_0$,
		\item $\pi_1\colon E_1 \rightarrow X_1$ is a principal $G_1$-bundle over $X_1$.
	\end{itemize}
	We will call $\mb{G}$ to be the  \textit{structure $2$-group of $\pi\colon\mb{E} \rightarrow \mb{X}$}. We will denote the above principal $\mb{G}$-bundle over the Lie groupoid $\mb{X}$  by $\mb{E}(\mb{X}, \mb{G})$. 
\end{definition}
If we replace the action functor $\rho\colon\mb{E} \times \mb{G} \rightarrow \mb{E}$ in \Cref{Definition:principal $2$-bundle over Liegroupoid}	by an $\eta$-twisted action $\rho\colon\mb{E} \rtimes_{\eta} \mb{G} \rightarrow \mb{E}$ for a smooth map $\eta\colon  \Mor(\mb{E})\times \Mor(\mb{G})\ra \Mor(\mb{E})$ satisfying conditions listed in	\Cref{E:Contwist}, then what we get will be called  a \textit{twisted principal $\mb{G}$-bundle over the Lie groupoid} $\mathbb{X}.$
\begin{remark}\label{Remark:OtherdefofLie2grp bundles}
	A slightly weaker version of a principal $2$-bundle over a Lie groupoid has been introduced in \cite[Definition 2.20]{MR3480061}. Our main purpose in this paper is to study the Atiyah sequences and gauge transformations on 
	the principal $2$-bundles over Lie groupoids, so it is essential that we adopt a framework that provides us a pair of ordinary principal bundles (for objects and morphisms) related by the functoriality conditions.
\end{remark}	

\begin{definition}[$2$-groupoid of principal $2$-bundles]\label{Definition:2-groupoidof$2$-groupbundle}
	Let $\mb{G}$ be a  Lie $2$-group. Let $\mb{E}(\mb{X},\mb{G})$ and  $\mb{E'}(\mb{X},\mb{G})$  be a pair of principal $\mb{G}$-bundles over a Lie groupoid $\mb{X}$.
	A \textit{morphism from $\mb{E}(\mb{X},\mb{G})$ to $\mb{E'}(\mb{X},\mb{G})$} is defined as a smooth $\mb{G}$-equivariant morphism $F:\mb{E} \rightarrow \mb{E'}$ such that the following diagram commutes up to nose
	\[
	\begin{tikzcd}
		\mb{E} \arrow[r, "F"] \arrow[d, "\pi"'] & \mb{E'} \arrow[ld, "\pi'"] \\
		\mb{X}                               &                  
	\end{tikzcd}.\]
	A \textit{$2$-morphism} from $F$ to $F'$ is defined as a smooth natural isomorphism $\eta\colon F \Longrightarrow F'$ such that  $\eta(pg)=\eta(p)1_g$ and $\pi_1(\eta(p))=1_{\pi_0(p)}$ for all $p \in E_0$ and $g \in G_0$. We have a strict $2$-groupoid, denoted ${\rm Bun}(\mb{X},\mb{G})$.
\end{definition}
\begin{definition}[Gauge $2$-group]\label{Remark:Gauge2-group}
	For an object $\mb{E}$ in $2$-groupoid ${\rm{Bun}}(\mb{X},\mb{G})$, $\mc{G}(\mb{E})$ will denote the strict automorphism $2$-group of  $\mb{E}$, and will be called the  \textit{Gauge $2$-group} of the principal $\mb{G}$-bundle $\mb{E}\ra \mb{X}$.
\end{definition}

\begin{example}
	A traditional principal $G$-bundle $\pi\colon E\to X$ over a manifold $X$ is same as a $[G\rra G]$-bundle $[E\rra E]$ over the Lie groupoid $[X\rra X].$
\end{example}

\begin{example}\label{E:Example of product bundle}
	Let $\mb{X}=[X_1\rra X_0]$ be a Lie groupoid and $\mb{G}=[G_1\rra G_0]$  a Lie $2$-group. Then we have the product principal $2$-bundle $\mb{X}\times \mb{G}=[X_1\times G_1\rra X_0\times G_0]$ over $\mb{X}.$
\end{example}

\begin{example}\label{E:Exampleprincipalpairlie2}
	For any Lie groupoid  $\mb{X}=[X_1\rra X_0]$ and a (traditional) principal $G$-bundle $\pi\colon E_0\ra X_0,$ we define a Lie groupoid $E_1=\big\{(p, \gamma, q)| \gamma\in X_1, p\in \pi^{-1}(s(\gamma)), q\in \pi^{-1}(t(\gamma))\big\}\rra E_0$ as follows. The source, target and composition maps are respectively given by $s(p, \gamma, q)=p, t(p, \gamma, q)=q$ and $(q, \gamma_2, r)\circ (p, \gamma_1, q)=(p, \gamma_2\circ \gamma_1, r)$. The Lie $2$-group $[G\times G\rra G]$ naturally acts on  $[E_1\rra E_0]$ by $(p, g)\mapsto p g$
	and $\bigl((p, \gamma, q), (g_1, g_2)\bigr)=(p g_1, \gamma, q g_2).$	
	Then we have a $[G\times G\rra G]$-bundle $[E_1\rra E_0]$ over $\mb{X}=[X_1\rra X_0],$ with obvious projection functor. 	
\end{example}
We will provide several other examples of principal $2$-bundles over  Lie groupoids. 

We recall the definition of a $G$-bundle over a Lie groupoid  \cite{MR2270285}.
In the following subsection, we will describe a construction of a principal $2$-bundle  from a given $G$-bundle over a Lie groupoid. The structure is general enough to produce a large class of examples relevant in many contexts. 

\begin{definition}[principal $G$-bundle over Lie groupoid] \label{definition:principal bundle over Lie groupoid}
	For a Lie group $G$, a \textit{principal $G$-bundle over a Lie groupoid} $\mb{X}$ is a principal $G$-bundle $\pi\colon E_G \rightarrow X_0$ along with a  smooth 
	map $\mu\colon X_1 \times_{s, X_0, \pi} E_G \rightarrow E_G$ satisfying 
	
	\begin{itemize}
		\item $\mu(1_{\pi(p)}, p)=p$ for all $p \in E_G$ where $1\colon X_0 \rightarrow X_1$ is the identity assigning map,
		\item for each $(\gamma, p) \in X_1 \times_{s,X_0,\pi}E_G $, we have $\bigl(\gamma, \mu(\gamma,p)\bigr) \in X_1 \times_{t, X_0, \pi}E_G$,
		\item if $\gamma_2 , \gamma_1 \in X_1$ such that $t(\gamma_1)=s(\gamma_2)$ and $(\gamma_1,p) \in X_1 \times_{s,X_0,\pi}E_G$ then 
		$\mu(\gamma_2 \circ \gamma_1,p)=\mu(\gamma_2,\mu(\gamma_1,p))$,
		\item for all $p \in E_G, g \in G$ and $\gamma \in X_1$ we have $\mu(\gamma, p)g=\mu(\gamma, pg).$			
	\end{itemize}
\end{definition}
We will denote  the above principal $G$-bundle over the Lie groupoid $\mathbb{X}= [X_1 \rightrightarrows X_0]$ by the notation  $\bigl(\pi\colon E_G \rightarrow X_0, \mu, [X_1 \rightrightarrows X_0]\bigr)$.
The map $\mu$ is called the \textit{action map} of the principal $G$-bundle.

The collection of principal $G$-bundles over $\mb{X}$ forms a groupoid whose morphisms are defined naturally. We will denote this groupoid by ${\rm{Bun}}(\mb{X}, G)$. 

\begin{example}\label{E:Example of principal 2-bundle}
	
	Let  $\bigl(\pi\colon E_G \rightarrow X_0, \mu, [X_1 \rightrightarrows X_0]\bigr)$ be a principal $G$-bundle.  The pull-back Lie groupoid   $[X_1 \times_{s,X_0,\pi} E_G \rra E_G],$ denoted  $[s^{*}E_G \rra E_G],$ gives the $[G\rra G]$-bundle over $[X_1\rra X_0].$ Conversely given a principal $[G\rra G]$-bundle $\pi\colon \mb{E} \rightarrow \mb{X}$  over the Lie groupoid $\mb{X},$ we take		the underlying principal $G$-bundle to be  $\pi_0\colon E_0 \rightarrow X_0$ and the action map $\eta:s^{*}E_0\rightarrow E_0$ to be $(\gamma,p) \mapsto t(\widetilde \gamma),$ where $\widetilde \gamma \in \pi_1^{-1}(\gamma)\subset s^*E_0$ is the unique element associated to $(\gamma,p)\in s^*E_0$ satisfying $s(\widetilde {\gamma})=p$.
\end{example}

\begin{example}\label{E:Example of principal 2-bundle ordinary}
	As in \Cref{Ex:ordinary}	now we consider the  single object Lie $2$-group $[G\rra e]$ of a Lie group $G$.
	Let  $\mb{E}=[E_1\rra E_0]$ be a $[G\rra e]$-bundle over the Lie groupoid ${\mb X}=[X_1\rra X_0].$	Then it is immediate from the definition of a principal $2$-bundle that  $E_0=X_0$ and $E_1$ is a principal $G$-bundle over $X_1$ such that 
\[
	\begin{tikzcd}[sep=small]
		E_1 \arrow[rr,"\pi_1"] \arrow[dd,xshift=0.75ex,"t"]
		\arrow[dd,xshift=-0.75ex,"s"'] &  & X_1 \arrow[dd,xshift=0.75ex,"t"]
		\arrow[dd,xshift=-0.75ex,"s"'] \\
		&  &                \\
		X_0 \arrow[rr,"{\rm Id}"]            &  & X_0           
	\end{tikzcd},\]	
	commutes. Moreover the functoriality of the action of $[G\rra e]$ on $E_1\rra X_0$ implies that the action $E_1\times G\ra E_1$ preserves the hom sets; that is, the restriction gives 
	\begin{equation}\label{E:restricthom}
		{\rm Hom}_{\mb{E}}(x, y)\times G \ra {\rm Hom}_{\mb{E}}(x, y)
	\end{equation}
	for all $x, y\in X_0$  and for any pair of composable morphisms $\gamma_2, \gamma_1 \in E_1,$ and $g, g'\in G$ we have \begin{equation}\label{E:compoequi}
		(\gamma_2 g)\circ (\gamma_1 g')=(\gamma_2\circ \gamma_1) g g'.
	\end{equation}

	On the other hand, suppose $[E\rra X_0]$ is a Lie groupoid, and a compact Lie group $G$  acts on $E$ freely and smoothly. 
	Assume that the source, target maps are $G$ invariant and the action satisfies the conditions in \Cref {E:restricthom,E:compoequi}. Then we have a principal $G$-bundle $E\ra E/G:=X_1,$ which defines a  $[G\rra e]$-bundle $\mb{E}=[E\rra X_0]$   over the Lie groupoid ${\mb X}=[E/G\rra X_0].$

\end{example}
The principal $2$-bundle of \Cref{E:Example of principal 2-bundle ordinary} is related to $G$-extensions as we see below.
\subsection{Lie groupoid $G$-extensions and twisted principal bundles}\label{SS:GExtension}	
We recall the definition of a Lie groupoid $G$-extension \cite{MR3480061}.

Let $G$ be a Lie group and $M$  a smooth manifold. Consider the Lie groupoid $[M\times G\rra M]$ 	with source, target maps both be projection at the first component $(m, g)\mapsto m,$  composition $(m, g_2)\circ (m, g_1)=(m, g_2 g_1),$ $1_m=(m, e)$ and $(m, g)^{-1}=(m, g^{-1}).$	A \textit{Lie groupoid $G$-extension} is a short exact sequence of Lie groupoids  
of the form
\begin{equation}\label{Dia:Gextension}
	\begin{tikzcd}[sep=small]
		1 \arrow[r] & M\times G  \arrow[dd,xshift=0.75ex]\arrow[dd,xshift=-0.75ex]  \arrow[rr, "i"] &  & \Gamma_2 \arrow[rr, "\phi"]  \arrow[dd,xshift=0.75ex]\arrow[dd,xshift=-0.75ex]   &  & \Gamma_1 \arrow[r] \arrow[dd,xshift=0.75ex]\arrow[dd,xshift=-0.75ex]  & 1 \\
		&                                                                    &  &                                                       &  &                           &   \\
		1 \arrow[r] & M \arrow[rr, "{\rm Id}"]                                                     &  & M \arrow[rr, "{\rm Id}"]                                        &  & M \arrow[r]             & 1
	\end{tikzcd},
\end{equation} where $i$ is an embedding and $\phi$ is a surjective submersion (\cite[Chapter $4$]{Moerdijk2}). 

Given a $[G\rra e]$-bundle $\mb{E}=[E_1\rra X_0]$ over the Lie groupoid $\mb{X}=[X_1\rra X_0]$, define the map $i\colon X_0\times G \ra E_1$ by 
$$(x, g)\mapsto  1^{\mb{E}}_{x} \,g,$$	
where $1^{\mb{E}}$	is the identity assigning map in the Lie groupoid $\mb{E}=[E_1\rra X_0].$  Since $1^{\mb E}\colon X_0\ra E_1$ is a diffeomorphism onto its image, by \Cref{E:restricthom}, the map $i$  is an embedding. As $E_1\ra X_1$ is a principal $G$-bundle, it follows that,
\[
\begin{tikzcd}[sep=small]
	1 \arrow[r] & X_0\times G  \arrow[dd,xshift=0.75ex]\arrow[dd,xshift=-0.75ex]  \arrow[rr, "i"] &  & E_1 \arrow[rr, "\pi_1"]  \arrow[dd,xshift=0.75ex]\arrow[dd,xshift=-0.75ex]   &  & X_1 \arrow[r] \arrow[dd,xshift=0.75ex]\arrow[dd,xshift=-0.75ex]  & 1 \\
	&                                                                    &  &                                                       &  &                           &   \\
	1 \arrow[r] & X_0 \arrow[rr, "{\rm Id}"]                                                     &  & X_0 \arrow[rr, "{\rm Id}"]                                        &  & X_0 \arrow[r]             & 1
\end{tikzcd},\]
is a Lie groupoid $G$-extension. 

Conversely, suppose we have a Lie groupoid $G$-extension as in Diagram \ref{Dia:Gextension}.  Observe that commutation of the second square from the left  implies that $$i(x, g)\in {\rm Hom}_{\Gamma_2}(x, x).$$

\begin{lemma}\label{L:Almost}
	
	\begin{enumerate}
		\item There exists a smooth free action $\Gamma_2\times G\ra \Gamma_2$ given by  \begin{equation}\label{E:ActionGamma2}
			(x \xrightarrow{\gamma}y, g)\mapsto \gamma\circ i(x, g).
		\end{equation} 
		
		\item The action is transitive on the fibers $\phi^{-1}(\lambda),$ for each $\lambda\in \Gamma_1$ 
		
		\item  $\phi$ is constant on the orbits of the action. 
		
		\item The action satisfies the condition in \Cref{E:restricthom};
		\begin{equation}\label{E:Restrichomgamma}
			{\rm Hom}_{\Gamma_2}(x, y)\times G \ra {\rm Hom}_{\Gamma_2}(x, y)
		\end{equation}
		for all $x, y\in M.$  
	\end{enumerate}
\begin{proof}
	\begin{enumerate}
		\item Functoriality of $i$ implies that
		\[	\gamma (gg')=\gamma \circ i(x, g\, g')
		=\gamma \circ \bigl(i(x, g)\circ i(x, g')\bigr)
		=\bigl(\gamma \circ i((x, g)\bigr)\circ i(x, g')
		=(\gamma g) g'.\]
		Injectivity of $i$ ensures the action is free. 
		\item If $\phi(\gamma_2)=\phi(\gamma_1).$ Then $\phi(\gamma_2^{-1}\circ \gamma_1)=1^{\Gamma_1}_x.$ Thus by  exactness of Diagram \ref{Dia:Gextension} we have $(x, g)\in M \times G$ such that $\gamma_2^{-1}\circ \gamma_1=i(x, g).$ Hence $\gamma_1=\gamma_2 g.$  
		
		\item $\phi(\gamma g)=\phi(\gamma\circ i(x, g))=\phi(\gamma)\circ \phi (i(x, g)).$ Since \ref{Dia:Gextension} is a short exact sequence 
		$\phi (i(x, g))=1^{\Gamma_1}_x.$ Therefore $\phi(\gamma g)=\phi(\gamma).$
		\item Readily follows from the observation that $i(x, g)\in {\rm Hom}_{\Gamma_2}(x, x).$
	\end{enumerate}
\end{proof}\end{lemma}
In conclusion,  since $\phi$ is a surjective submersion, the Lie groupoid $G$-extension in Diagram \ref{Dia:Gextension} produces a  principal $G$-bundle $\phi\colon \Gamma_2\ra \Gamma_1$ such that the diagram
\[
\begin{tikzcd}[sep=small]
	\Gamma_2 \arrow[rr,"\phi"] \arrow[dd,xshift=0.75ex,"t"]
	\arrow[dd,xshift=-0.75ex,"s"'] &  & \Gamma_1 \arrow[dd,xshift=0.75ex,"t"]
	\arrow[dd,xshift=-0.75ex,"s"'] \\
	&  &                \\
	M \arrow[rr,"{\rm Id}"]            &  & M          
\end{tikzcd},\]	
commutes, and the action satisfies \Cref{E:Restrichomgamma}. The  action defined in \Cref{L:Almost} in general \textit{does not} satisfy the other functoriality condition in \Cref{E:compoequi}! Thus  the principal $G$-bundle $\phi \colon \Gamma_2\ra \Gamma_1$ \textit{does not} define a  $[G\rra e]$-bundle $ [\Gamma_2\rra M]$   over  the Lie groupoid $[\Gamma_1\rra M].$
However  we show that  it is actually a {twisted principal $[G\rra e]$-bundle over the Lie groupoid} $[\Gamma_1\rightrightarrows M]$. For that we define 
\begin{equation}\label{E:TwistGextension}
	\begin{split}
		\eta\colon &\Mor(\Gamma_2)\times G \ra \Mor(\Gamma_2)\\
		&(x\xrightarrow{\gamma} y, g)\mapsto i(y, g)\circ \gamma\circ i(x, g^{-1}).
	\end{split}
\end{equation}	
It is readily verified that $\eta$ satisfies all the required conditions of \Cref{E:Contwist}. Thus we have the twisted product category $\bigl([\Gamma_2\rra M]\rtimes_{\eta} [G\rra e]\bigr).$ Then we verify,
\begin{equation}\label{E:Veriftwisfunc}
	\begin{split}
		\bigl(y\xrightarrow{\gamma_2}z\,  {g'}\bigr)\circ \bigl(x\xrightarrow{\gamma_1}y\,g\bigr)
		= &\bigl({\gamma_2}\circ  i(y, g')\bigr)\circ \bigl({\gamma_1}\circ i(x, g)\bigr)\\
		= &\bigl({\gamma_2}\circ  (i(y, g')\circ {\gamma_1}\circ i(x, g'^{-1})\bigr)\circ i(x, g' g)\\
		= & \bigl({\gamma_2}\circ  \eta (\gamma_1, g')\bigr)\circ i(x, g' g)
		= \bigl({\gamma_2}\circ  \eta (\gamma_1, g')\bigr)g' g, 
	\end{split}
\end{equation} 
which is precisely the condition, stated in \Cref{Remark:functorialitytwisted}, required for the action in \Cref{E:ActionGamma2} to be an $\eta$-twisted action
$$\bigl([\Gamma_2\rra M]\rtimes_{\eta} [G\rra e]\bigr)\ra [\Gamma_2\rra M].$$
Observe that in the other direction the construction proceeds for a twisted $[G\rra e]$-bundle exactly same way as was  for a  \textit{non-twisted} $[G\rra e]$-bundle (using the  third property in  \Cref{E:Contwist}).

\begin{proposition}
	Let $G$ be a Lie group. A twisted principal $[G\rra e]$-bundle defines a Lie groupoid $G$-extension and vice-versa. 
\end{proposition}

It is pertinent to note here that \cite{MR3480061} also discussed the relation between Morita equivalent classes of Lie groupoid $G$-extensions and Morita equivalence classes of what the authors called $[G\rra {\rm Aut} (G)]$-bundles (see \cite[Theorem $3.4$]{MR3480061}).  

\subsection{Decorated principal $2$-bundles}\label{SS:Decorated}
In this subsection we are going to construct the main example for this paper.
Given a Lie $2$-group $\mb{G}=[G_1\rra G_0]$ and a principal $G_0$-bundle over a Lie groupoid $\mb{X}$, we will construct a principal $\mb{G}$-bundle over $\mb{X}$. We will call them decorated principal $2$-bundles. Let $\ghta$ be the associated crossed module of $\mb{G}.$ Here, of course, $G_0=G.$

An analogous notion for a principal $2$-bundle over a path space groupoid already exists in literature,  introduced in \cite{MR3126940} by decorating the space of ${\bar A}$-horizontal paths ${\mathcal P}_{\bar A}P,$ for a connection   $\bar A$ on a principal $G_0$-bundle $P\to M$.  

\begin{proposition}\label{Prop:Decoliegpd}
	Let $(G, H, \tau\colon H \rightarrow G, \alpha:G \times H \rightarrow H)$ be a Lie crossed module and $\mb{G}=[H\rtimes G\rra G]$ the associated Lie $2$-group. Let  $\bigl(\pi\colon E_G \rightarrow X_0, \mu, [X_1 \rightrightarrows X_0]\bigr)$  a principal $G$-bundle over the Lie groupoid $\mb{X}=[X_1 \rightrightarrows X_0]$. Let us denote $s^{*}E_G \times H= (X_1 \times_{s, X_0, \pi} E_G) \times H$ by $(s^{*}E_G)^{\rm{dec}}$. 
	\begin{enumerate}
		\item	The manifolds $(s^{*}E_G)^{\rm{dec}}$ and $E_G$ determines a Lie groupoid $[(s^{*}E_G)^{\rm{dec}} \rightrightarrows E_G]$ whose structure maps are given as 
		\begin{itemize}
			\item source map $\tilde{s}$: $(\gamma, p, h) \mapsto p$,
			\item target map $\tilde{t}$: $(\gamma, p, h) \mapsto \mu(\gamma, p) \tau(h^{-1})$,
			\item composition map $\tilde{m} \colon \big((\gamma_2, p_2, h_2), (\gamma_1, p_1, h_1) \big) \mapsto (\gamma_2 \circ \gamma_1, p_1 ,h_2h_1)$, 
			\item unit map $\tilde{u} : p \mapsto (1_{\pi(p)},p,e_H)$,
			\item inverse map $ \tilde{i}\colon \bigl(\gamma, p, h) \mapsto (\gamma^{-1}, \mu(\gamma,p)\tau(h^{-1}), h^{-1}\bigr)$.
		\end{itemize}
		\item The Lie groupoid  $\mb{E}^{\rm{dec}}:=[(s^{*}E_G)^{\rm{dec}} \rightrightarrows E_G]$ forms a principal $\mb{G}$-bundle over $[X_1\rra X_0]$, whose action functor and the bundle functor respectively given by 
		\begin{equation}\label{E:Actionondeco}
			\begin{split}
				\rho\colon &\mb{E}^{\rm dec}\times \mb{G}\ra \mb{E}^{\rm dec}\\
				&(p, g) \mapsto p\, g\\
				\bigl((\gamma, p, h)&, (h', g)\bigr)\mapsto \bigl(\gamma, p g, \alpha_{g^{-1}}(h'^{-1}\, h)\bigr),
			\end{split}
		\end{equation}
		and
		\begin{equation}\label{E:Projondeco}
			\begin{split}
				\pi\colon &\mb{E}^{\rm dec}\ra \mb{X}\\
				& p \mapsto \pi(p)\\
				\bigl(\gamma,&  p, h\bigr) \mapsto \gamma.
			\end{split}
		\end{equation}
	\end{enumerate}
	We call $\mb{E}^{\rm{dec}}\ra [X_1\rra X_0]$ the \textit{decorated $\mb{G}$-bundle associated to $(E_G\ra X_0,[X_1\rra X_0])$}.
	\begin{proof}
		(1)		It is straightforward to verify that  $[(s^{*}E_G)^{\rm{dec}} \rightrightarrows E_G]$   is a groupoid.
		
		Since $E_G\ra X_0$ is a surjective submersion,  $(s^{*}E_G)^{\rm{dec}}$ is a smooth manifold. Moreover, as ${\tilde s}$ is the composition of surjective submersions, it (hence the target map as well) is a surjective submersion. The smoothness of other structure maps directly follows from the smoothness of $\mu$ and smoothness of structure maps of the Lie groupoid $[X_1\rra X_0].$ 
		
		(2) Since the action of $G$ on $E_G$ is free and 
		$\alpha_g\in {\rm Aut}(H)$ for each $g\in G,$ the action of $H\rtimes_{\alpha}G$ on $(s^{*}E_G)^{\rm{dec}}$ in \Cref{E:Actionondeco} is free. 
		Observe that for any $\gamma \in X_1,$ $\pi^{-1}({\gamma})=\pi^{-1}(s(\gamma))\times H.$ Now it is straightforward to verify that the action is transitive on $\pi^{-1}({\gamma})$, and $\pi\bigl((\gamma, p, h) (h', g)\bigr)=\pi(\gamma, p, h).$	
		
		The local trivialization on the principal $G$-bundle $\pi\colon E_G\ra X_0$ induces a local trivialization on  the map $\pi\colon (s^{*}E_G)^{\rm{dec}}\ra \mb{X}_1$	as follows. For any $\gamma\in X_1,$ there exists a local trivialization $(U\subset X_0, \phi)$ such that $s(\gamma)\in U$ and $\phi\colon \pi^{-1}(U)\ra U\times G$ a diffeomorphism. Then we have a local trivialization $(s^{-1}(U), \tilde \phi)$ around $\gamma,$ given by
		\begin{equation}\nonumber
			\begin{split}
				\tilde \phi\colon &\pi^{-1}\bigl(s^{-1}(U)\bigr)\ra s^{-1}(U)\times (H\rtimes_{\alpha} G)\\
				&(\gamma, p, h)\mapsto (\gamma, \alpha_g(h), g),
			\end{split}
		\end{equation}
		where $g={\rm pr}_2\circ \phi(p)\in G.$ This completes the proof.		
	\end{proof}
\end{proposition}

	\begin{example}
	The product  Lie $2$-group $\mb{G}=[G_1\rra G_0]$-bundle $\mb{X}\times \mb{G}$ over $\mb{X}=[X_1\rra X_0]$
	in \Cref{E:Example of product bundle} can be viewed as a decorated bundle. Note that $s^{*}(X_0\times G_0)=X_1\times G_0.$ The action of $X_1$ on $X_0\times G_0$ is given by $s^* (X_0\times G_0)\ra X_0\times G_0$ by $(\gamma, g)\mapsto (t(\gamma), g).$ 
	Thus the Lie groupoid $[X_1\times G_0\rra X_0\times G_0]$ is the pullback Lie groupoid, and by decoration we recover the $\mb{X}\times \mb{G}.$
	
\end{example}

\begin{example}\label{Ex:EoXodecobundle}
	Let $\ghta$ be a Lie crossed module associated to a Lie $2$-group $\mb{G}$. Consider a traditional principal $G$-bundle $E\ra X.$ The Lie groupoid $[X\rra X]$ trivially acts on $E$ by $(x, p)\mapsto p.$ Then the  Lie  groupoid ${E}^{\rm dec}:=[E\times H\rra E]$ defines the associated decorated $\mb{G}$-bundle. 
\end{example}

\subsection{Categorical connections}\label{SS:Catconnection}
We will, later on, discuss the differential geometric notion of a connection on a 
principal $2$-bundle over a Lie groupoid. Apart from that, a principal $2$-bundle can also be equipped with a lifting property for arrows, which we call a `categorical connection'.
\begin{definition}\label{Def:categorical connection}
	Let $\mb{G}$ be a Lie 2-group and $\pi\colon \mb{E} \ra \mb{X}$ a $\mb{G}$-bundle over $\mb{X}$. A \textit{categorical connection} $\mathcal{C}$ on $\mb{E}(\mb{X}, \mb{G})$ is  an embedding $\mathcal{C}\colon {s}^{*}E_0 \ra E_1$, and   satisfy following conditions:
	\begin{itemize}
		\item $s(\mathcal{C}(\gamma,p))=p$  for all $(\gamma,p) \in s^{*}E_0$,
		\item $\pi_1(\mathcal{C}(\gamma,p))= \gamma$ for all $(\gamma,p) \in s^{*}E_0$,
		\item if $(\gamma_2, p_2), (\gamma_1, p_1) \in {s}^{*}E_0$ such that ${s}(\gamma_2)={t}(\gamma_1)$ and $p_2=t\bigl({\mathcal C}(\gamma_1, p_1)\bigr),$ then $\mathcal{C}(\gamma_2 \circ \gamma_1 , p_1)= \mathcal{C}(\gamma_2, p_2) \circ \mathcal{C}(\gamma_1, p_1)$,
		\item $\mathcal{C}(1_x,p)=1_p$  for any $x\in X_0$ and $p\in \pi^{-1}(x)$,
		\item $\mathcal{C}(\gamma, p. g)= \mathcal{C}(\gamma, p) \cdot 1_g$ for all $(\gamma, p) \in {s}^{*}E_0$ and $g \in G_0$.
	\end{itemize}
	We call $\mathcal{C}(\gamma, p)$ the (unique) $\mathcal{C}$\textit{-horizontal lift of $\gamma \in X_1$ through the point} $p \in \pi_0^{-1}\bigl(s(\gamma)\bigr).$
\end{definition}

The definition is motivated by the notion of horizontal lifting of a path by a connection in the traditional set-up of principal bundles. See \cite{MR3126940} for an analogous notion of a categorical connection on
the set-up of path space groupoid. 
We will observe that categorical connections characterize the decorated bundles.

\begin{example}
	There exists a unique categorical connection on a  $[G\rra G]$-bundle over a Lie groupoid. 
\end{example}	
\begin{example}\label{Ex:CatconnDeco}
	Let  $\bigl[(s^{*}E_G)^{\rm{dec}} \rightrightarrows E_G\bigr]\ra [X_1\rra X_0]$ be  the 
	decorated  $\mb{G}$-bundle obtained in \Cref{Prop:Decoliegpd}. Any smooth map $\beta\colon E_G\ra H$ satisfying $\beta(p g)=\alpha_{g^{-1}}(\beta(p))$ for $p\in E_G, g\in G$, defines  a categorical connection
	${\mathcal C}\colon (\gamma, p)\mapsto \bigl(\gamma, p, 
	\beta(p)\beta(\mu(\gamma,p))^{-1} \bigr)$. 
	For the constant map $\beta\colon p\mapsto e,$ we have the categorical connection $(\gamma, p)\mapsto (\gamma, p, e).$ We call the categorical connection $(\gamma, p)\mapsto (\gamma, p, e),$ the \textit{canonical categorical connection} on the decorated bundle. 
\end{example}	

Given a categorical connection  $\mathcal C$ on a $\mb{G}$-bundle  $\pi\colon \mb{E} \ra \mb{X}$  over $\mb{X}=[X_1\rra X_0],$ we can define an action of $X_1$ on $E_0$ 
\begin{equation}\label{actionwithcatcon}
	\begin{split}
		\mu\colon s^{*}E_0&\ra E_0,\\
		(\gamma, p)&\mapsto t\bigl({\mathcal C}(\gamma, p)\bigr),
	\end{split}
\end{equation}
which turns $s^*E_0$ into a Lie groupoid  $[s^{*}E_0\rra E_0]$, and ${\mathcal C}$ allows us to identify $[s^*E_0\rra E_0]$ as a Lie sub groupoid of $[E_1\rra E_0]$, 	\begin{equation}\label{Dia:Liegpdembedding}
	\begin{tikzcd}[sep=small]
		s^*E_0 \arrow[rr,"\mathcal C"] \arrow[dd,xshift=0.75ex]
		\arrow[dd,xshift=-0.75ex] &  & E_1 \arrow[dd,xshift=0.75ex]
		\arrow[dd,xshift=-0.75ex] \\
		&  &                \\
		E_0 \arrow[rr,"{\rm Id}"]            &  & E_0          
	\end{tikzcd}.
\end{equation}
Note that  $\mu$ satisfies $\mu (\gamma, p) g=t({\mathcal C}(\gamma, p)) g=t({\mathcal C}(\gamma, p)1_g)=t({\mathcal C}(\gamma, p g))=\mu(\gamma, pg).$ Identifying $G$ as  a  subgroup of $G_1=H\rtimes G,$ we observe that the fifth condition in \Cref{Def:categorical connection} implies  $[s^*E_0\rra E_0]\simeq [{\mathcal C}(s^*E_0)\rra E_0]$ is in fact a sub Lie groupoid bundle over $\mb{X}=[X_1\rra X_0]$ with a reduced structure Lie $2$-group $[G\rra G].$ Given any ${\widetilde \gamma}\in E_1$ it follows that we have a unique $h\in H$ such that $[{\mathcal C}\bigl(\pi(\widetilde \gamma), s(\widetilde \gamma)\bigr)](h, e)=\widetilde \gamma,$ since $\pi(\widetilde \gamma)=\pi \bigl({\mathcal C}\bigl(\pi(\widetilde \gamma), s(\widetilde \gamma)\bigr)\bigr)$ and $s(\widetilde \gamma)=s\bigl({\mathcal C}\bigl(\pi(\widetilde \gamma), s(\widetilde \gamma)\bigr).$
In turn, we get an isomorphism of $\mb{G}$-bundles
\begin{equation}
	\begin{split}
		&\theta\colon s^*E_0\times H \ra E_1\\
		&((\gamma, p), h\bigr)\ra {\mathcal C}(\gamma, p) (h^{-1}, e).
	\end{split}
\end{equation}
\begin{lemma}\label{Lem:isodecgeneral} 
	The map $(\theta,\rm{Id})$ is an isomorphism from $\mb{E}^{\rm{dec}}(\mb{X},\mb{G})$ to $\mb{E}(\mb{X,\mb{G}})$.
	\begin{proof}
		The functoriality of $\theta$ follows from the composition law for the decorated Lie groupoid in \Cref{Prop:Decoliegpd} and the functoriality of the action of $\mb{G}$,  
		\begin{equation}
			\begin{split}
				&\theta\bigl((\gamma_2, p_2), h_2)\bigr)\circ \theta\bigl((\gamma_1, p_1), h_1\bigr)\\
				&=\bigl({\mathcal C}(\gamma_2, p_2) (h_2^{-1}, e)\bigr)\circ\bigl({\mathcal C}(\gamma_1, p_1) (h_1^{-1}, e)\bigr)\\
				&=\bigl({\mathcal C}(\gamma_2, p_2) (h_2^{-1}, e)\bigr)\circ\bigl({\mathcal C}(\gamma_1, p_1)(e, \tau(h_1^{-1})) (e, \tau(h_1)) (h_1^{-1}, e)\bigr)\\
				&=\bigl({\mathcal C}(\gamma_2, p_2) (h_2^{-1}, e)\bigr)\circ\bigl({\mathcal C}(\gamma_1, p_1\tau(h_1)^{-1}) (e, \tau(h_1)) (h_1^{-1}, e)\bigr)\\
				&=\bigl({\mathcal C}(\gamma_2, p_2) (h_2^{-1}, e)\bigr)\circ\bigl({\mathcal C}(\gamma_1, p_1\tau(h_1)^{-1}) (e, \tau(h_1)) (h_1^{-1}, e)\bigr)\\
				&=\bigl({\mathcal C}(\gamma_2, p_2) \circ {\mathcal C}(\gamma_1, p_1\tau(h_1)^{-1})    \bigr)\bigl((h_2^{-1}, e) \circ \bigl((e, \tau(h_1)) (h_1^{-1}, e)\bigr)\bigr)\\
				&={\mathcal C}\bigl(\gamma_2\circ \gamma_1, p_1\tau(h_1)^{-1}   \bigr)\bigl((h_2^{-1}, e) \circ (h_1^{-1}, \tau(h_1))\bigr)\\
				&={\mathcal C}\bigl(\gamma_2\circ \gamma_1, p_1\tau(h_1)^{-1}\bigr) \bigl(h_2^{-1} h_1^{-1}, \tau(h_1)\bigr)\\
				&={\mathcal C}\bigl(\gamma_2\circ \gamma_1, p_1\bigr) \bigl(e, \tau(h_1^{-1})\bigr)   \bigl(h_2^{-1} h_1^{-1}, \tau(h_1)\bigr)\\
				&={\mathcal C}\bigl(\gamma_2\circ \gamma_1, p_1\bigr) \bigl(h_1^{-1} h_2^{-1}, e \bigr) \\
				&=\theta\bigl((\gamma_2\circ \gamma_1, p_1), h_2 h_1\bigr).
			\end{split}
		\end{equation}
		To see the equivarinacy, note that by the action defined in \Cref{Prop:Decoliegpd},
		$(\gamma, p, h)(h', g')=(\gamma, p g', \alpha_{g'^{-1}}(h'^{-1}h)).$ Thus 
		\begin{equation}
			\begin{split}
				&\theta\bigl((\gamma, p, h)(h', g')\bigr)\\
				&={\mathcal C}(\gamma, p g')\bigl(\alpha_{g'^{-1}}(h^{-1}h'), e\bigr)\\
				&={\mathcal C}(\gamma, p)(e, g')\bigl(\alpha_{g'^{-1}}(h^{-1}h'), e\bigr)\\
				&={\mathcal C}(\gamma, p)(h^{-1}, e)(h, e)(e, g')\bigl(\alpha_{g'^{-1}}(h^{-1}h'), e\bigr)\\
				&=\theta \bigl((\gamma, p), h\bigr)(h', g').
			\end{split}
		\end{equation}
		Obviously  $(\theta, \rm Id)$ is a morphism of principal $2$-bundles.
	\end{proof}
\end{lemma}
In conclusion, we have the following characterization of principal $2$-bundles over Lie groupoids in terms of the categorical connections.  
\begin{proposition}\label{prop:Characterisdecorated}
	A principal $2$-bundle over a Lie groupoid is a decorated bundle if and only if it admits  a categorical connection.
\end{proposition}

\begin{corollary}\label{Corollary:discreteisdecorated}
	Let $\ghta$ be a Lie crossed module associated to a Lie $2$-group $\mb{G}$. Let $\mb{E}=[E_1\rra E_0]$ be a $\mb{G}$-bundle over $[X\rra X].$ Then $\mb{E}=[E_1\rra E_0]$ is the decorated bundle defined in \Cref{Ex:EoXodecobundle}.\end{corollary}
\begin{proof}
	Any such principal $2$-bundle admits a unique categorical connection $(1_x, p)\mapsto 1_p$ for $p\in E, x=\pi(p).$
\end{proof}
\begin{example}\label{E:Exampleprincipalpairlie2deco}
	The $[G\times G\rra G]$-bundle $[E_1\rra E_0]$ over $\mb{X}=[X_1\rra X_0]$ in \Cref{E:Exampleprincipalpairlie2} admits a categorical connection if and only if $E_0\ra X_0$ is a principal $G$-bundle over the Lie groupoid $\mb{X}=[X_1\rra X_0].$	Suppose $E_0\ra X_0$ is a principal $G$-bundle over the Lie groupoid $\mb{X}=[X_1\rra X_0]$ with respect to the action $\mu\colon X_1\times_s E_0\ra E_0.$ Then the correspondence with the decorated bundle is given by $(p, \gamma, q)\mapsto\bigl ((p, \gamma), g\bigr),$ where $\mu(\gamma, p)=q\, g.$

\end{example}


\section{Atiyah sequence for principal $2$-bundles over Lie groupoids}\label{Atiyahsequencefor$2$-groupbundle}
Let $G$ be a Lie group, and $ L(G)$ the Lie algebra of $G$. 
For a principal $G$-bundle $\pi\colon P\ra M$	one can associate  a short exact sequence 
\begin{equation}\label{Atiyahseqclass}
	0\ra {\rm Ad}(P)\xra{j^{/G}} {\rm At}(P)\xra{\pi_{*}^{/G}} TM\ra 0,
\end{equation}
of vector bundles over the manifold $M$, known as \textit{Atyiah sequence} \cite{MR86359}. Here ${\rm Ad}(P):=(P\times  L(G))/G\ra M$ and  ${\rm At}(P):= TP/G\ra M$ are respectively the  adjoint bundle  and the Atiyah  bundle 	associated to the principal $G$-bundle $\pi\colon P\ra M,$ and the associated maps are  respectively given by
\begin{equation}\label{E:JPi}
	\begin{split}
		&j^{/G}\colon [(p, B)]\mapsto [\bigl(p, \delta_p(B)\bigr)],\\
		&\pi_{*}^{/G}\colon [(p, v)]\mapsto \bigl(\pi(p), \pi_{*, p}(v)\bigr),	
	\end{split}
\end{equation}
where $\delta(B)$ denotes the vertical vector field generated by the element $B\in L(G)$. The Atiyah bundle $\At(P)\ra M$,  with  $\pi_{*}^{/G}:\At(P)\ra TM$ being the anchor map, forms a Lie algebroid, known as the  \textit{Atiyah Lie algebroid} of the principal $G$-bundle $\pi:P\ra M$.

A connection on a principal $G$-bundle $P(M, G)$ is a splitting of the Atiyah sequence in \Cref{Atiyahseqclass}. Equivalently, a connection  is same as  a retract $R\colon {\rm At}(P)\ra {\rm Ad}(P)$ of $j^{/G}$, or  a section $\sigma\colon TM\ra {\rm At}(P)$ of $\pi_*^{/G}$.	The first equivalent condition for the splitting  in turn  gives an $L(G)$-valued  differential $1$-form $\omega\colon TP\ra P\times  L(G)$ satisfying,
\begin{itemize}
	\item $\omega (p g, v g)={\rm ad}_{g^{-1}}(\omega(p, v))$, for $g\in G, p\in P, v\in T_pP$ and
	\item $\omega(p, \delta_p(B))=B,$ for $p\in P$, $B\in L(G),$ where as before $\delta(B)$ denotes the vertical vector field generated by the element $B\in L(G),$	
\end{itemize}	
while the second equivalent condition gives the horizontal sub-bundle $\mc{H}\subset TE$ satisfying,
\begin{itemize}
	\item $\mc{H}_{p g}=(R_g)_{*,p}\mc{H}_p$ for all $p\in P, g\in G$ and
	\item giving a smooth splitting  $\mc{H}_p\oplus \ker{\pi_{*, p}}=T_pP$ for each $p\in P.$
\end{itemize}	
Here, $R_g:P\ra P$ is the right translation map by the element $g\in G$.

We refer to \cite[Appendix $A$]{MR896907} for a more detailed discussion on this association. 

For a given $\mb{G}=[G_1\rra G_0]$-bundle $\mb{E}=[E_1\rra E_0]$ over the Lie groupoid $\mb{X}=[X_1\rra X_0,]$ we have a pair of Atiyah sequences,
\begin{equation}\label{Dia:Obj-MorAtiyah}
	\begin{tikzcd}[sep=small]
		0 \arrow[r] & {{\rm Ad}}(E_i) \arrow[dd] \arrow[rr, "j_i^{/G_i}"] &  & {\rm {At}}(E_i) \arrow[rr, "(\pi_i)_*^{/G_i}"] \arrow[dd] &  & TX_i \arrow[r] \arrow[dd] & 0 \\
		&                                                                    &  &                                                       &  &                           &   \\
		0 \arrow[r] & X_i \arrow[rr]                                                     &  & X_i \arrow[rr]                                        &  & X_i \arrow[r]             & 0
	\end{tikzcd},
\end{equation}
for $i\in\{0, 1\}$.	 These pairs of short exact sequences are intertwined by the functoriality of various maps. We will see the right framework to describe the Atiyah sequences of a
principal $2$-bundle is that of so-called `a Vector bundle groupoid' or `a VB-groupoid' in short. For more details about VB-groupoids, we refer the reader to \cite{Bursztyn2016163}. Let us recall the definition given in \cite{MR3696590}.

\begin{definition}[VB-groupoid {\cite[Definition $3.1.$]{MR3696590}}]
	Let $[X_1\rra X_0]$ be a Lie groupoid. A \textit{VB-groupoid over $[X_1\rra X_0]$} is given by a morphism of  Lie groupoids 
	\[
	\begin{tikzcd}[sep=small]
		V_1 \arrow[rr,"\pi_1"] \arrow[dd,xshift=0.75ex,"t_V"]
		\arrow[dd,xshift=-0.75ex,"s_V"'] &  & X_1 \arrow[dd,xshift=0.75ex,"t_X"].  		\arrow[dd,xshift=-0.75ex,"s_X"'] \\
		&  &                \\
		V_0 \arrow[rr,"\pi_0"]            &  & X_0           
	\end{tikzcd},\]
	satisfying the following conditions:
	\begin{enumerate}
		\item the maps $\pi_1\colon V_1\ra X_1$ and $\pi_0\colon V_0\ra X_0$ are vector bundles,
		\item the maps $(s_V,s_X),(t_V, t_X)$ are morphisms of vector bundles,
		\item  for appropriate $\gamma_1,\gamma_2,\gamma_3,\gamma_4\in E_1$, we have 
		$(\gamma_3\circ \gamma_1)+(\gamma_4\circ \gamma_2)=
		(\gamma_3+\gamma_4)\circ (\gamma_1+\gamma_2)$.
	\end{enumerate}
\end{definition}
Equivalently, a VB-groupoid is a Lie groupoid object in the category of smooth vector bundles.

A \textit{$1$-morphism of VB-groupoids} 
$[V_1\rra V_0]\ra [V_1'\rra V_0']$ over the base Lie groupoid $[X_1\rra X_0]$ is a morphism of Lie groupoids $\Phi=(\Phi_1.\Phi_0)\colon [V_1\rra V_0]\ra [V_1'\rra V_0']$ such that $(\Phi_0,\rm{Id}_{X_0})$, and $(\Phi_1,\rm{Id}_{X_1})$ are morphisms of vector bundles. For a pair of such  $1$-morphisms $(\Phi,\Phi'$, a \textit{$2$-morphsim}  $\eta\colon \Phi\Lrrw \Phi '$ is  a smooth natural transformation such that $(\eta, 1)$ is a morphism of vector bundles from $V_0\ra X_0$ to $V_1\ra X_1$.
\begin{proposition}	\label{Prop:naturaltransformationinVBgroupoids}
	The collection of VB-groupoids, $1$-morphisms and $2$-morphisms forms a strict $2$-category, and will be denoted as $2-\rm {VBGpd(\mb{X})}$.
\end{proposition}

	\subsection{Atiyah sequence  of VB-groupoids}       

Let $\mb{E}= [E_1\rra E_0]$ be a $\mb{G}= [G_1\rra G_0]$-bundle over $\mb{X}=[X_1\rra X_0].$ Then we show  that the pair of short exact sequences in Diagram \ref{Dia:Obj-MorAtiyah} give a short exact sequence 
of VB-groupoids. To be precise, each pair   $\{TX_i\}_{i\in\{0, 1\}}$, $\{\Ad(E_i)\}_{i\in\{0, 1\}}$ and $\{{\rm At}(E_i)\}_{i\in\{0, 1\}} $    forms a VB-groupoid over $\mb{X}=[X_1\rra  X_0],$ and the sequences define a short exact sequence of 
morphisms of VB-groupoids \cite[Section $2.4$]{MR3744376}. 

\noindent{ \bf{VB-groupoid $T\mb{X}.$}}

It is obvious that the tangent bundle maps $TX_1\ra X_1$ and $TX_0\ra X_0$  give a VB-groupoid $[TX_1\rra TX_0]\ra [X_1\rra X_0]$.

\noindent{ \bf{VB-groupoid $\Ad(\mb{E}).$}}	

To see  the groupoid structure of $[{\rm Ad}(E_1)\rra {\rm Ad}(E_0)],$ we first define 	
\begin{itemize}
	\item the source map $s^{/\sim}\colon {\rm Ad}(E_1)\ra {\rm Ad}(E_0)$  as \[[(\widetilde \gamma, K)]\mapsto [(s(\widetilde \gamma),s_{*,e}(K))],\]
	\item the target map $t^{/\sim}\colon {\rm Ad}(E_1)\ra {\rm Ad}(E_0)$  as \[[(\widetilde \gamma, K)]\mapsto [(t(\widetilde\gamma),t_{*,e}(K))],\]
\end{itemize}
where $\widetilde \gamma\in E_1$ and $K\in L(G_1).$
Next we make the following observation. Suppose \begin{equation}\label{E:Sourcetargetad}
	t^{/\sim}([(\widetilde \gamma_2 ', K_2')])=s^{/\sim}([(\widetilde \gamma_1, K_1)]).
\end{equation} 
Then there exists a $\theta \in G_0$ such that $s(\widetilde \gamma_2')\theta=t(\widetilde \gamma_1)$ and ${\ad}_{\theta}(s_{*, e}(K_2'))=t_{*, e}(K_1).$
That means $(\widetilde \gamma_2', K_2')1_{\theta}=\bigl(\widetilde \gamma_2'\, 1_{\theta}, \ad_{1_{\theta}}(K_2')\bigr)$ is composable with $(\widetilde\gamma_1, K_1).$ Note that $(\widetilde \gamma_2', K_2')1_{\theta}\in [(\widetilde \gamma_2, K_2)].$ In other words, whenever the condition in \Cref{E:Sourcetargetad}	met, there exists a composable pair respectively belonging to $ [(\widetilde \gamma_2, K_2)]$ and $[(\widetilde \gamma_1, K_1)].$
We choose such a pair and define the composition as,
\begin{equation}\label{E:Compoad}
	\bigl([(\widetilde \gamma_2, K_2)]\bigr)\circ \bigl([(\widetilde \gamma_1, K_1)]\bigr)=[\bigl(\widetilde \gamma_2\circ \widetilde \gamma_1, K_2\circ K_1\bigr)].
\end{equation} 
If $(\widetilde \gamma_2', K_2')$ and $(\widetilde \gamma_1', K_1')$ are	another pair of such elements, then there exist composable $k_2, k_1\in G_1$ such that
$(\widetilde \gamma_2', K_2')=\bigl(\widetilde \gamma_2 k_2, \ad_{k_2}(K_2)\bigr)$ and $(\widetilde \gamma_1', K_1')=\bigl(\widetilde \gamma_1 k_1, \ad_{k_1}(K_1)\bigr).$
Then \begin{equation}
	\begin{split}
		&(\widetilde \gamma_2' \circ \widetilde \gamma_1', K_2'\circ K_1')\\
		&=\bigl((\widetilde \gamma_2 k_2)\circ (\widetilde \gamma_1 k_1),   (\ad_{k_2}(K_2))\circ   (\Ad_{k_1}(K_1)) \bigr)\\
		&=\bigl((\widetilde \gamma_2\circ \widetilde \gamma_1)(k_2 k_1),   (\ad_{k_2 k_1}(K_2\circ K_1)) \bigr)\\
		&=(\widetilde \gamma_2 \circ \widetilde \gamma_1, K_2\circ K_1) (k_2 k_1),
	\end{split}
\end{equation}
where in the third step we have used the functoriality of the Lie $2$-group action. Thus \Cref{E:Compoad} is well defined. The inverse and identity maps are the obvious ones. 

Considering the following commutative diagram
\[\begin{tikzcd}
	E_1 \times L(G_1) \arrow[rr, "\delta_0"] \arrow[d, "s_{\mb{E}} \times s_{{\mb{G}}_{*,e}}"'] &  & {\rm Ad}(E_1) \arrow[d, "s^{/\sim}", dotted] \\
	E_0 \times L(G_0) \arrow[rr, "\delta_1"']                &  & {\rm Ad}(E_0)
\end{tikzcd},\]
we see that as  $s_{\mb E}$ and $s_{{\mb{G}}_{*,e}}$ are surjective submersions, we  have $s^{/\sim} \circ \delta_0$  a surjective submersion. Since $\delta_0$ is a surjective submersion, it follows that $s^{/\sim}$ is also a surjective submersion. Similarly, $t^{/\sim}$ is  also a surjective submersion. Thus $\Ad{(\mb{E})}=[\Ad(E_1)\rra \Ad(E_0)]$ is a Lie groupoid.

The (quotient) vector bundles ${\rm Ad}(E_1)\ra X_1$, and ${\rm Ad}(E_0)\ra X_0$ form a  VB-groupoid $[{\rm Ad}(E_1)\rra {\rm Ad}(E_0)]\ra [X_1\rra X_0]$.

\noindent{ \bf{VB-groupoid ${\rm At}(\mb{E}).$}}

We leave out the technical details of the construction of the Lie groupoid  $[{\rm At}(E_1)\rra {\rm At}(E_0)],$ which are almost identical  to the construction  of the Lie groupoid $\Ad{(\mb{E})}=[\Ad(E_1)\rra \Ad(E_0)].$
The Lie groupoid structure on $[{\rm At}(E_1)\rra {\rm At}(E_0)]$
has the following description. The source, target maps are respectively given by
\begin{itemize}
	\item  $s_*^{/\sim}\colon {\rm At}(E_1)\ra {\rm At}(E_0),$ \[[(\widetilde \gamma, \widetilde X)]\mapsto [(s(\widetilde \gamma),s_{*,\widetilde \gamma}(\widetilde X))],\]
	\item $t_*^{/\sim}\colon {\rm At}(E_1)\ra {\rm At}(E_0),$ \[[(\widetilde \gamma, \widetilde X)]\mapsto [(t(\widetilde \gamma),t_{*,\widetilde \gamma}(\widetilde X))]\]	
\end{itemize}
and composition is given by 
\begin{equation}\label{E:AtCompo}
	[(\widetilde \gamma_2, \widetilde X_2)]\circ [(\widetilde \gamma_1, \widetilde X_1)]=[(\widetilde \gamma_2\circ \widetilde \gamma_1, \widetilde X_2\circ \widetilde X_1)]
\end{equation}
for appropriate choice of representative elements of the equivalence classes. Then the vector bundles ${\rm At}(E_1)\ra X_1$ and ${\rm At}(E_0)\ra X_0$ give a VB-groupoid $[{\rm At}(E_1)\rra {\rm At}(E_0)]\ra [X_1\rra X_0]$.

\begin{lemma}
	Let $\mb{E}=[E_1\rra E_0]$ be a $\mb{G}=[G_1\rra G_0]$-bundle over the Lie groupoid $\mb{X}=[X_1\rra X_0].$
	The vertical vector fields generating maps $\delta_p\colon L(G_0)\ra T_p E_0$ and $\delta_{(p\xra{\widetilde \gamma} q)}\colon L(G_1)\ra T_{\widetilde \gamma}E_1$ define a functor $\delta\colon \mb{E}\times L(\mb{G})\ra T(\mb{E}).$ Moreover the 
	functor $\delta$ is $\mb{G}$ equivariant in the sense that, $\delta \bigl(p g, {\ad}_{g^{-1}}(B)\bigr)=\delta (p, B) \cdot g$ and $\delta \bigl(\widetilde \gamma k, {\ad}_{k^{-1}}(K)\bigr)=\delta (\widetilde \gamma, K) \cdot k,$
	for any $p\in E_0, g\in G_0, B\in L(G_0)$ and $\widetilde \gamma\in E_1, k\in G_1, K\in L(G_1).$
\begin{proof}
It is an immediate consequence of the functoriality of the Lie $2$-group action. 
\end{proof}
\end{lemma}
Now it is  evident  that the maps $\{j_i^{/G_i}\}_{i\in \{0, 1\}}$ and $\{(\pi_i)_*^{/G_i}\}_{i\in \{0, 1\}}$ in Diagram \ref{Dia:Obj-MorAtiyah} respectively define the morphisms of VB-groupoids, $j^{/\mb{G}}\colon \Ad{\mb{E}}\ra {\rm At}(\mb{E})$ and $\pi_*^{/\mb{G}}\colon {\rm At}(\mb{E})\ra T\mb{X}.$ Hence we conclude the following.
\begin{proposition}\label{Prop:AtiyahLie2gpd}
	Let $\mb{E}=[E_1\rra E_0]$ be a  $\mb{G}=[G_1\rra G_0]$-bundle over the Lie groupoid $\mb{X}=[X_1\rra X_0].$ Then we have a short exact sequence 
	\begin{equation}\label{E:Atiyahgpd}
		\begin{tikzcd}
			0 \arrow[r, ""] & {\rm Ad}(\mb{E})    \arrow[r, "j^{/\mb{G}}"] & {\rm At}(\mb{E})  \arrow[r, "\pi_{*}^{/\mb{G}}"] & T \mb{X} \arrow[r, ""] & 0
		\end{tikzcd}
	\end{equation}
	of VB-groupoids  over $\mb{X}=[X_1\rra X_0].$	
\end{proposition} 
We call the short exact sequence  above the \textit{Atiyah sequence} associated to the  $[G_1\rra G_0]$-bundle  $[E_1\rra E_0]\ra [X_1\rra X_0]$.

\section{Connection on principal $2$-bundles over Lie groupoids} \label{Connection on principal $2$-bundles over Lie groupoids}
We have observed that in a traditional principal bundle setup, a splitting of the associated Atiyah sequence gives a connection on the principal bundle. We can adapt the same idea for the framework of the Atiyah sequence of a principal $2$-bundle  (Diagram \ref{E:Atiyahgpd}). We will see that the categorical structure involved in a principal $2$-bundle offers a more enriched differential geometric connection structure.

\subsection{Connection as a splitting of the Atiyah sequence}
\begin{definition}[strict connection and semi strict connection] \label{strict and semi-strict connection definition}
	Let $\mb{G}$ be a Lie $2$-group. Let  $\pi\colon \mb{E} \rightarrow \mb{X}$ be 
	a $\mb{G}$-bundle over a Lie groupoid $\mb{X}$, and 	
	
	\begin{equation}\label{Dia:Obj-MorAtiyahsplitting}
		\begin{tikzcd}
			0 \arrow[r, ""] & {\rm Ad}(\mb{E})    \arrow[r, "j^{/\mb{G}}"] & {\rm At}(\mb{E})  \arrow[r, "\pi_{*}^{/\mb{G}}"] & T \mb{X} \arrow[r, ""] & 0
		\end{tikzcd}
	\end{equation}
	its associated Atiyah sequence. A morphism $R: {\rm At}(\mb{E}) \ra {\rm Ad}(\mb{E})$ of VB-groupoids is said to be a \textit{strict connection} if
	$$ R \circ j^{/\mb{G}}=1_{{\rm Ad}(\mb{E})}.$$
	
	A morphism $R\colon  {\rm At}(\mb{E}) \ra {\rm Ad}(\mb{E})$ of VB-groupoids is said to be a  \textit{ semi-strict connection}  if  $R \circ j^{/\mb{G}}$ is  $2$-isomorphic to $1_{{\rm Ad}(\mb{E})},$
	$$R \circ j^{/\mb{G}}\simeq 1_{{\rm Ad}(\mb{E})}$$
	in the category 2-$\rm{VBGpd}(\mb{X})$  (\Cref{Prop:naturaltransformationinVBgroupoids}).
\end{definition}

\begin{definition}[category of strict and semi-strict connections] \label{category of strict connection}
	Let $\mb{G}$ be a Lie $2$-group, and $\pi\colon \mb{E} \rightarrow \mb{X}$ a  
	$\mb{G}$-bundle over a Lie groupoid $\mb{X}$. We define \textit{the category of  strict (resp. semi strict) connections  for  $\mb{E}(\mb{X},\mb{G})$}  as a category whose objects are  $R:{\rm At}(\mb{E}) \ra {\rm Ad}(\mb{E})$ such that $R$ is a strict (resp. semi strict) connections on $\mb{E}(\mb{X},\mb{G})$ and 
	morphisms are  $2$-morphisms  $\eta\colon R \Longrightarrow R'$ of 2-$\rm{VBGpd}(\mb{X}).$	
	We will denote the   categories of strict and semi strict connections  respectively by the notations $C_{\mb{E}}^{\rm{strict}}$ and $C_{\mb{E}}^{\rm{semi}}.$
\end{definition}

It is evident that  a strict connection gives a functorial section  $\begin{tikzcd}
	T \mb{X} 	   \arrow[r, "\Sigma"] & \At(\mb{E})		\end{tikzcd},$	
with respect 	to the VB-groupoid morphism $\begin{tikzcd}
	{\rm At}(\mb{E})  \arrow[r, "\pi_{*}^{/\mb{G}}"] & T \mb{X} 
\end{tikzcd},$
splitting the tangent bundles $TE_i\ra E_i, i\in\{0, 1\}$ into horizontal and vertical subbundles.  On the other hand for a semi-strict connection, the natural isomorphism $R \circ j^{/\mb{G}}\Lrrw1_{{\rm Ad}(\mb{E})}$	gives an obstruction for the corresponding VB-groupoid morphism  $\begin{tikzcd}
	T \mb{X} 	   \arrow[r, "\Sigma"] & {\rm At}(\mb{E})	\end{tikzcd}$ to be a section of $\begin{tikzcd}
	{\rm At}(\mb{E})  \arrow[r, "\pi_{*}^{/\mb{G}}"] & T \mb{X}. 
\end{tikzcd}$
The issue is related to the parallel transport with respect to a connection, which we will pursue in a subsequent paper.

Before providing a description of the connections on a principal $2$-bundle in terms of differential forms, we note an immediate consequence of our definition. 

\begin{lemma} \label{VB-groupoid=Lie groupoid}
	Let $\pi\colon \bigl(\mb{E} =[E_1\rra E_0]\bigr)\rightarrow \bigl(\mb{X}=[X_1\rra X_0]\bigr)$ be a $\mb{G}$-bundle over a Lie groupoid $\mb{X}$.  Let $R_0\colon {\rm At}(E_0) \rightarrow {\rm Ad}(E_0)$ and $R_1\colon {\rm At}(E_1) \rightarrow {\rm Ad}(E_1)$ be the splittings associated to be splittings of the respective Atiyah sequences. Then $R_0, R_1$ defines a strict connection on $\pi\colon \mb{E} \rightarrow \mb{X}$ if and only if $R_0, R_1$ define a morphism of  Lie groupoids from ${\rm At}(\mb{E})$ to ${\rm Ad}(\mb{E})$.
\end{lemma}
\begin{remark}
	For a given  $\mb{G}=[G_1\rra G_0]$-bundle $\mb{E}=[E_1\rra E_0]$ over $\mb{X}=[X_1\rra X_0],$ we have a simplicial principal $\mb{G}_\bullet=\{G_i\}_{i\geq 0}$ bundles $\mb{E}_{\bullet}=\{E_i\}_{i\geq 0}$ over ${\mb X}_\bullet=\{X_i\}_{i\geq 0},$ where $G_\bullet, E_\bullet, X_\bullet$ are the respective simplicial manifolds of groupoid nerves associated to Lie groupoids $\mb{G}=[G_1\rra G_0], \mb{E}=[E_1\rra E_0]$ and $\mb{X}=[X_1\rra X_0].$ Note each $G_i$ has a natural Lie group structure compatible with the composition.  
	Then one can further extend  Diagram \ref{Dia:Obj-MorAtiyah} to obtain a family of Atiyah sequences
	\begin{equation}\nonumber
		\begin{tikzcd}[sep=small]
			0 \arrow[r] & {{\rm Ad}}(E_i) \arrow[dd] \arrow[rr, "j_i^{/G_i}"] &  & {\rm {At}}(E_i) \arrow[rr, "(\pi_i)_*^{/G_i}"] \arrow[dd] &  & TX_i \arrow[r] \arrow[dd] & 0 \\
			&                                                                    &  &                                                       &  &                           &   \\
			0 \arrow[r] & X_i \arrow[rr]                                                     &  & X_i \arrow[rr]                                        &  & X_i \arrow[r]             & 0
		\end{tikzcd},
	\end{equation}
	for $i\geq 0,$ where $i$-th level is related to $i-1$-th level by the usual face and degeneracy  maps. In our follow-up paper we will introduce a Chern-Weil theory by incorporating the Atiyah sequences as above into our framework. 
\end{remark}

\subsection{Connections as  differential forms}
\begin{definition}
	\label{Definition:LGvaluedformOnLiegroupoid}
	Let $\mb{G}=[G_1\rra G_0]$ be a Lie $2$-group, and $[E_1\rra E_0]$ a Lie groupoid. An \textit{$L(\mb{G})$-valued $1$-form on $[E_1\rra E_0]$} is a morphism of Lie groupoids 
	$\omega:=(\omega_1,\omega_0)\colon T\mb{E} \rightarrow L(\mb{G})$ such that  $\omega_i$ is an $L(G_i)$-valued differential $1$-form on $E_i,$ for $i\in\{0, 1\}.$
	If $\mb{G}$ acts on $\mb{E}$ and $\omega\colon T\mb{E} \ra L(\mb{G})$ is  $\mb{G}$-equivariant, then $\omega$ called a \textit{$\mb{G}$-equivariant $1$-form}.
\end{definition}

The notion of a differential form 
on a Lie groupoid, called \textit{multiplicative forms}, is already available in literature \cite{bursztyn2009linear}. It would be illustrative  to express an $L(\mb{G})$ valued differential $1$-form  in terms of the crossed module to see the relation between our 
definition and the notion of multiplicative forms. For that, suppose $\ghta$ is the Lie crossed module associated to $\mb{G}.$ Then $G_0=G$ and $G_1=\hrtag$ and $L(G_0)=L(G),$ $L(G_1)=L(H)\oplus L(G).$ Let 
$\omega:=(\omega_1,\omega_0)\colon T\mb{E} \rightarrow L(\mb{G})$	is an $L(\mb{G})$-valued $1$-form.  
Let $\omega_{1 G}, \omega_{1 H}$ respectively be the $L(G)$ and $L(H)$-valued components of $\omega_1$. The functoriality condition on $\omega,$ can be expressed as:
\begin{equation}\label{E:Relmultipl}
	\begin{split}
		&\omega_{1 G}=s^{*}\omega_0,\\
		&t^{*}\omega_0=s^{*}\omega_0+\tau (\omega_{1 H}),\\
		&m^{*}\omega_{1 H}={\rm pr}_1^{*}\omega_{1 H}+{\rm pr}_2^{*}\omega_{1 H},
	\end{split}
\end{equation}
where $m\colon E_1\times_{s, E_0, t}E_1 \ra E_1$ is the composition map and $\pr_i\colon E_1\times_{s, E_0, t}E_1 \ra E_1, i\in \{1, 2\}$ are the projection maps to respectively first and second components. Precisely the third condition implies 
$\omega_{1 H}$ is an $L(H)$-valued multiplicative form on the groupoid $[E_1\rra E_0].$ In particular, if $G=e$ in the crossed module set-up, we recover the definition of a multiplicative form given in  \cite{bursztyn2009linear}.

We also refer to \cite{MR3894086} for a different, but related,  notion of an  $L(\mb{G})$-valued differential form on $\mb{E}=[E_1\rra E_0],$ where the author considers 
simplicial manifold ${\mb{E}}_{\bullet}$ defined by the nerve of the Lie groupoid $\mb{E},$ and defines an $L(\mb{G})$-valued differential form as a `suitably chosen' subcomplex of the double complex $T^p:=\bigoplus_{p=i+j+k}\Omega^i({E_j}, {\mathfrak G}_{k}),$ where ${\mathfrak G}_{-1}=L(G), {\mathfrak G}_0=L(H), {\mathfrak G}_i=0\, \forall\, i\neq -1\, \rm{or}\, 0.$ While the choice of the subcomplex
eluded above is in fact motivated by the connection structure on a  principal $2$-bundle over a manifold,
 motivation for our definition of $L(\mb{G})$-valued differential forms is to find an infinitesimal representation of the strict and semi-strict connection arising out of splitting of the Atiyah sequence. Note further if we attach an $L(H)$-valued differential $2$-form on $E_0$ with our differential form $\omega$ in 
\Cref{Definition:LGvaluedformOnLiegroupoid} we obtain a differential $1$-form as per 
the definition in \cite{MR3894086}. 

\begin{example}\label{E:Hemultiplicative}
	Let $\mb{E}=[E_1\rra E_0]$ be a Lie groupoid. Suppose the Lie $2$-group $\widehat H:=[H\rra {e}]$, associated to a Lie group $H,$ acts on $\mb{E}.$   Then an $L(\widehat H)$-valued 
	$1$-form on $\mb{E}$ is  same as an $L(H)$-valued multiplicative $1$-form on $\mb{E}.$ If the $L(H)$-valued multiplicative $1$-form on $\mb{E}$ is $H$-equivariant, then so is the corresponding $L(\widehat H)$-valued $1$-form 
	on $\mb{E}.$
\end{example}
\begin{example}
	If a  Lie group $G$ acts on a smooth manifold $M,$ then an $L(G)$-valued  $1$-form is same as an $L(G\rra G)$-valued $1$-form on the Lie groupoid $[M\rra M].$ If the form on $M$ is $G$-equivariant, then the corresponding 
	$L(G\rra G)$-valued $1$-form on the Lie groupoid $[M\rra M]$ is $[G\rra G]$ invariant.
\end{example}
\begin{example}\label{Ex:LGGvalued}
	If  the Lie $2$-group $[G\rra G]$ acts on a Lie  groupoid $\mb{E}=[E_1\rra E_0].$ Then an $L(\mb{G})$-valued $1$-form on the Lie groupoid $[E_1\rra E_0]$
	is an $L(G)$-valued $1$-form on $E_0$ such that $t^*\omega=s^*\omega.$ The equivariancy of one implies the equivariancy of the other.
\end{example}
\begin{example} Let $\ghta$ be a Lie crossed module associated to a Lie $2$-group $\mb{G}$. Let the Lie $2$-group $[G\rra G]$ acts on a Lie  groupoid $\mb{E}=[E_1\rra E_0].$ We can define another Lie  groupoid $\mb{E}\times H:=[E_1\times H\rra E_0]$ with source, target and composition maps  are respectively given by  $\tilde{s}(\gamma, h)=s(\gamma), \tilde{t}(\gamma, h) = t(\gamma) \tau(h^{-1})$ and 
	$(\gamma_2, h_2)\circ (\gamma_1, h_1)=\bigl((\gamma_2\tau(h_1)) \circ \gamma_1, h_2h_1\bigr).$ The unit and inverse maps respectively  are $1_p=(1_p, e)$ and  $\bigl(\gamma, h)^{-1}=(\gamma^{-1} \tau(h^{-1}), h^{-1}\bigr).$ 
	Moreover $\mb{G}$ has an action on  $\mb{E}\times H:=[E_1\times H\rra E_0]$  given by $(p, g)\mapsto p g$ for the objects and $(\gamma, h)(h', g')=(\gamma g,  \alpha_{g^{-1}}(h'^{-1}\, h))$ for the morphisms. Let $\omega$ be an 
	$L(G)$-valued $G$-equivariant $1$-form on $E_0$ such that $s^*\omega=t^*\omega.$ Then 
	$\widetilde \omega$,  defined as $${\widetilde \omega}_{\gamma,\, h}(X, \mathfrak K):=\ad_{(h, e)}\omega(s_{*\gamma}(X))-\mathfrak{K}\cdot h^{-1},$$ and $\omega$ gives an $L(\mb{G})$-valued $\mb{G}$-equivariant $1$-form on  
	Lie  groupoid $\mb{E}\times H:=[E_1\times H\rra E_0].$   In the next subsection we will see this example follows from a more general construction. 
\end{example}

Our next goal is to express the connections (strict and semi-strict) on a principal
$2$-bundle as differential forms. The following definition is natural.
\begin{definition}[Category of $\mb{G}$-equivariant $L(\mb{G})$-valued $1$-forms]\label{Def:Equivardiffcateg}
	Let $\mb{G}=[G_1\rra G_0]$ be a Lie 2 group acting on a Lie groupoid $\mb{E}=[E_1\rra E_0]$.  The \textit{ category of $\mb{G}$-equivariant $L(\mb{G})$-valued $1$-forms} on $\mb{E}$, denoted $\Omega_{\mb{E}^{\mb{G}}},$ is defined 
	as follows. 
	\begin{itemize}
		\item objects are $\mb{G}$-equivariant $1$-forms on $\mb{E}$, 
		\item morphisms are smooth natural isomorphisms  $\eta\colon \omega \Longrightarrow \omega'$  such that $\eta\bigl((p, v)\cdot g\bigr)=\eta(p, v)\cdot 1_g$ for all $(p, v) \in TE_0,g \in G_0$, and $\eta:TE_0\ra L(G_1)$ is an $L(G_1)$-valued $1$-form on $E_0$. 
	\end{itemize}
\end{definition}
Observe that a morphism of VB-groupoids $S\colon {\rm At}(\mb{E})\ra {\rm Ad}(\mb{E})$
\begin{equation}\nonumber
	\begin{tikzcd}[sep=small]
		{\rm At}(\mb{E}) \arrow[rr, "S"] \arrow[dd,xshift=0.75ex,"\pi"]&  & \Ad({\mb{E}}) \arrow[dd,xshift=0.75ex,"\pi"] \\
		&  &                \\
		\mb{X} \arrow[rr]            &  & \mb{X}          
	\end{tikzcd},
\end{equation}
is of the form
\begin{equation}\nonumber
	\begin{split}
		&[(p, v)]\mapsto [(p, \omega(p, v))],  \forall [p, v]\in {\rm At}(E_0),\\
		& [(\widetilde \gamma, \widetilde X)]\mapsto [(\widetilde \gamma, \omega(\widetilde \gamma, \widetilde X))], \forall [(\widetilde \gamma, \widetilde X)]\in  {\rm At}(E_1)
	\end{split}
\end{equation}
and in turn defines a $\mb{G}$-equivariant $L(\mb{G})$-valued $1$-form  $\omega\colon T\mb{E} \rightarrow L(\mb{G}).$ Obviously converse also holds.

\begin{proposition}\label{Prop:Corresconndiffform}
	Let  $\pi\colon \mb{E} \rightarrow \mb{X}$ be a  $\mb{G}$-bundle over a Lie groupoid $\mb{X}.$	 Let   $R \colon {\rm At}(\mb{E}) \ra {\rm Ad}(\mb{E})$ be a morphism of VB-groupoids and $\omega\colon T\mb{E} \rightarrow L(\mb{G})$ the associated  $\mb{G}$-equivariant $L(\mb{G})$-valued differential $1$-form  on the Lie groupoid $\mb{E}.$  
	
	\begin{enumerate}
		\item  $R \colon {\rm At}(\mb{E}) \ra {\rm Ad}(\mb{E})$ is  a strict connection as in \Cref{strict and semi-strict connection definition} if and only if  the following diagram  of morphisms of Lie groupoids
		\begin{equation} \label{Dia:condver}
			\begin{tikzcd}
				T{\mb{E}} \arrow[r, "\omega"]                  & L(\mb{G}) \\
				\mb{E} \times L(\mb{G}) \arrow[u, "\delta"] \arrow[ru, "\rm{pr_2}"'] &  
			\end{tikzcd},
		\end{equation}
		commutes up to nose. 
		
		\item $R \colon {\rm At}(\mb{E}) \ra {\rm Ad}(\mb{E})$ is  a semi-strict connection if and only if 
		the  Diagram~\ref{Dia:condver}  commutes up to a  $\mb{G}$-equivariant, fiber-wise linear natural isomorphism.
	\end{enumerate}
	\begin{proof}
		\begin{enumerate}
			\item If $R$ is a strict connection then $R \circ j^{/\mb{G}}=1_{{\rm Ad}(\mb{E})}.$  That means the corresponding $\mb{G}$-invariant $L(\mb{G})$ valued differential form $(\omega_1, \omega_0)=\omega\colon T\mb{E} \rightarrow L(\mb{G}) $ define a pair of classical connections on $G_0$-bundle $E_0\ra X_0$ and $G_1$-bundle $E_1\ra X_1$ respectively. Hence
			$\omega_0, \omega_1$ satisfy
			\[\omega_i\bigl(x_i, \delta_{x_i}(A_i)\bigr)=A_i, \qquad i\in \{0, 1\}
			\]  
			for all $x_i\in E_i, A_i\in L(G_i),$ giving the commutation relation in Diagram  \ref{Dia:condver}. The converse also easily follows. 
			\item  Now suppose $R$ is a semi strict connection. Let  $\epsilon\colon R \circ j^{/\mb{G}}\Lrrw 1_{{\rm Ad}(\mb{E})}$ be a $2$ isomorphism in $2$-$\rm VBGpd(\mb{X})$: 
			\begin{equation}\nonumber
				\begin{split}
					&\epsilon_{([p, B])}\colon R([p, \delta_p(B)])\longrightarrow [(p, B)].
				\end{split}
			\end{equation}
			
			We claim  that, since  $\pi(\epsilon([p, B])=1_{\pi(p)},$ we have a unique $\bigl(\omega(p, {\delta_p(B)})\xrightarrow{\kappa_{(p, B)}}B\bigr)\in L(G_1)$ such that $\eta([p, B])=[(1_p, \kappa_{(p, B)})],$ for a chosen $p\in E_0, B\in L(G_0).$ 
			To see this note that if $\epsilon([p, B])=[(\gamma_{p, B}, \alpha_{p, B})],$ for $\gamma_{p, B}\in E_1, \alpha_{p, B}\in L(G_1),$ then $[s(\gamma_{p, B})]=[p]. \Rightarrow \exists! g\in G_0$ such that $s(\gamma_{p, B})g=p.$ That means 
			$\epsilon([p, B])=[(\gamma_{p, B}1_g, \ad_{{1_g}^{-1}}(\alpha_{p, B}))],$ with $\gamma_{p, B}1_g$ starts at $p$ and projects down to $1_{\pi(p)}=\pi(1_p)$ in the $G_1$-bundle $E_1\ra X_1.$ Hence choosing the unique element   of
			$G_1$ to translate $\gamma_{p, B}1_g$ to $1_p,$ we obtain the unique $\biggl(\omega(p, {\delta_p(B)})\xrightarrow{\kappa_{(p, B)}}B\biggr)\in L(G_1)$ such that $\eta([p, B])=[(1_p, \kappa_{(p, B)})].$
			It is readily checked that $\kappa\colon \omega\circ \delta\Lrrw \pr_2$ is a natural isomorphism. Moreover smoothness of $\epsilon$ implies smoothness $\kappa.$ 
			If we represent the equivalence class $[(p, B)]$ by some other element $(p', B')=\bigl(p g, \ad_{g^{-1}}(B)\bigr),$ we end up with  $\epsilon([p', B'])=[(1_{p'}, \kappa_{(p', B')})].$ Implying $(1_{p'}, \kappa_{(p', B')})\sim (1_p, \kappa_{(p, B)}).$ 
			But that means $\bigl(1_{p'}, \kappa_{(p', B')}\bigr)=\bigl(1_p, \kappa_{(p, B)}\bigr)\cdot 1_g$ [this follows from the fact that $p'=p g$ and freeness of action on the fiber of $E_1\ra X_1$]. Which gives us the equivariancy of $\kappa:$
			$$\kappa\bigl(p g, \ad_{g^{-1}}(B)\bigr)=\ad_{{1_g}^{-1}}\bigl(\kappa(p, B)\bigr).$$
			
			Conversely given a $\mb{G}$-equivariant, fiber-wise linear  natural transformation $\kappa\colon \omega\circ \delta\Lrrw \pr_2,$ we define $\epsilon\colon R \circ j^{/\mb{G}}\Lrrw 1_{{\rm Ad}(\mb{E})}$ by  $\epsilon([p, B]):=[(1_p, \kappa_{(p, B)})].$ 
			The map is well defined for the following reason. If $(p', B')\in [(p, A)],$ then $p'=pg, B'=\ad_{g}(B)$ for some $g\in G.$ Then equivariancy of $\kappa$ implies $(1_{p'}, \kappa_{(p', B')})=(1_p, \kappa_{(p, B)})1_g.$ Thus 
			$[(1_{p'}, \kappa_{(p', B')})]=[(1_p, \kappa_{(p, B)})].$ $\epsilon$ satisfies $\pi(\epsilon_{[p, B]})=1_{\pi(p)}.$
			The map is obviously smooth. Verification that $\eta$ is a natural transformation and a morphism in $2$-$\rm VBGpd(\mb{X})$ is straightforward.
		\end{enumerate}
	\end{proof}
\end{proposition}

Let $\pi\colon \mb{E} \rightarrow \mb{X}$ be a $\mb{G}$-bundle over a Lie groupoid $\mb{X}.$ A  $\mb{G}$-equivariant $L(\mb{G})$-valued $1$-form  $\omega\colon T\mb{E} \rightarrow L(\mb{G})$  
satisfying the property  (1) in \Cref{Prop:Corresconndiffform} will be called a \textit{strict connection $1$-form}, whereas a  $\mb{G}$-equivariant $L(\mb{G})$-valued $1$-form  $\omega\colon T\mb{E} \rightarrow L(\mb{G})$  
satisfying the property  (2) in \Cref{Prop:Corresconndiffform} will be called a \textit{semi-strict connection $1$-form}. 

\begin{definition} [Category of strict and semi-strict connection $1$-forms]\label{Def:Semisemistriccat}
	Let $\pi\colon \mb{E} \rightarrow \mb{X}$ be a  $\mb{G}$-bundle over a Lie groupoid $\mb{X}.$ The \textit{category of strict and semi strict connections}  have objects respectively the strict connection $1$-forms and semi-strict 
	connection $1$-forms and have  morphisms the smooth natural isomorphisms  $\eta\colon \omega \Longrightarrow \omega'$  such that $\eta\bigl((p, v)\cdot g\bigr)=\eta(p, v)\cdot 1_g$ for all $(p, v) \in TE_0$ and $g \in G_0.$
	The category of strict and semi-strict connections 
	will be denoted respectively as $\Omega_\mb{E}^{\rm{strict}}$ and $\Omega_\mb{E}^{\rm{semi}}.$
\end{definition}
Now it is obvious that we have a sequence of categories
$$\Omega_\mb{E}^{\rm{strict}} \subset \Omega_\mb{E}^{\rm{semi}}\subset \Omega_{\mb{E}^{\mb{G}}}.$$

Let us further examine the natural isomorphism $\kappa\colon \omega\circ \delta\Lrrw \pr_2$ of the Diagram \ref{Dia:condver}
for a semi-strict connection $\omega\colon T\mb{E}\ra L(\mb{G}).$  We express $\mb{G}$ by its crossed module $\ghta.$
Then for each $(p, B)\in E_0\times L(G_0)$ we have a ${\kappa{(p, B)}}$  of the form $\kappa(p, B)=(\kappa_{\omega}(p, B), \omega\bigl(p, {\delta_p(B)}\bigr),$ where $\kappa_{\omega}(p, B)\in L(H)$ and 
\begin{equation}\label{semi-strict connection: condition on omega_0}
	\omega\bigl(p, {\delta_p(B)}\bigr)-B=-\tau\bigl(\kappa_{\omega}(p, B)\bigr)
\end{equation}
Thus $\kappa_{\omega}(p, B)$ measures the deviation  $\omega\bigl(p, {\delta_p(B)}\bigr)-B.$ In turn given a $\kappa$ we have a smooth fiber-wise linear map $\kappa_{\omega}\colon E_0\times L(G)\ra L(H)$. The equivariancy of $\kappa$ implies 
\begin{equation}\label{E:equilambda}
	\kappa_{\omega}\bigl(p g, \ad_{g^{-1}} B\bigr)=\alpha_{g^{-1}}\bigl(\kappa_{\omega}(p, B)\bigr).
\end{equation}
Since $\kappa$ is a natural transformation, for $\bigl(A\xrightarrow{K}B\bigr)\in L(G_1)$  and $\bigl(p\xrightarrow{\widetilde \gamma}q\bigr)\in L(E_1)$
the following diagram commutes,
\[
\begin{tikzcd}[sep=small]
	\omega_{0, p}\bigl(\delta_p(A)\bigr) \arrow[dd, "{\kappa(p, A)}"'] \arrow[rrr, "\omega_{1, \widetilde \gamma}\bigl((\delta_{\widetilde \gamma})(K) \bigr)"] &  &  & \omega_{0, p}\bigl(\delta_q(B)\bigr)  \arrow[dd, "{\kappa(q, B)}"] \\
	&  &  &                                                 \\
	A \arrow[rrr, "K"]                                                                            &  &  & B                                              
\end{tikzcd},\]
and we arrive at the condition
\begin{equation} \label{semi-strict connection: condition on omega_1}		 
	\omega_{1, \widga}\bigl(\delta_{\widga}(K)\bigr)-K=\bigl(\kappa_{\omega}(p, A)-\kappa_{\omega}(q, B),-\tau(\kappa_{\omega}(p, A))\bigr).
\end{equation}
Consider the Lie $2$-sub algebra $L(\mb{G})^\tau:=[L(H)\rtimes \tau\bigl(L(H)\bigr)\rra \tau\bigl(L(H)\bigr)$] of $L(\mb{G})=[L(H)\rtimes L(G)\rra L(G)].$ Observe that $L(\mb{G})^\tau$ is invariant under the adjoint action of $\mb{G}.$

Observe that \Cref{semi-strict connection: condition on omega_1} gives us the following uniqueness property.
\begin{lemma}
	Let $\omega$ be a semi strict connection $1$-form on the $\mb{G}=\ghta$-bundle $\mb{E}=[E_1\rra E_0]$ over $\mb{X}=[X_1\rra X_0].$ Then the natural transformation $\kappa$ in 
	Diagram \ref{Dia:condver} is unique.
\end{lemma}

Now we conclude the following. 
\begin{proposition}\label{Prop:Semisirictnatural}
	Let $\mb{E}=[E_1\rra E_0]$ be a $\mb{G}=\ghta$-bundle over $\mb{X}=[X_1\rra X_0].$  A $\mb{G}$-equivariant $1$-form $\omega\colon T\mb{E}\ra L(\mb{G})$ is a semi-strict connection $1$-form if and only if the functor 
	$(\omega\circ \delta-\pr_2):={\widehat \kappa}_{\omega}$  is of the form $\widehat{\kappa}_{\omega}(p, B)=\tau(\kappa_\omega(p,B))$ and $\widehat{\kappa}_{\omega}(\gamma, K)=\bigl(\kappa_{\omega}(p, A)-\kappa_{\omega}(q, B),-\tau(\kappa_{\omega}(p, A))\bigr)$ for a smooth, fiber wise linear map $\kappa_{\omega}\colon E_0 \times L(\mb{G}) \ra L(H)$ satisfying \Cref{E:equilambda}.	 In particular, if $\omega\circ \delta-\pr_2$ is the zero functor then we have a strict connection $1$-form.
\end{proposition}
\begin{corollary}\label{cor:stricconnto semi}
	Let $\omega=(\omega_0,\, \omega_1)\colon T\mb{E}\rra L(\mb{G})$ be a strict connection $1$-form on the $\mb{G}=\ghta$-bundle $\mb{E}=[E_1\rra E_0]$ over $\mb{X}=[X_1\rra X_0].$ Let $\lambda\colon T E_0\ra L(H)$
	an $L(H)$-valued $1$-form on $E_0$ satisfying the equivariance property
	$\lambda(p g, v\cdot g)=\alpha_{g^{-1}}\bigl(\lambda(p, v)\bigr).$
	Then ${\widetilde \omega}_0=\omega_0 + \tau (\lambda)$ and ${\widetilde \omega}_1=\omega_1 + \tau (s^{*}\lambda)+ t^{*}\lambda-s^{*}\lambda$ define a semi strict connection $\widetilde \omega=({\widetilde \omega}_0, {\widetilde \omega}_1)\colon T\mb{E}\ra L(\mb{G})$.
	\begin{proof}
		The functoriality of $({\widetilde \omega}_0, {\widetilde \omega}_1)$ is immediate. Since $\omega_{0 p} (\delta_p(A))=A,$ for $A\in L(G_0),$ only thing remained to be verified   is the $\mb{G}$-equivariancy of the functor 
		\begin{equation}
			\begin{split}
				&\Lambda\colon T{\mb{E}}\ra L(\mb{G}),\\
				&(p, v)\mapsto \tau\bigl(\lambda (p, v)\bigr),\\
				&(\widetilde \gamma, \widetilde X)\mapsto  \tau \bigl(\lambda(s(\widetilde \gamma), s_{*, \widetilde \gamma}(X) )\bigr)+\lambda\bigl(t(\widetilde \gamma), t_{*, \widetilde \gamma}(X)\bigr) -\lambda\bigl(s(\widetilde \gamma), s_{*, \widetilde \gamma}(X)\bigr).
			\end{split}
		\end{equation}
		Then $\Lambda\bigl((p, v)g\bigr)=\tau\bigl(\lambda(p g, v\cdot g)\bigr)=\ad_{g^{-1}}\bigl(\Lambda(p, v)\bigr).$ To see the equivariancy at the morphism level, consider $s(\widetilde \gamma, \widetilde X)=(p, v), t(\widetilde \gamma, \widetilde X)=(q, w).$ Then 
		\begin{equation}\nonumber
			\begin{split}
				&\Lambda\bigl((\widetilde \gamma, \widetilde X) (h, g)\bigr)\\
				&=\tau \bigl(\lambda(p g, v g)\bigr)-\lambda(p g, v g)+\lambda(q \tau(h) g, w \tau(h) g)\\
				&=\ad_{(e, g)^{-1}}\tau \bigl(\lambda(p , v)\bigr)-\alpha_{g^{-1}}\bigl(\lambda(p, v)\bigr)+\alpha_{g^{-1}\tau(h^{-1})}\bigl(\lambda(q, w)\bigr)\\
				&=\ad_{(e, g)^{-1}}\tau \bigl(\lambda(p , v)\bigr)-\alpha_{g^{-1}}\bigl(\lambda(p, v)\bigr)+\underbrace{\alpha_{g^{-1}}\bigl(\ad_{h^{-1}}(\lambda(q, w))\bigr)}_{\ad_{(h, g)^{-1}}(\lambda(q, w))}\\
				&=\underbrace{\ad_{(h, g)^{-1}}\tau \bigl(\lambda(p , v)\bigr)-\alpha_{g^{-1}}(h^{-1})\,\cdot \bigl({\bar \alpha}_{\alpha_{g^{-1}}(h)}{({\rm ad}_{g^{-1}}(\tau(\lambda(p, v))))}\bigr)}_{\ad_{(e, g)^{-1}}\tau \bigl(\lambda(p , v)\bigr) [\textit {using formulae in \Cref{E:Adjonalgebras}}] }\\
				&\underbrace{-\ad_{(h, g)^{-1}}\bigl((\lambda(p , v) \bigr)+\bigl(\alpha_{g^{-1}}(\ad_{h^{-1}}(\lambda(p, v)))-\alpha_{g^{-1}}\bigl(\lambda(p, v)\bigr) \bigr)}_{-\alpha_{g^{-1}}\bigl(\lambda(p, v)\bigr) [\textit{using formulae in \Cref{E:Adjonalgebras}}]}+\ad_{(h, g)^{-1}}(\lambda(q, w))\\
				&=\ad_{(h, g)^{-1}}\tau \bigl(\lambda(p , v)\bigr)-\ad_{(h, g)^{-1}}\bigl((\lambda(p , v) \bigr)+\ad_{(h, g)^{-1}}(\lambda(q, w))\\
				&\underbrace{+[-\alpha_{g^{-1}}(h^{-1})\,\cdot \bigl({\bar \alpha}_{\alpha_{g^{-1}}(h)}   {({\ad}_{g^{-1}}(\tau(\lambda(p, v))))\bigr)}
					+\bigl(\alpha_{g^{-1}}(\ad_{h^{-1}}(\lambda(p, v)))   - \alpha_{g^{-1}}\bigl(\lambda(p, v)\bigr) \bigr)]}_{\textit{vanishes by (2) of Lemma~\ref{Lemma:Identityfor later}}}\\
				&=\ad_{(h, g)^{-1}}\tau \bigl(\lambda(p , v)\bigr)-\ad_{(h, g)^{-1}}\bigl((\lambda(p , v) \bigr)+\ad_{(h, g)^{-1}}(\lambda(q, w)).
			\end{split}
		\end{equation}
	\end{proof}	
\end{corollary}
The \Cref{cor:stricconnto semi} gives a very straightforward way of finding a semi-strict connection from a strict connection.
\begin{example}
	Let $P\ra M$ be a traditional principal $G$-bundle over the smooth manifold $M.$  A connection  $\omega$ on $P\ra M$ defines a strict connection $(\omega, \omega)$ on principal Lie $2$-group $[G\rra G]$-bundle  
	$[P\rra P]\ra [M\rra M]$ over  the Lie groupoid $[M\rra M].$ A strict connection is same as a semi strict connection for such a Lie $2$-group bundle.
\end{example}
\begin{example}\label{Ex:GGbundleconnection}
	Let $[E_1\rra E_0]$ be a principal $[G\rra G]$-bundle over $[X_1\rra X_0].$ Then as noted in \Cref{E:Example of principal 2-bundle}, $E_0\ra X_0$ is a principal $G$-bundle over $[X_1\rra X_0]$ and $[E_1\rra E_0]=[s^*E_0\rra E_0].$ Then 
	a strict connection $\omega$ on $[G\rra G]$-bundle $[E_1\rra E_0]\ra [X_1\rra X_0]$ is a connection on the principal $G$-bundle $E_0\ra X_0$ such that $s^*\omega=t^*\omega.$ A strict connection is same as a semi strict connection for such a Lie $2$-group bundle.\end{example}
\begin{remark}
	In \Cref{Ex:GGbundleconnection} we recover the definition of a connection on a principal $G$-bundle over a Lie groupoid $[X_1\rra X_0]$ as in \cite{MR2270285}.
\end{remark}
\begin{example}\label{E:Example of principal 2-bundle ordinary conn}
	Consider the principal  $[H\rra e]$-bundle   $[E_1\rra X_0]$ over $[X_1\rra X_0]$ of \Cref{E:Example of principal 2-bundle ordinary}. Then a strict connection is  a classical connection $\omega$ on the principal $H$-bundle $[E_1\rra X_1]$ such that $\omega$ is a multiplicative $1$-form on the Lie groupoid $[E_1\rra X_0]$. On the other hand any $L(H)$-valued $H$-equivariant multiplicative form on $[E_1\rra X_1]$ is a semi strict connection (see also \Cref{E:Hemultiplicative}).
\end{example}

\begin{example}\label{E:Exampleprincipalpairlie2conn }
	Given a connection $\omega_0$ on a principal $G$-bundle $E_0\ra X_0,$ we define a connection on the $[G\times G\rra G]$-bundle $[E_1\rra E_0]$ over $\mb{X}=[X_1\rra X_0]$ in \Cref{E:Exampleprincipalpairlie2} by 
	$$\omega_{1 (p, \gamma, q)}(v_1, X, v_2)=\bigl(\omega_0(v_1), \omega_0(v_2)\bigr),$$
	for all $(v_1, X, v_2)\in T_{(p, \gamma, q)}E_1.$
\end{example}

In \Cref{SS:Conndeco}, we will prescribe a construction of connections on the decorated bundle.  

We have shown a bijective correspondence between strict  (resp. semi strict) connections and strict (resp. semi strict) connection $1$-forms for a principal $2$-bundle. We show here this correspondence extends to the respective categories as well. 

\begin{theorem} \label{strict connection=strict forms}
	Let $\mb{E}=[E_1\rra E_0]$ be a $\mb{G}=[G_1\rra G_0]$-bundle over a Lie groupoid $\mb{X}=[X_1\rra X_0]$.
	\begin{enumerate}
		\item The categories $C^{\rm{semi}}_{\mb{E}}$ and $\Omega_{\mb{E}}^{\rm{semi}}$ are isomorphic. 	
		\item The categories $C^{\rm{strict}}_{\mb{E}}$ and $\Omega_{\mb{E}}^{\rm{strict}}$ are isomorphic. 	
	\end{enumerate}	
	\begin{proof} 
		\begin{enumerate}
			\item The object level association has already been established in \Cref{Prop:Corresconndiffform}. Let $\ghta$ be the associated Lie cross module of $\mb{G}.$ Let $R, R'\colon {\rm At}(\mb{E})\ra \Ad(\mb{E})$ be a pair of strict connections and  $\eta\colon R\Lrrw R'$ a natural transformation such 
			that $\pi\bigl(\eta([p, v])\bigr)=1_{\pi(p)},$ for any $[(p, v)]\in {\rm At}(E_0)=TE_0/G_0.$ Let $\omega, \omega'\colon T\mb{E}\ra L(\mb{G})$ be respective semi-strict connection $1$-forms of $R$ and $R'.$ As noted earlier then $R[(p, v)]=[\bigl(p, \omega(p, v)\bigr)], R'[(p, v)]=[\bigl(p, \omega'(p, v)\bigr)].$ Let $\eta_{[p, v]}\colon [\bigl(p, \omega(p, v)\bigr)]\ra [\bigl(p, \omega'(p, v)\bigr)].$ We claim that there exists $\omega(p, v)\xrightarrow{\bar \eta(p, v)}\omega'(p, v)\in L(G_1)$ such that $[\bigl(1_p, \bar \eta (p, v)\bigr)]=\eta_{[p, v]}.$	To see this suppose $[(\gamma, K)]\in \eta_{[p, v]}.$ That means there exist $g, g'\in G_0$ such that $s(\gamma) g=p, t(\gamma) g'=p, \ad_{g^{-1}}\bigl(s(K)\bigr)=\omega(p, v)$ and  $\ad_{g'^{-1}}\bigl(t(K)\bigr)=\omega'(p, v).$ Then $(\gamma, K)\cdot 1_g=(\gamma 1_g, \ad_{{1_g}^{-1}}(K)):=(\gamma_p, K_{\omega})\in  \eta_{[p, v]}.$ Now as $\pi\bigl(\eta([p, v])\bigr)=1_{\pi(p}),$ 	we have $\pi(\gamma_p)=\pi(1_p).$ Thus both $\gamma_p, 1_p$ are elements of the same fiber on the $G_1$-bundle $E_1\ra X_1$ with the same source $p.$ That means there exists a unique $h\in H$ such that $\gamma_p (h, e)=1_p$.  Comparing the targets $t(\gamma_p)\tau(h)=t(\gamma)g \tau(h)=p=t(\gamma) g'.$ Which gives us $g'=g\tau(h)=g\tau(h)g^{-1}g=\tau(\alpha_g(h)) g.$ Thus 
			$(\gamma_p, K_{\omega})\cdot (h, e)=\bigl(1_p, \ad_{(h, e)^{-1}}(K_{\omega})\bigr)	\in    \eta_{[p, v]}.$	Now $\ad_{(h, e)^{-1}} K_{\omega}=\ad_{(h, e)^{-1}} \ad_{{1_g}^{-1}}(K)=\ad_{(h, e)^{-1}} \ad_{{(e, g)}^{-1}}(K)=\ad_{(h^{-1}, g^{-1})} K.$ Now using \Cref{E:Adjonalgebras} we see $s(\ad_{(h^{-1}, g^{-1})} K)=\ad_{g^{-1}}\bigl(s(K)\bigr)=\omega(p, v)$ and $t(\ad_{(h^{-1}, g^{-1})} K)=\ad_{\tau(h^{-1})g^{-1}}\bigl(t(K)\bigr)=\ad_{g'^{-1}}(t(K))=\omega'(p, v).$ We define $$\bar{\eta}(p, v):=\ad_{(h, e)^{-1}} K_{\omega}.$$
			It is a straightforward verification that $\bar \eta$ defines a smooth natural transformation $\omega\Lrrw \omega'.$ 
			Moreover if $(p', v')=(p, v) \theta$, for some $\theta \in G_0,$ then $(p',  v')\in [(p, v)].$ Then above construction gives $\bigl(1_{p'}, \bar \eta(p', v')\bigr)\in [\bigl(1_p, \bar \eta(p, v)\bigr)]$ and  
			$$\bigl(1_{p'}, \bar \eta(p', v')\bigr)=\bigl(1_p, \bar \eta(p, v)\bigr) 1_{\theta}=\bigl(1_{p \theta}, \ad_{1_{\theta}^{-1}}\bar \eta(p, v)\bigr).$$ 
			Hence we obtain the $\mb{G}$ equivariancy condition: $\ad_{1_{\theta}^{-1}}\bigl(\bar\eta(p, v)\bigr)=\bar \eta(p \theta , v \theta).$
			
			In the other direction,  if $\bar \eta\colon \omega\Lrrw \omega'$ is a smooth, $\mb{G}$-equivariant natural transformation between the semi-strict connection $1$-forms, we define a natural transformation between the corresponding semi-strict connections 
			$R, R'\colon {\rm At}(\mb{E})\longrightarrow \Ad{\mb{E}},$ by $\eta_{[p, v]}=[\bigl(1_p, \bar \eta(p, v)\bigr)].$ The $\mb{G}$-equivariancy implies the map is well defined. Now $\pi(\eta_{[(p, v)]})=\pi (1_{p})=1_{\pi(p)}.$ Thus we obtain a smooth natural transformation $\eta\colon R\Lrrw R'$ satisfying $\pi(\eta_{[(p, v)]})=1_{\pi(p)}.$
			
			It is readily checked that both the maps are functorial and inverses of each other. 
			
			\item The proof is almost identical to that of the case of the semi-strict connections.
		\end{enumerate}
	\end{proof}
\end{theorem}
Let  ${\bar C}_{\mb{E}}^{\rm{strict}},{\bar C}_{\mb{E}}^{\rm{semi}},{\bar \Omega}_{\mb{E}}^{\rm{semi}}$, and ${\bar \Omega}_{\mb{E}}^{\rm{strict}}$	be the respective collections of connected components of the categories  
${ C}_{\mb{E}}^{\rm{strict}},{C}_{\mb{E}}^{\rm{semi}},{ \Omega}_{\mb{E}}^{\rm{semi}}$, and ${ \Omega}_{\mb{E}}^{\rm{strict}}$.
Then the \Cref{strict connection=strict forms}	tells us the following.
\begin{corollary} \label{Cor:Setstrict connection=strict forms}
	Let $\mb{E}=[E_1\rra E_0]$ be a $\mb{G}=[G_1\rra G_0]$-bundle over a Lie groupoid $\mb{X}=[X_1\rra X_0].$
	\begin{enumerate}
		
		\item  There exists a bijection between ${\bar C}_{\mb{E}}^{\rm{strict}}$ and ${\bar \Omega}_{\mb{E}}^{\rm{strict}}$.
		
		\item There exists a bijection between ${\bar C}_{\mb{E}}^{\rm{semi}}$ and ${\bar \Omega}_{\mb{E}}^{\rm{semi}}$.	
	\end{enumerate}
\end{corollary}

\subsection{Connection on decorated principal 2-bundles}\label{SS:Conndeco}
Let $\ghta$ be a Lie crossed module and $\mb{G}=[G_1\rra G_0]$ the associated Lie $2$-group. Let $\mb{E}^{\rm dec} \ra \mb{X}$ be the decorated bundle associated to a principal $G$-bundle $\bigl(\pi\colon E_0 \rightarrow X_0, \mu, [X_1 \rightrightarrows X_0]\bigr)$
(\Cref{Prop:Decoliegpd}). First we 
compute the differentials of the various structure maps and the action maps associated to the Lie groupoid 
$\mb{E}^{\rm dec}=\bigl[(s^{*}E_0)^{\rm{dec}} \rightrightarrows E_0\bigr].$

The source, target and composition maps of the tangent Lie groupoid $\bigl[T(s^{*}E_0)^{\rm{dec}} \rightrightarrows TE_0\bigr]$
are respectively given by 
\begin{equation}\label{E:SoTaDiff}
	\begin{split}
		&{\widetilde s}_{* ((\gamma, p), h)}((X, v), \mathfrak K)=v,\\
		&{\widetilde t}_{* ((\gamma, p), h)}((X, v), \mathfrak K)=\mu_{* (\gamma, p)} (X, v)\cdot \tau(h^{-1})-{\delta_{\mu(\gamma, p)\tau({h}^{-1})}
			( \tau(\mathfrak K)\cdot \tau(h^{-1}))},\\
		&\bigl((\gamma_2, p_2, h_2), (X_2, v_2, {\mathfrak K}_2)\bigr)\circ \bigl((\gamma_1, p_1, h_1), (X_1, v_1, {\mathfrak K}_1)\bigr)\\
		&=\bigl((\gamma_2\circ \gamma_1, p_1, h_2 h_1), (X_2\circ X_1, v_1, h_2\cdot {\mathfrak K}_1+{\mathfrak K}_2\cdot h_1\bigr).
	\end{split}
\end{equation}
The differential of the target map $\widetilde t$ is computed by applying the chain-rule in the  following composition of maps
\[s^*E_0\times H\xra{({\rm Id},^{-1})}s^*E_0\times H\xra{(\mu,\tau)}E_0\times G\xra{} E_0.\]

Let us  compute the vertical vector field generating functor $\delta\colon \mb{E}^{\rm dec}\times L(\mb{G})\to T(\mb{E}^{\rm dec})$ with respect to the action defined in  \Cref{E:Actionondeco}.
We have
\begin{equation}\label{E:verdeco}
	\begin{split}
		\delta\colon \mb{E}^{\rm dec}\times L(\mb{G})&\ra T(\mb{E}^{\rm dec})\\
		(p, B) &\mapsto \delta_p(B)\\
		\bigl((\gamma, p, h) (A, B)\bigr)&\mapsto  \delta_p(B)-\bar{\alpha}_h(B)-A\cdot h.
	\end{split}
\end{equation}	

Note that, for a fixed $((\gamma, p), h)\in s^*E_0\times H,$ the map $\delta_{((\gamma, p), h)}\colon H\rtimes G\ra s^*E_0\times H$ is given by $(h',\, g)\mapsto \bigl(\gamma, pg, \alpha_{g^{-1}} (h'^{-1}h)\bigr)$. The first coordinate is the constant map, the second coordinate is the right translation map, whereas the third coordinate can be decomposed as 
\[H\times G\xra{^{-1},^{-1}}H\times G\xra{(-,-)}G\times H\xra{({\rm Id}, R_h)}G\times H\xra{\alpha}H.\] Then the formula in \ref{E:verdeco} is computed by applying the chain-rule.

The derivatives of the right action of 	$\mb{G}$ on $\mb{E}$ are given as, for fixed $g\in G, (h', g)\in G_1,$
\begin{equation}\label{E:Hordeco} 
	\begin{split}
		v&\mapsto v g\\
		(X, v, \mathfrak K)&\mapsto \bigl(X, v g, \alpha_{g^{-1}}(h'^{-1}\cdot \mathfrak K)\bigr),	
	\end{split}
\end{equation}		
for $v\in T_p E_0, \bigl((X, v), \mathfrak K\bigr)\in T_{((\gamma,\, p), h)}\bigl(s^{*}E_0^{\rm dec}\bigr).$
Here we have adopted the notations of \Cref{Remark:notations}.		

Suppose the principal $G$-bundle $\bigl(\pi\colon E_0 \rightarrow X_0, \mu, [X_1 \rightrightarrows X_0]\bigr)$ over $\mb{X}$ admits a connection $\omega,$ as defined in \cite{MR2270285}; that is $\omega$ is a connection on $G$-bundle $E_0 \rightarrow X_0$ satisfying $s^*\omega=t^*\omega,$ where $s, t\colon [s^*E_0\rra E_0$] respectively given by $s\colon (\gamma, p)\mapsto p,$ $t\colon (\gamma, p)\mapsto \mu(\gamma, p).$ The condition explicitly reads 
\begin{equation}\label{E:pullconcond}
	\omega_p(v)=\omega_{\mu(\gamma, p)}\bigl(\mu_{*, (\gamma, \, p)}(X, v)\bigr),
\end{equation}
for any $(\gamma, p)\in s^{*}E_0, (X, v)\in T_{(\gamma, p)}\bigl(s^* E_0\bigr)$.

Define an $L(H\rtimes G)$-valued differential form on $s^*E_0\times H$, \[\omega^{\rm dec}_{(\gamma, p,\, h)}(X, v, \mathfrak K)=\ad_{(h,\, e)}\bigl(\omega_p(v)\bigr)
-\mathfrak{K} \cdot h^{-1}.\] 
We will prove here that   $(\omega^{\rm dec}, \omega)$ is a strict connection $1$-form on $\mb{G}$-bundle $\mb{E}^{\rm dec}$ over $\mb{X}.$ Note that 
\begin{equation}\label{E:connectiongeneraal}
	\omega^{\rm dec}_{(\gamma,\, p,\, h)}(X, v, \mathfrak K)=\ad_{(h, e)}\bigl(\omega_p(v)\bigr)-\mathfrak{K}\cdot h^{-1},
\end{equation}
is same as 
$$\omega^{\rm dec}_{((\gamma,\, p), h)}(X, v, \mathfrak K)=\ad_{(h, e)}\bigl((s^*\omega)_{((\gamma, p), h)}((X, v), \mathfrak K)\bigr)-{\Theta}_h(\mathfrak K),$$  
where $\Theta$ is the Maurer-Cartan  form on $H,$ $\bigl((\gamma, p), h\bigr)\in s^* E_0\times H$ and $\bigl((X, v), \mathfrak K\bigr)\in T_{((\gamma,\, p), h)}(s^* E_0\times H)=T_{(\gamma, p)}s^{*}E_0\oplus T_hH.$

Let us verify that $(\omega^{\rm dec}, \omega)$ is a functor.
\begin{lemma}\label{lemma:Funcdecconn}
	$(\omega^{\rm dec}, \omega)$ define a functor $T\mb{E}^{\rm dec}\ra L(\mb{G}).$
	\begin{proof}
		Using \Cref{E:Adjonalgebras} we write
		\begin{equation}\nonumber
			\begin{split}
				&\omega^{\rm dec}_{(\gamma , p, h)}(X, v, \mathfrak K)\\
				&=\ad_{(h, e)}\bigl(\omega_p(v)\bigr)- \mathfrak{K}\cdot h^{-1}\\
				&=\underbrace{\omega_p(v)}_{\in L(G)}+\underbrace{h\cdot ({\bar \alpha}_{h^{-1}}({\omega_p(v)}))  -\mathfrak{K}\cdot h^{-1}}_{\in L(H)}.
			\end{split}
		\end{equation}
		
		Let $\bigl((\gamma_2, p_2, h_2), (X_2, v_2, {\mathfrak K}_2)\bigr)$ and  $\bigl((\gamma_1, p_1, h_1), (X_1, v_1, {\mathfrak K}_1)\bigr)$ be a pair of composable morphisms in $T{\mb{E}}^{\rm dec},$ which means, by \Cref{E:SoTaDiff}, in  particular
		\begin{equation}\label{E:sour-targetmatch}
			\begin{split}
				&p_2=\mu(\gamma_1, p_1)\tau(h_1)^{-1}\\
				&v_2=\mu_{* (\gamma_1, p_1)} (X_1, v_1)\cdot \tau(h_1^{-1})-\delta_{\mu(\gamma_1, p_1)\tau({h_1}^{-1})}(\tau({\mathfrak K}_1)\cdot\tau(h_1^{-1})).
			\end{split}
		\end{equation}
		Note \begin{equation}\label{E:sourcetargetexpli}
			\begin{split}
				&s\bigl(\omega^{\rm dec}_{(\gamma_2 , p_2, h_2)}(X_2, v_2, \mathfrak K_2)\bigr)=\omega_{p_2}(v_2),\\
				&t\bigl(\omega^{\rm dec}_{(\gamma_1 , p_1, h_1)}(X_1, v_1, \mathfrak K_1)\bigr)=\omega_{p_1}(v_1)+\tau (h_1^{-1}\cdot ({\bar \alpha}_{h_1}({\omega_{p_1}(v_1)}))
				-\mathfrak{K}_1\cdot h_{1}^{-1}).
			\end{split}
		\end{equation}
		Now plugging the equations of  \Cref{E:sour-targetmatch} into the equations above and noting that $\omega$ is a connection  on $G$-bundle $E_0\ra X_0$ satisfying $\omega_{(\mu(\gamma_1, p_1))}(\mu_{*}(X_1, v_1))=\omega_{p_1}(v_1)$ (see \Cref{E:pullconcond}), we obtain the required source-target consistency. 
		
		To verify the composition we separately compute 
		$$\bigl(\omega^{\rm dec}_{(\gamma_2 , p_2, h_2)}(X_2, v_2, \mathfrak K_2)\bigr)\circ \bigl(\omega^{\rm dec}_{(\gamma_1 , p_1, h_1)}(X_1, v_1, \mathfrak K_1)\bigr)$$
		and $$\omega^{\rm dec}_{(\gamma_2 , p_2, h_2)\circ (\gamma_1, p_1, h_1)}\bigl((X_2, v_2, \mathfrak K_2)\circ (X_1, v_1, \mathfrak K_1)\bigr),$$ using last of the equations in \Cref{E:SoTaDiff}. As before, we plug-in the relations  in \Cref{E:sourcetargetexpli} and use the properties of $\omega.$ After some calculation we derive,
		\begin{equation}\label{E:onecompo}
			\begin{split}
				&\bigl(\omega^{\rm dec}_{(\gamma_2 , p_2, h_2)}(X_2, v_2, \mathfrak K_2)\bigr)\circ \bigl(\omega^{\rm dec}_{(\gamma_1 , p_1, h_1)}(X_1, v_1, \mathfrak K_1)\bigr)\\
				=&\omega_{p_1}(v_1) 
				-\mathfrak{K}_1\cdot h_{1}^{-1}  -\mathfrak{K}_2\cdot h_{2}^{-1}\\
				+& \bigg[h_2\cdot \biggl({\bar \alpha}(h_2^{-1})({\rm ad}_{\tau(h_1)}(\omega_{p_1}(v_1)) )\biggr)+h_1\cdot \big({\bar \alpha}(h_1^{-1}) (\omega_{p_1}(v_1)) \big)\bigg]
			\end{split}
		\end{equation}
		and
		\begin{equation}\label{E:secondcompo}
			\begin{split}
				&\omega^{\rm dec}_{(\gamma_2 , p_2, h_2)\circ (\gamma_1, p_1, h_1)}\bigl((X_2, v_2, \mathfrak K_2)\circ (X_1, v_1, \mathfrak K_1)\bigr)\\
				=& \omega_{p_1}(v_1)-\mathfrak{K}_1\cdot h_{1}^{-1}  -\mathfrak{K}_2\cdot h_{2}^{-1}
				+ \bigg[h_2h_1 \cdot \big({\bar \alpha}(h_1^{-1}h_2^{-1}) (\omega_{p_1}(v_1))\big)\bigg].		
			\end{split}
		\end{equation}
		Then \Cref{Lemma:Identityfor later} implies the  terms inside [\, ] on both equations are identical. Since $\omega$ was multiplicative, we conclude $(\omega^{\rm dec}, \omega)$ is a functor $T\mb{E}^{\rm dec}\ra L(\mb{G}).$
\end{proof}	\end{lemma}	
\begin{lemma}\label{lemma:Connectioncdecconn}
	$\omega^{\rm dec}$ is a connection on $\hrtag=G_1$-bundle  $s^*E_0\times H\ra X_1.$
	\begin{proof}	
		Using \Cref{E:Hordeco,E:verdeco}  one directly verifies 
		$$\omega^{\rm dec}_{((\gamma,\, p),\, h)\cdot (h',\, g')}\bigl((X,\, v,\, \mathfrak K))\cdot (h',\, g')\bigr)=\ad_{(h'\, g')^{-1}}\bigl(\omega^{\rm dec}_{((\gamma,\, p),\, h)}\bigl((X,\, v),\, \mathfrak K\bigr)\bigr)$$
		and $$\omega^{\rm dec}_{((\gamma,\, p),\, \, h)}\bigl(\delta^{(A+B)}_{((\gamma, \, p),\,\, h))})=A+ B$$ for $A\in L(H), B\in L(G).$
\end{proof}	\end{lemma}	

Combining the last two lemmas we obtain the following result.	
\begin{proposition}\label{Prop:ConOnDeco}
	Let $\bigl(\pi\colon E_0 \rightarrow X_0, \mu, [X_1 \rightrightarrows X_0]\bigr)$ 
	be a principal $G$-bundle over the Lie groupoid $\mb{X}=[X_1\rra X_0]$	and $\omega$  a connection on it. Let $\ghta=\mb{G}$ be a Lie $2$-group. Then  $(\omega^{\rm dec}, \omega)$  is a strict connection $1$-form on $\mb{G}$-bundle $\mb{E}^{\rm dec}$ over $\mb{X},$ where $\omega^{\rm dec}$ is as defined in \Cref{E:connectiongeneraal}.		Appropriately modifying the connection $(\omega^{\rm{dec}},\omega)$ using \Cref{cor:stricconnto semi}, we obtain a semistrict connection.
\end{proposition}

\begin{corollary}\label{Ex:EoXodecobundleconnect}
	Every principal $2$-bundle over a discrete Lie groupoid  admits a connection. 
	\begin{proof}Let $[E_1\rra E_0]$ be a $[G_1\rra G_0]$-bundle over a discrete Lie groupoid $[X\rra X]$.
		Every such principal $2$-bundle is a decorated bundle (\Cref{Corollary:discreteisdecorated}). Then any choice of a connection on the underlying principal $G_0$-bundle over $X$ gives a connection on the principal $2$-bundle (\Cref{Prop:ConOnDeco}).
	\end{proof}
\end{corollary}

The pair of groups $({\mathbb R}^n\rtimes G, G),$ where $G\subset GL(n, {\mathbb R})$ is of particular interest in  Cartan geometry (See \cite{kobayashi1956connections,slovak2009parabolic,cattafi2021cartan}). We make one passing observation here relating our construction with that of Cartan's connections (without going into the specifics).  	

\begin{example}
	Suppose $E_0\ra X_0$ is a $G$-bundle.  Consider the $\mb{G}=\ghta$-bundle $[(E_0\times H)\rra E_0]$ over the Lie groupoid $[X_0\rra X_0]$ (see \Cref{Ex:EoXodecobundle}).		
	Let $\theta$ be a \textit{Cartan connection} for the pair  $\bigl(G\subset {\ker(\tau)\rtimes G}\, , \ker(\tau)\rtimes G\bigr)$ on the principal $G$-bundle $E_0\ra X_0$ (\cite[Definition $2.2$]{cattafi2021cartan}), that is a $L(\ker(\tau)\rtimes G)$-valued $1$-form $\theta$ on $E_0$ satisfying the properties,
	\begin{equation}\nonumber
		\begin{split}
			&\theta_p\colon T_p E_0\ra L(\ker(\tau)\rtimes G) \, {\rm is\, a \, linear\,  isomorphism \, for \, every}\, p\in E_0,\\ 
			&\theta_{p g}(v \cdot g)=\ad_{(e, g)^{-1}}\bigl(\theta(v)\bigr),\qquad \forall p\in E_0, g\in G,\\
			&\theta_p \bigl(\delta_p(B)\bigr)=B, \qquad \forall B\in L(G).
		\end{split}
	\end{equation}

	Then one obtains a connection 
	\begin{equation}\nonumber
		\widetilde \theta_{(p, h)}(v, \mathfrak K)=\ad_{(h, e)}\bigl(\theta_p(v)\bigr)-\mathfrak{K}\cdot h^{-1}
	\end{equation}	
	on $\hrtag$-bundle $[(E_0\times H)\rra E_0].$ The source maps $E_0\times H\ra E_0, (p, h)\mapsto p$ and $\hrtag \to G, (h, g)\mapsto g$ define a map from the principal $\hrtag$-bundle $E_0\times H\ra X_0$ to the principal $G$-bundle $E_0\ra X_0.$	
	By push forwarding the connection 	$\widetilde \theta$ on 
	$\hrtag$-bundle $E_0\times H\ra X_0$ along this map of principal bundles we obtain a connection $\theta_0$ on the $G$-bundle $E_0\ra X_0.$ Then $(\widetilde \theta, \theta_0)$ is a connection $1$-form on $\mb{G}$-principal bundle $[E_0\times H\rra E_0]$ over $[X_0\rra X_0].$ 
\end{example}

\subsection{On the existence of connections on a principal $2$-bundle}
In this section, we comment on the existence of connections on principal $2$-bundle over a Lie groupoid. 

We call a Lie groupoid $\mb{X}=[X_1\rra X_0]$ \textit{proper \'etale} if the source map (and hence the target map as well) is a proper \'etale map. It is already known that a principal $G$-bundle $\bigl(E_0\ra X_0, \mu, \mb{X}\bigr)$ over a proper \'etale Lie groupoid $\mb{X}=[X_1\rra X_0]$ admits a connection (\cite[Theorem 3.16]{MR2270285}), that is there exists a connection $\omega$ on the principal $G$-bundle $E_0\ra X_0$ satisfying
$$\omega_p=\omega_{\mu(\gamma,\, p)}.$$
Let $\ghta=\mb{G}$ be a Lie $2$-group. We have seen in \Cref{SS:Conndeco} that if a Lie group $G$-bundle $\bigl(E_0\ra X_0, \mu, \mb{X}\bigr)$ admits a connection then the decorated bundle $\mb{E}^{\rm dec}$ over $\mb{X}$ admits a connection. The \Cref{prop:Characterisdecorated} gives us a criteria for existence of 
a connection on a Lie $2$-group bundle over a proper \'etale Lie groupoid.

\begin{proposition}
	Let $\ghta=\mb{G}$ be a Lie $2$-group,and   $\mb{E}=[E_1\rra E_0]$  a  $\mb{G}$-bundle over a proper \'etale Lie groupoid $\mb{X}=[X_1\rra X_0].$ If the $\mb{G}$-bundle  $\mb{E}$ admits a categorical connection, then it also admits a connection. 
\end{proposition}

\section{Gauge 2-group of a principal 2-bundle} \label{Gauge 2 group of a principal 2 bundle}

In \Cref{Remark:Gauge2-group} we have defined Gauge $2$-group $\mc{G}(\mb{E})$ of a  $\mb{G}$-bundle $\mb{E}\ra \mb{X}$. Consider the  $2$-group $C^{\infty}(\mb{E},\mb{G})^{\mb{G}},$ whose objects are morphisms of Lie groupoids $\sigma\colon \mb{E} \ra \mb{G}$ such that the diagram 
\[	\begin{tikzcd}[sep=small]
	\mb{E}\times \mb{G}\arrow[dd, ""'] \arrow[rrr, "\sigma\times {}^{-1}"] &  &  & \mb{G}\times \mb{G}  \arrow[dd, "{\Ad}"] \\
	&  &  &                                                 \\
	\mb{E} \arrow[rrr, "\sigma"]                                                                            &  &  & \mb{G}                                             
\end{tikzcd}\]
commutes, and morphisms are natural transformations $\Phi\colon \sigma_1\Lrrw \sigma_2$ satisfying $\Phi(p g)={1_g}^{-1}\Phi(p) {1_g}$. The strict $2$-group structure 
in $C^\infty(\mb{E},\mb{G})^{\mb{G}}$ is given by pointwise multiplication.

The  $2$-group  isomorphism between $\mc{G}(\mb{E})$ and $C^{\infty}(\mb{E},\mb{G})^{\mb{G}}$   is   given by
\begin{equation}\label{E:gaugeequi}
	\begin{split}
		&F(x_i)=x_i\sigma(x_i), \forall x_i\in E_i, i\in \{0, 1\}\\
		&\Psi(p)=1_p\Phi(p), \forall p\in E_0.
	\end{split}
\end{equation}

Let $\ghta$ be the Lie crossed module associated to the Lie $2$-group $\mb{G}=[G_1\rra G_0].$ Then a functor $\sigma\in C^{\infty}(\mb{E},\mb{G})^{\mb{G}}$ is of the form 
$\sigma(p)\in G$ and $\sigma(\widetilde \gamma)=(\sigma_{\widetilde \gamma}, g_{\widetilde \gamma})\in \hrtag.$ The functoriality of $\sigma$ imposes following conditions on 
$g_{\widetilde \gamma}\in G, \sigma_{\widetilde \gamma}\in H,$
\begin{equation}\label{E:Functorgauge}
	\begin{split}
		&g_{\widetilde \gamma}=\sigma_{s(\widetilde \gamma)},\\
		&\sigma_{t(\widetilde \gamma)}=\tau(\sigma_{\widetilde \gamma})\sigma_{s(\widetilde \gamma)},\\
		&\sigma_{{\widetilde \gamma}_2} \sigma_{{\widetilde \gamma}_1}=\sigma_{{\widetilde \gamma}_2\circ {\widetilde \gamma}_1}.
	\end{split}
\end{equation}
 A straightforward but tedious computation, using \Cref{E:Adjongroups},  gives us the following $\mb{G}$-equivariancy condition on $(\sigma_{\widetilde \gamma}, \sigma_p)$
\begin{equation}\label{E:Gaugequi} 
	\begin{split}
		&\sigma_{p \, g}=\Ad_{g^{-1}}(\sigma_p)\\
		&\sigma_{{\widetilde \gamma}\, (h, g)}=\alpha_{g^{-1}}\bigl(\Ad_{h^{-1}}\,(\sigma_{\widetilde \gamma})\bigr).
	\end{split}
\end{equation}
Now suppose $\sigma, \sigma '\colon \mb{E}\ra \mb{G}$ are a pair of  objects in $C^{\infty}(\mb{E},\mb{G})^{\mb{G}}$ and $\widetilde \Phi\colon \sigma\Lrrw \sigma '$ an arrow in $C^{\infty}(\mb{E},\mb{G})^{\mb{G}}.$ Then ${\widetilde \Phi}(p)\colon \sigma(p)\ra \sigma '(p)$ is of the form $( \Phi_{p},\sigma_p),$ where $\Phi(p)\in H$ satisfies the following conditions,
\begin{equation}\label{E:Natugauge}
	\begin{split}
		&\sigma'(p)=\tau(\Phi(p))\, \sigma(p),\\
		&\sigma'_{\widetilde \gamma}=\Phi_q\, \sigma_{\widetilde \gamma}\, {\Phi_p}^{-1},\\
		&\Phi_{p\, g}=\alpha_{g^{-1}}(\Phi_p),
	\end{split}
\end{equation}
for all $p\xrightarrow{\widetilde \gamma}q\in E_1, g\in G.$

Though our definition of gauge transformation is a straightforward categorification  of the traditional definition, as we will see, it offers an enrichment that does not have a traditional counterpart. 

Let  $\pi\colon \mb{E} \rightarrow \mb{X}$ be a $\mb{G}=[G_1\rra G_0]$-bundle over a Lie groupoid $\mb{X}=[X_1\rra X_0],$ which admits a  categorical connection. 
Let $F\colon \mb{E}\ra \mb{E}$ be a gauge transformation. Observe that $F$ acts on the set of categorical connections by
\begin{equation}\label{E:Gaugeactcatcon}
	({\mathcal C}\cdot F) (\gamma, p):=F\bigl(\mathcal C(\gamma, F^{-1}(p))\bigr),
\end{equation}
where  $\mathcal{C}$ is a categorical connection and $(\gamma, p)\in s^* E_0$. The corresponding transformation of the action reads as
$\mu(\gamma,p)\mapsto {\mu}_F (\gamma, p):=F\bigl(\mu(\gamma, F^{-1}(p))\bigr).$

Choose  a categorical connection $\mc{C}$ to realize the $\mb{G}$-bundle as a decorated  bundle $\mb{E}^{\rm{dec}}=[s^*E_0\times H \rra E_0]\rra \mb{X}$
over $\mb{X},$ as described in  \Cref{SS:Catconnection}.  
Let $\overline{\mb{E}}^{\rm{dec}}$ be the decorated bundle corresponding to the action ${\mu}_F.$ Note that  $\overline{\mb{E}}^{\rm{dec}}$ and  ${\mb{E}}^{\rm{dec}}$ both have same set of objects and morphisms. 

\begin{proposition}
	Let $\theta_{{\mathcal C}\cdot F}\colon \overline{\mb{E}}^{\rm{dec}}\ra \mb{E}$ be the map defined in \Cref{Lem:isodecgeneral}.	
	Let $\sigma\colon {\mb{E}}\ra \mb{G}$ be  a gauge transformation. Then $\sigma\circ \theta_{{\mathcal C}\cdot F} \colon \overline{\mb{E}}^{\rm{dec}}\ra \mb{G}$ is a gauge transformation on $\overline{\mb{E}}^{\rm{dec}}.$
	\begin{proof}  We verify $\sigma\circ \theta_{{\mathcal C}\cdot F} $ is a functor $\overline{\mb{E}}^{\rm{dec}}\ra \mb{G}$ satisfying conditions in \Cref{E:Gaugequi}. 
\end{proof}	\end{proposition}

We can further simplify the calculations for the gauge transformation on a decorated bundle as follows. Note that $\bigl((\gamma, p), h\bigr)=\bigl((\gamma, p), e\bigr)\, (h^{-1}, e).$ Therefore if 
$F\colon {\mb{E}}^{\rm{dec}}\ra{\mb{E}}^{\rm{dec}}$ is a gauge transformation, then $F(\gamma, p, h)=F(\gamma, p, e) (h^{-1}, e).$ In turn the gauge transformation   $F\colon {\mb{E}}^{\rm{dec}}\ra{\mb{E}}^{\rm{dec}}$ is completely determined by a smooth functor ${F_0}\colon s^* E_0\ra s^*E_0\times H$ satisfying $F_0(p\, g)=F_0(p)\, g, {F_0}(\gamma, p\, g)={F_0}(\gamma, p)\, 1_g.$ Consider the inclusion $1\colon [G\rra G]\hookrightarrow [H\rtimes G\rra G]$. Then, $(F_0,1)$ is a map of principal $2$-bundles $s^*E_0\ra \mb{E}^{\rm{dec}}$ from the $[G\rra G]$-bundle $s^*E_0\ra \mb{X}$ to 
to the $[H\rtimes G\rra G]$-bundle $\mb{E}^{\rm{dec}}\ra \mb{X}$. 

The association is given as $F(p)=F_0(p)$ and $F(\gamma, p, h)=F_0(\gamma, p)(h^{-1}, e).$ Equivalently if $\sigma\in C^{\infty}(\mb{E}^{\rm dec},\mb{G})^{\mb{G}}$ corresponds to $F\colon {\mb{E}}^{\rm{dec}}\ra{\mb{E}}^{\rm{dec}},$ then $\sigma$ is completely determined  by a smooth functor $\sigma_0\colon s^*E_0\ra \mb{G}$ satisfying $\sigma_0(p\, g)=g^{-1}\sigma_0(p) g, \sigma_0(\gamma, p \, g)=1_g^{-1}\sigma_0(\gamma, p) 1_g,$ and then $\sigma$ is given as $\sigma(p)=\sigma_0(p)$ and $\sigma\bigl((\gamma, p), h\bigr)=(h, e)\, \sigma_0(\gamma, p)\, (h^{-1}, e).$
\begin{proposition}\label{Prop:Decogauge}
	Let $\mb{E}^{\rm dec}=[s^*E_0\times H\rra E_0]$ be the decorated $\ghta=\mb{G}$-bundle associated to a $G$-bundle $(E_0\ra X_0\, \mu, \mb{X})$ over the Lie groupoid $\mb{X}=[X_1\rra X_0].$ Then a gauge transformation $\sigma\colon {\mb{E}}^{\rm{dec}}\ra{\mb{G}}$  is completely determined by a smooth functor  $\sigma_0\colon s^*E_0\ra \mb{G}$ satisfying $\sigma_0(p\, g)=g^{-1}\sigma_0(p) g, \sigma_0(\gamma, p \, g)=1_g^{-1}\sigma_0(\gamma, p) 1_g,$ and $\sigma$ is given as
	\begin{equation}
		\begin{split}
			&\sigma(p)=\sigma_0(p),\\
			&\sigma\bigl((\gamma, p), h\bigr)=(h, e)\, \sigma_0(\gamma, p)\, (h^{-1}, e). 
		\end{split}
	\end{equation}
	Equivalently it is completely determined by 
	a map of principal $2$-bundles  $({F_0},1)\colon s^* E_0\ra \mb{E}^{\rm{dec}}$. 
	The association is given as $F(p)=F_0(p)$ and $F\bigl((\gamma, p), h\bigr)=F_0(\gamma, p)(h^{-1}, e).$   
\end{proposition}

\begin{example}
	Any gauge transformation on a principal $G$-bundle $E\ra M$ defines a gauge transformation on the $[G\rra G]$-bundle $[E\rra E]$ over $[X\rra X]$ and vice versa. 
\end{example}

\begin{example}Recall that a principal $[G\rra G]$-bundle $[E_1\rra E_0]\ra [X_1\rra X_0]$ is same as the pullback bundle $[s^*E_0\rra E_0]\ra [X_1\rra X_0]$ defined with respect to 
	an action $\mu\colon s^*E_0\ra E_0$ (\Cref{E:Example of principal 2-bundle}). Then, any gauge transformation on a principal $[G\rra G]$-bundle $[E_1\rra E_0]\ra [X_1\rra X_0]$ is given by a gauge transformation $F\colon E_0\ra E_0$ on the principal $G$-bundle $E_0\ra X_0$ such that $F\big(\mu (\gamma, p)\big)=\mu (\gamma,F(p))$.
\end{example}

\begin{example}\label{Ex:EoXodecobundlegaguge}
	Let $\ghta=\mb{G}$ be a Lie $2$-group and $E\ra X$ a principal $G$-bundle. Consider the decorated Lie $2$-group $\ghta$-bundle ${E}^{\rm dec}:=[E\times H\rra E]$ over $[X\rra X]$ constructed in \Cref{Ex:EoXodecobundle}. Any gauge transformation $F\colon E\ra E$ on the principal $G$-bundle $E\ra X$ defines a gauge transformation on the
	$\ghta$-bundle ${E}^{\rm dec}\ra E^{\rm dec}$
	by $p\mapsto F(p)$ and $(p, h)\mapsto \bigl(F(p), h\bigr),$ for $p\in E, h\in H.$
\end{example} 

\begin{example}\label{E:Example of principal 2-bundle ordinary Guage}
	Consider the $[H\rra e]$-bundle   $[E_1\rra X_0]$ over $[X_1\rra X_0]$ of \Cref{E:Example of principal 2-bundle ordinary}.   Then a gauge transformation $F\colon E_1\ra E_1$ on the (traditional) principal $H$-bundle $E_1\ra X_1$ defines a gauge transformation on $[H\rra e]$-bundle   $[E_1\rra X_0]$ over $[X_1\rra X_0]$ by $x\mapsto x$ and $\widetilde \gamma \mapsto F(\widetilde \gamma),$ for $x\in X_0, \widetilde \gamma \in E_1$.  
\end{example}

\subsection{Action of Gauge $2$-group on the category of Connections} \label{Action of Gauge $2$-group on the category of Connections}

Let $\mb{G}$ be a Lie $2$-group, $\pi\colon \mb{E} \rightarrow \mb{X}$  a $\mb{G}$-bundle over $\mb{X}$ and $\mathcal{G}(\mb{E})$ its gauge $2$-group. Let  $\Omega_{\mb{E}}^{\rm{semi}}$ and  $\Omega_{\mb{E}}^{\rm{strict}}$ respectively be the categories of strict and semi-strict connection $1$-forms. Let $F\colon \mb{E}\ra \mb{E}$ be a gauge transformation. Then by taking the differentials of $F$  we obtain the  $\mb{G}$-equivariant Lie groupoid isomorphisms $F_*\colon T\mb{E}\ra T\mb{E}$
\begin{equation}\nonumber
	\begin{split}
		&(p, v)\mapsto F_{*}\bigl(p, v\bigr):=\bigl(p, F_{*, p}(v)\bigr),\\
		&(\widetilde \gamma, \widetilde X)\mapsto  F_{*}(\widetilde \gamma, \widetilde X):=\bigl(\widetilde \gamma, F_{*, \widetilde \gamma}(\widetilde X)\bigr) \end{split}
\end{equation}
for any $(p, v)\in TE_0, (\widetilde \gamma, \widetilde X) \in TE_1.$ Similarly if $\Psi\colon F \Lrrw F'$ is an arrow in the category $\mathcal{G}(\mb{E}),$ then the differential of $\Psi$ gives us a $\mb{G}$-equivariant natural isomorphism $\Psi_*\colon F_*\Lrrw F_*':$
\begin{equation}
	(p, v)\mapsto \Psi_*(p, v):=\bigl(1_p, \Psi_{*, p}(v)\bigr)\colon F_*(p, v)\longrightarrow F'_*(p, v).
\end{equation}
Now given an object  $\omega \in \Omega_{\mb{E}}^{\rm{strict}}\, \rm{or}\,  \Omega_{\mb{E}}^{\rm{semi}}$ and an arrow $\bigl(\zeta\colon \omega\Lrrw \omega'\bigr) \in \Omega_{\mb{E}}^{\rm{strict}}\, \rm{or}\,  \Omega_{\mb{E}}^{\rm{semi}}$  we respectively define 
\begin{equation}\label{E:gauge2action}
	\begin{split}
		&(F,\omega) \mapsto \omega\circ F_*^{-1} \colon T\mb{E}\ra L(\mb{G}),\\
		&\bigl(\Psi\colon F \Lrrw F',\zeta\colon \omega \Lrrw \omega'\bigr) \mapsto \zeta\circ_{H} {\Psi_*}^{-1}\colon \omega\circ F_{*}^{-1}\Lrrw  \omega\circ F_{*}^{-1},
	\end{split}
\end{equation}  
where $\circ_H$ denote the horizontal composition of natural transformations. We claim that  \Cref{E:gauge2action}  defines an action of the Lie $2$-group $\mathcal{G}(\mb{E})$  
on 	$\Omega_{\mb{E}}^{\rm{strict}}$ and $\Omega_{\mb{E}}^{\rm{semi}}.$ If $\omega\in \Omega_{\mb{E}}^{\rm{strict}}$ then it is obvious that  $\omega\circ F_*^{-1}\in  \Omega_{\mb{E}}^{\rm{strict}}.$ Moreover $\mb{G}$-equivariance of $\zeta\colon \omega \Lrrw \omega'$ and  $\Psi\colon F \Lrrw F'$ give the $\mb{G}$-equivariance of $\zeta\circ_{H} {\Psi_*}^{-1}$ for a morphism $\Psi$ either in $\Omega_{\mb{E}}^{\rm{strict}}$ or $\Omega_{\mb{E}}^{\rm{semi}}.$ Perhaps the only non-trivial part is to show that $\omega\circ F_*^{-1}\in  \Omega_{\mb{E}}^{\rm{semi}}$ for $\omega\in  \Omega_{\mb{E}}^{\rm{semi}}.$  Precisely we have to show that if there exists a $\mb{G}$-equivariant natural isomorphism  $\kappa\colon \omega \circ \delta \Longrightarrow \rm{pr_2}$ then there exists a $\mb{G}$-equivariant  natural isomorphism  $\kappa^{*}\colon (\omega\circ F_*^{-1}) \circ \delta \Longrightarrow \rm{pr_2},$
where  $\delta$ is the vertical vector field generating functor for the action of  $\mb{G}$ on $\mb{E}.$ To see this we proceed as follows

For $(p, B) \in E_0 \times L(G_0)$ we define $\kappa_{*}\colon \bigl(\omega\circ F_{*}^{-1}\bigr)\circ \delta  \Lrrw{\rm pr_2}$ as $\kappa_{*}(p, B)=\kappa\bigl(F^{-1}(p), B\bigr)$. Note  that since a gauge transformation preserves the vertical vector field,  $F^{-1}_{*,\, p}(\delta_p(B))=\delta_{F^{-1}(p)}(B),$ we have 
$s\bigl(\kappa_{*}(p, B)\bigr)=s\bigl(\kappa\bigl(F^{-1}(p), B\bigr)\bigr)=\omega(\delta_{F^{-1}(p)}(B))=\bigl(\omega\circ F_*^{-1}\bigr)\circ \delta(p, B)$ and $t\bigl(\kappa_{*}(p, B)\bigr)=t\bigl(\kappa\bigl(F^{-1}(p), B\bigr)\bigr)=B.$ Thus 
$$\kappa_{*}(p, B)\colon \bigl(\omega\circ F_{*}^{-1}\bigr)\circ \delta(p, B) \rightarrow B.$$

Moreover since $\kappa\colon \omega \circ \delta \Longrightarrow \rm{pr_2}$ is a natural transformation, for any  $(p, \, B)\xrightarrow{(\widetilde \gamma,\, K)}(q,\, B')\in E_1\times L(G_1),$ we have 
\begin{equation}
	K \circ \kappa(F^{-1}(p),\,  B)= \kappa(F^{-1}(q), \,B') \circ \omega\bigl(F^{-1}(\gamma), \delta_{F^{-1}(\gamma)}(K)\bigr).
\end{equation}
Substituting $\delta_{F^{-1}(\widetilde \gamma)}(K)$ by $F^{-1}_{*,\, \widetilde \gamma}\bigl(\delta_{\widetilde \gamma}(K)\bigr)$ we confirm that $\kappa_{*}$ is indeed a natural transformation. Now we verify $\mb{G}$-equivariancy of $\kappa_*,$
$$\kappa_* \bigl((p, \, B)\, g\bigr)=\kappa \bigl((F^{-1}(p), B) \, g\bigr)=\kappa (F^{-1}(p), B)\, 1_g=\kappa_*(p, \, B)\, 1_g.$$

\begin{proposition}\label{Prop:Actiongaugeconnecat}
	Let $\pi\colon \mb{E} \rightarrow \mb{X}$ be a $\mb{G}$-bundle over $\mb{X}$ and $\mathcal{G}(\mb{E})$  its gauge $2$-group. Then \Cref{E:gauge2action}		
	defines an action of $\mathcal{G}(\mb{E})$ on $\Omega_{\mb{E}}^{\rm{semi}}.$ The action restricts to $\Omega_{\mb{E}}^{\rm{strict}}\subset \Omega_{\mb{E}}^{\rm{semi}}.$ 
\end{proposition}
If one represents $F$ by $\sigma\in C^{\infty}(\mb{E},\mb{G})^{\mb{G}},$ then the gauge transformations of a strict connection $1$-form $\omega^{\rm st}$ and a semi-strict connection $1$-form $\omega^{\rm se}$ are respectively  expressed as
\begin{equation}\label{E:gaugewithgroup}
	\begin{split}
		&\omega^{\rm st}\mapsto \Ad_{\sigma}\omega^{\rm st}-(d\sigma) \sigma^{-1},\\
		&\omega^{\rm se}\mapsto \Ad_{\sigma}\omega^{\rm se}-{\widehat\kappa}_{\omega^{\rm se}}\circ (d\sigma) \sigma^{-1}-(d\sigma) \sigma^{-1},
	\end{split}
\end{equation}
where ${\widehat\kappa}_{\omega^{\rm se}}\colon \mb{E}\times L({\mb{G}})\longrightarrow  L({\mb{G}})^{\tau}=[L(H)\rtimes \tau\bigl(L(H)\bigr)\rra \tau\bigl(L(H)\bigr)]$
is as  in \Cref{Prop:Semisirictnatural}. Note that here everything is defined point-wise.  	

\begin{remark}
	Since $\Omega_{\mb{E}}^{\rm{strict}}$ and $\Omega_{\mb{E}}^{\rm{semi}}$ respectively are isomorphic to $C^{\rm{strict}}_\mb{E}$ and $C^{\rm{semi}}_\mb{E},$ the action in  \Cref{Prop:Actiongaugeconnecat} induces actions of $\mathcal{G}(\mb{E})$ on $C^{\rm{strict}}_\mb{E}$ and $C^{\rm{semi}}_\mb{E}.$ The induced action is given as
	$(R, , F) \mapsto \overline{F} \circ R \circ \overline{T F^{-1}},$
	where  $\overline{F}\colon {\rm Ad}(\mb{E}) \rightarrow {\rm Ad}(\mb{E})$ sends  $[x_i, A_i]$ to $[F(x_i), A_i]$ for any $[x_i, A_i]\in \Ad (E_i), i\in \{0, 1\}$ and 
	$\overline{TF}\colon {\rm At}(\mb{E}) \rightarrow {\rm At}(\mb{E})$ sends an object $[x_i, v_i]$ to $[F(x_i), F_{{*, x_i}}(v_i)]$ for any $[x_i, v_i]\in {\rm At}(E_i), i\in \{0, 1\}.$
\end{remark}	

As with the connections, one can consider the set of categorical connected components $\bar{\mc{G}}(\mb{E})$ of the gauge $2$-group $\mc{G}(\mb{E})$ for a $\mb{G}$-bundle $\pi\colon \mb{E}\ra \mb{X}.$ Then we have the following.
\begin{corollary}
	Let $\pi\colon \mb{E} \rightarrow \mb{X}$ be a $\mb{G}$-bundle over $\mb{X}$ and $\bar{\mathcal{G}(\mb{E})}$  its connected components of the gauge $2$-group. Then \Cref{E:gauge2action}		
	defines an action of ${\bar{\mathcal{G}}(\mb{E})}$ on $\bar{\Omega_{\mb{E}}}^{\rm{semi}}.$ The action restricts to $\bar{\Omega}_{\mb{E}}^{\rm{strict}}\subset \bar{\Omega}_{\mb{E}}^{\rm{semi}}.$ 
\end{corollary}

\subsection{An Extended symmetry of $\Omega_{\mb{E}}^{\rm{semi}}$} In fact, $\Omega_{\mb{E}}^{\rm{semi}}$ enjoys a larger symmetry than given in \Cref{Prop:Actiongaugeconnecat}. To see this, consider 
$$\Omega^{G}\bigl({E_0}, L(H)\bigr)=\bigl\{\lambda\colon TE_0\ra L(H)\big|\, \lambda_{p\cdot g}(v\cdot g)=\alpha_{g^{-1}}\bigl(\lambda_p(v)\bigr)\bigr\},$$
the $G$-equivariant $L(H)$-valued smooth $1$-forms on $E_0.$ Observe that every such $\lambda\in \Omega^{G}\bigl({E_0}, L(H)\bigr)$ can be viewed as a $\mb{G}$-equivariant
$1$-form $\overline\lambda \colon T\mb{E}\ra L(\mb{G})$ on $\mb{E}$ taking values in $L(\mb{G})^{\tau}=\big[L(H)\rtimes \tau\bigl(L(H)\bigr)\rra \tau\bigl(L(H)\bigr)\big]\subset L(\mb{G}),$
\begin{equation}\label{E:lambdabar}
	\begin{split}
		&\overline \lambda:=\tau(\lambda)\colon TE_0\ra \tau\bigl(L(H)\bigr),\\
		&\overline \lambda:=\bigl(t^*\lambda-s^*\lambda+(\tau(s^*\lambda)\bigr)  \colon TE_1\ra L(H)\oplus \tau(L(H).
	\end{split}
\end{equation}

Let $\widehat{\Omega}^{\mb{G}}(E_0,L(H))$ be the category whose objects are functors of the form \Cref{E:lambdabar} and morphisms are fiber-wise linear natural transformations.

Now, we intertwine the category  $\widehat{\Omega}^{\mb{G}}(E_0,L(H))$ with the gauge 
$2$-group $\mc{G}(\mb{E})$ to define a $2$-group
$\mc{G}^{\rm ext}:=\mc{G}{(\mb{E})}\times  \widehat{\Omega}^{\mb{G}}(E_0,L(H))$ with group products
\begin{equation}\label{E:semidirectgengauge}\begin{split}
		&(F, \lambda)(F', \lambda ')=\bigl(F F ', \lambda+(\lambda'\circ F_*^{-1})\bigr),\\
		&(\Psi, \theta)(\Psi', \theta ')=\bigl(\Psi \Psi ', \theta+(\theta'\circ_H \Psi_*^{-1})\bigr).
\end{split}\end{equation}
One has to apply \Cref{Remark:Interchange law of Lie 2-algebra} to verify that $\mc{G}^{\rm ext}$ is a $2$-group. Note that here we have made the identification $\Obj \big( \mc{G}^{\rm ext} \big) \simeq \Obj \big( \mc{G}{(\mb{E})} \big) \times  {\Omega}^{G}(E_0,L(H))$.

Then \Cref{Prop:Semisirictnatural} allows us to define a left action of $\mc{G}^{\rm ext}$ on $\Omega_{\mb{E}}^{\rm{semi}}$ by
\begin{equation}\label{E:GengaugeactdefwithF}
	\begin{split}
		&\bigl((F, \lambda),\omega\bigr)\mapsto \omega\circ F_*^{-1}+\overline \lambda\colon T\mb{E}\ra L(\mb{G}),\\
		&\bigl((\Psi, \theta),\zeta \bigr)\mapsto \zeta\circ_H \Psi_*^{-1}+ \theta.
\end{split}\end{equation}
One can as well express the action above as 
\begin{equation}\label{E:Gengaugeactdef}
	\begin{split}
		\bigl((\sigma, \lambda),\omega \bigr)\mapsto \bigl(\Ad_{\sigma}(\omega)-\kappa_{\omega}\circ (d\sigma) \sigma^{-1}-(d\sigma) \sigma^{-1}+\overline \lambda\bigr)\colon T\mb{E}\ra L(\mb{G}),
	\end{split}
\end{equation}
where $\sigma\in C^{\infty}(\mb{E},\mb{G})^{\mb{G}}$.

We call the action above an \textit{extended gauge transformation}.
It is obvious that this action does not restrict to $\Omega_{\mb{E}}^{\rm{strict}}$ (see \Cref{cor:stricconnto semi}). Particularly, the action we derived in 
\Cref{E:GengaugeactdefwithF}	is of some importance in higher gauge theory and related areas of physics (for example $2$-BF theories). We refer to the work of Martins and Mikovi\' c \cite{martins2011lie}, and works of Martins and Picken in this context \cite{MR2661492, martins2011fundamental, MR2764890}. It would be instructive to consider the following example. 

\begin{example}\label{Ex:EoXodecobundleconnectgen}
	Consider the $\mb{G}=\ghta$-bundle ${E}^{\rm dec}:=[E\times H\rra E]$ over $[X\rra X]$ defined by a traditional principal $G$-bundle $E\ra X$ as in \Cref{Ex:EoXodecobundle}. A gauge transformation $\sigma\colon {E}^{\rm dec}\ra \mb{G}$ on  ${E}^{\rm dec}=[E\times H\rra E]$ is given, by a gauge transformation $\sigma_0\colon E\ra G$ on the $G$-bundle $E\ra X,$ as $\sigma(p)=\sigma_0(p), \sigma(p, h)=\bigl(h, \,\sigma_0(p)\bigr).$ Thus an extended gauge transformation $(\sigma, \lambda)$ is given by a classical gauge transformation $\sigma_0\colon E\ra G$ on the $G$-bundle   $E\ra X$ and a $G$-equivariant $L(H)$-valued $1$-form $\lambda\in  \Omega^{G}\bigl({E}, L(H)\bigr)$ on $E.$
	In \Cref{Ex:EoXodecobundleconnect} we have seen, any connection $\omega_0$ on the principal $G$-bundle $E\ra X$ gives a strict connection $\overline\omega:=(\omega^{\rm dec}, \omega_0)$. Under the action of $(\sigma, \lambda),$ the connection $\overline\omega$ transforms as:
	$$\overline\omega \mapsto \Ad_{\sigma}\overline \omega-(d\sigma)\sigma^{-1}+\overline\lambda.$$
	Which explicitly reads
	\begin{equation}\label{E:gaugehighergaue}
		\begin{split}
			\omega_0(p)&\mapsto \Ad_{\sigma_0(p)}\omega_0(p)-(d\sigma_0(p))\sigma_0(p)^{-1}+\tau(\lambda(p)),\\
			\omega^{\rm dec}(p,\, h) &\mapsto \Ad_{(h,\,\sigma_0(p))}\omega^{\rm dec}(p,\, h)-\biggl(d\bigl(h,\, \sigma_0(p)\bigr)\biggr)\bigl(h,\, \sigma_0(p)\bigr)^{-1}\\
			&\hskip 3 cm +\bigl(\tau(\lambda(p))-\lambda(p)+\Ad_h\lambda({p})\bigr).
		\end{split}
	\end{equation}
\end{example}

The gauge transformation obtained here in \Cref{E:gaugehighergaue}
is precisely the gauge transformation of a connection $1$-form in the higher gauge theories (see \cite{MR3645839} or \cite{martins2011lie}).
Note that the extended gauge transformation is a direct consequence of obstruction to the splitting of the Atiyah sequence in Diagram \ref{Dia:Obj-MorAtiyahsplitting}, and global in nature.
Of course, higher gauge theories also involve higher-order differential forms, which we have not dealt with here. In our follow-up paper, we will construct the  Chern-Weil theory for the framework developed here, which is expected to make the relation with higher gauge theories more transparent. 

\section*{Acknowledgements}
The first named author acknowledges research support from SERB, DST, Government of India grant MTR/2018/000528.

\bibliography{references}

\begin{thebibliography}{10}

\bibitem{MR2117631}
Paolo Aschieri, Luigi Cantini, and Branislav Jur\v{c}o.
\newblock Nonabelian bundle gerbes, their differential geometry and gauge
  theory.
\newblock {\em Comm. Math. Phys.}, 254(2):367--400, 2005.

\bibitem{MR86359}
M.~F. Atiyah.
\newblock Complex analytic connections in fibre bundles.
\newblock {\em Trans. Amer. Math. Soc.}, 85:181--207, 1957.

\bibitem{baez2004higher}
John Baez and Urs Schreiber.
\newblock Higher gauge theory: 2-connections on 2-bundles.
\newblock {\em arXiv:hep-th/0412325}, 2004.

\bibitem{MR2068522}
John~C. Baez and Alissa~S. Crans.
\newblock Higher-dimensional algebra. {VI}. {L}ie 2-algebras.
\newblock {\em Theory Appl. Categ.}, 12:492--538, 2004.

\bibitem{MR2068521}
John~C. Baez and Aaron~D. Lauda.
\newblock Higher-dimensional algebra. {V}. 2-groups.
\newblock {\em Theory Appl. Categ.}, 12:423--491, 2004.

\bibitem{MR2342821}
John~C. Baez and Urs Schreiber.
\newblock Higher gauge theory.
\newblock In {\em Categories in algebra, geometry and mathematical physics},
  volume 431 of {\em Contemp. Math.}, pages 7--30. Amer. Math. Soc.,
  Providence, RI, 2007.

\bibitem{bakovic2009simplicial}
Igor Bakovic.
\newblock The simplicial interpretation of bigroupoid 2-torsors.
\newblock {\em arXiv:0902.3436}.

\bibitem{bakovic2008bigroupoid}
Igor Bakovi{\'c}.
\newblock {\em Bigroupoid 2-torsors}.
\newblock PhD thesis, Ludwig Maximilian University of Munich, 2008.

\bibitem{MR2709030}
Tobias~Keith Bartels.
\newblock {\em Higher gauge theory: 2-bundles}.
\newblock ProQuest LLC, Ann Arbor, MI, 2006.
\newblock Thesis (Ph.D.)--University of California, Riverside.

\bibitem{biswas2020chern}
Indranil Biswas, Saikat Chatterjee, Praphulla Koushik, and Frank Neumann.
\newblock Chern-{W}eil theory for principal bundles over {L}ie groupoids.
\newblock {\em arXiv:2012.08447}, 2020.

\bibitem{MR3150770}
Indranil Biswas and Frank Neumann.
\newblock Atiyah sequences, connections and characteristic forms for principal
  bundles over groupoids and stacks.
\newblock {\em C. R. Math. Acad. Sci. Paris}, 352(1):59--64, 2014.

\bibitem{MR1291599}
Francis Borceux.
\newblock {\em Handbook of categorical algebra. 1}, volume~50 of {\em
  Encyclopedia of Mathematics and its Applications}.
\newblock Cambridge University Press, Cambridge, 1994.
\newblock Basic category theory.

\bibitem{MR2183393}
Lawrence Breen and William Messing.
\newblock Differential geometry of gerbes.
\newblock {\em Adv. Math.}, 198(2):732--846, 2005.

\bibitem{Bursztyn2016163}
Henrique Bursztyn, Alejandro Cabrera, and Matias {del Hoyo}.
\newblock Vector bundles over lie groupoids and algebroids.
\newblock {\em Advances in Mathematics}, 290:163--207, 2016.

\bibitem{bursztyn2009linear}
Henrique Bursztyn, Alejandro Cabrera, and Cristi{\'a}n Ortiz.
\newblock Linear and multiplicative 2-forms.
\newblock {\em Letters in Mathematical Physics}, 90(1-3):59, 2009.

\bibitem{cattafi2021cartan}
Francesco Cattafi.
\newblock Cartan geometries and multiplicative forms.
\newblock {\em Differential Geometry and its Applications}, 75:101722, 2021.

\bibitem{MR3126940}
Saikat Chatterjee, Amitabha Lahiri, and Ambar~N. Sengupta.
\newblock Path space connections and categorical geometry.
\newblock {\em J. Geom. Phys.}, 75:129--161, 2014.

\bibitem{MR3213404}
Saikat Chatterjee, Amitabha Lahiri, and Ambar~N. Sengupta.
\newblock Twisted actions of categorical groups.
\newblock {\em Theory Appl. Categ.}, 29:No. 8, 215--255, 2014.

\bibitem{MR3504595}
Saikat Chatterjee, Amitabha Lahiri, and Ambar~N. Sengupta.
\newblock Construction of categorical bundles from local data.
\newblock {\em Theory Appl. Categ.}, 31:Paper No. 14, 388--417, 2016.

\bibitem{fiorenza2011cech}
Domenico Fiorenza, Urs Schreiber, and Jim Stasheff.
\newblock Cech cocycles for differential characteristic classes: an
  infinity-{L}ie theoretic construction.
\newblock {\em Adv. Theor. Math. Phys}, 16(1):149--250, 2012.

\bibitem{MR3480061}
Gr\'{e}gory Ginot and Mathieu Sti\'{e}non.
\newblock {$G$}-gerbes, principal 2-group bundles and characteristic classes.
\newblock {\em J. Symplectic Geom.}, 13(4):1001--1047, 2015.

\bibitem{MR3696590}
Alfonso Gracia-Saz and Rajan~Amit Mehta.
\newblock Vb-groupoids and representation theory of {L}ie groupoids.
\newblock {\em J. Symplectic Geom.}, 15(3):741--783, 2017.

\bibitem{MR3351282}
Branislav Jur\v{c}o, Christian S\"{a}mann, and Martin Wolf.
\newblock Semistrict higher gauge theory.
\newblock {\em J. High Energy Phys.}, (4):087, front matter+66, 2015.

\bibitem{MR3548195}
Branislav Jur\v{c}o, Christian S\"{a}mann, and Martin Wolf.
\newblock Higher groupoid bundles, higher spaces, and self-dual tensor field
  equations.
\newblock {\em Fortschr. Phys.}, 64(8-9):674--717, 2016.

\bibitem{kobayashi1956connections}
Sh{\^o}shichi Kobayashi.
\newblock On connections of cartan.
\newblock {\em Canadian Journal of Mathematics}, 8:145--156, 1956.

\bibitem{MR2493616}
Camille Laurent-Gengoux, Mathieu Sti\'{e}non, and Ping Xu.
\newblock Non-abelian differentiable gerbes.
\newblock {\em Adv. Math.}, 220(5):1357--1427, 2009.

\bibitem{MR2270285}
Camille Laurent-Gengoux, Jean-Louis Tu, and Ping Xu.
\newblock Chern-{W}eil map for principal bundles over groupoids.
\newblock {\em Math. Z.}, 255(3):451--491, 2007.

\bibitem{MR896907}
Kirill C.~H. Mackenzie.
\newblock {\em Lie groupoids and {L}ie algebroids in differential geometry},
  volume 124 of {\em London Mathematical Society Lecture Note Series}.
\newblock Cambridge University Press, Cambridge, 1987.

\bibitem{MR3744376}
Kirill C.~H. Mackenzie, Anatol Odzijewicz, and Aneta Sli\.{z}ewska.
\newblock Poisson geometry related to {A}tiyah sequences.
\newblock {\em SIGMA Symmetry Integrability Geom. Methods Appl.}, 14:Paper No.
  005, 29, 2018.

\bibitem{MR2661492}
Jo\~{a}o~Faria Martins and Roger Picken.
\newblock On two-dimensional holonomy.
\newblock {\em Trans. Amer. Math. Soc.}, 362(11):5657--5695, 2010.

\bibitem{MR2764890}
Jo\~{a}o~Faria Martins and Roger Picken.
\newblock Surface holonomy for non-abelian 2-bundles via double groupoids.
\newblock {\em Adv. Math.}, 226(4):3309--3366, 2011.

\bibitem{martins2011lie}
Jo{\~a}o Martins, Aleksandar Mikovi{\'c}, et~al.
\newblock Lie crossed modules and gauge-invariant actions for 2-bf theories.
\newblock {\em Advances in Theoretical and Mathematical Physics},
  15(4):1059--1084, 2011.

\bibitem{martins2011fundamental}
Joao~Faria Martins and Roger Picken.
\newblock The fundamental gray 3-groupoid of a smooth manifold and local
  3-dimensional holonomy based on a 2-crossed module.
\newblock {\em Differential Geometry and its Applications}, 29(2):179--206,
  2011.

\bibitem{Moerdijk2}
Ieke Moerdijk.
\newblock Introduction to the language of stacks and gerbes.
\newblock {\em arXiv:math/0212266}, 2002.

\bibitem{MR3423073}
Thomas Nikolaus, Urs Schreiber, and Danny Stevenson.
\newblock Principal {$\infty$}-bundles: general theory.
\newblock {\em J. Homotopy Relat. Struct.}, 10(4):749--801, 2015.

\bibitem{MR3385700}
Thomas Nikolaus, Urs Schreiber, and Danny Stevenson.
\newblock Principal {$\infty$}-bundles: presentations.
\newblock {\em J. Homotopy Relat. Struct.}, 10(3):565--622, 2015.

\bibitem{MR3089401}
Thomas Nikolaus and Konrad Waldorf.
\newblock Four equivalent versions of nonabelian gerbes.
\newblock {\em Pacific J. Math.}, 264(2):355--419, 2013.

\bibitem{MR2800361}
Christopher~J. Schommer-Pries.
\newblock Central extensions of smooth 2-groups and a finite-dimensional string
  2-group.
\newblock {\em Geom. Topol.}, 15(2):609--676, 2011.

\bibitem{schreiber2005loop}
Urs Schreiber.
\newblock From loop space mechanics to nonabelian strings.
\newblock {\em arXiv:hep-th/0509163}, 2005.

\bibitem{schreiber2013differential}
Urs Schreiber.
\newblock Differential cohomology in a cohesive infinity-topos.
\newblock {\em arXiv:1310.7930}, 2013.

\bibitem{slovak2009parabolic}
Jan Slov{\'a}k and Andreas Cap.
\newblock {\em Parabolic Geometries I, Background and General Theory}, volume
  1000.
\newblock American Mathematical Society, 2009.

\bibitem{stevenson}
Danny Stevenson.
\newblock Lie 2-algebras and the geometry of gerbes.
\newblock {\em available online as
  https://math.ucr.edu/home/baez/namboodiri/stevenson\_maclane.pdf}.

\bibitem{MR3894086}
Konrad Waldorf.
\newblock A global perspective to connections on principal 2-bundles.
\newblock {\em Forum Math.}, 30(4):809--843, 2018.

\bibitem{MR3645839}
Wei Wang.
\newblock On the global 2-holonomy for a 2-connection on a 2-bundle.
\newblock {\em J. Geom. Phys.}, 117:151--178, 2017.

\bibitem{MR2805195}
Christoph Wockel.
\newblock Principal 2-bundles and their gauge 2-groups.
\newblock {\em Forum Math.}, 23(3):565--610, 2011.

\end{thebibliography}
\bibliographystyle{plain}

\end{document}